\documentclass[10pt,a4paper]{article}

\usepackage{textcomp} 
\usepackage{amsmath,amsthm,amsfonts,amssymb}
\usepackage{mathtools}
\usepackage{enumitem}\setitemize{leftmargin=5mm}
\usepackage{accents}
\usepackage{color}
\usepackage[usenames,dvipsnames,svgnames,table]{xcolor}
\usepackage{cite}
\usepackage{soul}
\usepackage{pgf,tikz}
\usetikzlibrary{matrix, calc, arrows,shapes,decorations}
\usepackage{pgfplots} 
\pgfplotsset{
    discard if not/.style 2 args={
        x filter/.code={
            \edef\tempa{\thisrow{#1}}
            \edef\tempb{#2}
            \ifx\tempa\tempb--
            \else
                
            \fi
        }
    }
}
\pgfplotsset{compat=1.14}
\usepackage[lined,boxed,commentsnumbered]{algorithm2e}
\usepackage[bookmarks=true, bookmarksopen=true, bookmarksopenlevel=5]{hyperref} 
\definecolor{myblue}{rgb}{0,0,0.6}     
\hypersetup{pdftitle={A space-time quasi-Trefftz DG method for the wave equation with piecewise-smooth coefficients},  
            pdfauthor={Andrea Moiola},
     colorlinks=true, linkcolor=myblue,  citecolor=myblue, filecolor=myblue,   urlcolor=myblue,  }
\allowdisplaybreaks[4]

\allowdisplaybreaks[4]
\providecommand{\red}{\textcolor{red}}
\definecolor{myblue}{rgb}{0,0,0.6}

\definecolor{amcol}{rgb}{0.8,0,0}

\definecolor{rvcol}{rgb}{1,0,0}
\newcommand{\rv}[1]{{{#1}}}

\definecolor{lmigcol}{rgb}{0.8,0.3,0}

\newcommand{\di}{\,\mathrm{d}}
\newcommand{\iin}{\;\text{in}\;}
\newcommand{\oon}{\;\text{on}\;}
\newcommand{\deO}{{\partial\Omega}}
\newcommand*{\conj}[1]{\overline{#1}}
\newcommand*{\N}[1]{\left\|#1\right\|}
\newcommand*{\abs}[1]{\left|#1\right|}
\newcommand{\Tnorm}[1]{|||#1|||}
\newcommand*{\jmp}[1]{[\![#1]\!]}    
\newcommand*{\mvl}[1]{\{\!\!\{#1\}\!\!\}}
\DeclareMathOperator{\diam}{diam}

\newcommand{\Uu}[1]{{\mathbf{#1}}}
\newcommand{\IN}{\mathbb{N}}\newcommand{\IP}{\mathbb{P}}
\newcommand{\IR}{\mathbb{R}}
\newcommand{\IU}{\mathbb{U}}
\newcommand{\IW}{\mathbb{W}}
\newcommand{\ba}{{\Uu a}}\newcommand{\bb}{{\Uu b}}
\newcommand{\be}{{\Uu e}}\newcommand{\bn}{{\Uu n}}
\newcommand{\bx}{{\Uu x}}
\newcommand{\bH}{{\Uu H}}\newcommand{\bN}{{\Uu N}}
\newcommand{\bV}{{\Uu V}}
\newcommand{\bsigma}{{\boldsymbol \sigma}}\newcommand{\btau}{{\boldsymbol\tau}}		
\newcommand{\bzero}{\Uu{0}}
\newcommand{\calA}{{\mathcal A}}
\newcommand{\calE}{{\mathcal E}}
\newcommand{\calF}{{\mathcal F}}

\newcommand{\calO}{{\mathcal O}}
\newcommand{\calT}{{\mathcal T}}

\newtheorem{theorem}{Theorem}[section]

\newtheorem{prop}[theorem]{Proposition}
\newtheorem{rem}[theorem]{Remark}

\newcommand{\OO}{{(\Omega)}}
\newcommand{\mmbox}[1]{\fbox{\ensuremath{\displaystyle{ #1 }}}}
\newcommand{\ee}{{\rm e}}

\newcommand{\tand}{\text{ and }}
\newcommand{\tfor}{\text{ for }}
\newcommand{\LtO}{L^2(\Omega)}
\newcommand{\supp}{\operatorname{supp}}
\newcommand{\deK}{{\partial K}}
\newcommand{\Fh}{\calF_h}
\newcommand{\Th}{{(\calT_h)}}
\newcommand{\GD}{{\Gamma_{\mathrm D}}}
\newcommand{\GN}{{\Gamma_{\mathrm N}}}
\newcommand{\GR}{{\Gamma_{\mathrm R}}}
\newcommand{\FD}{{\Fh^{\mathrm D}}}
\newcommand{\FN}{{\Fh^{\mathrm N}}}
\newcommand{\FR}{{\Fh^{\mathrm R}}}
\newcommand{\FT}{{\Fh^T}}
\newcommand{\FO}{{\Fh^0}}
\newcommand{\rtime}{{\mathrm{time}}}
\newcommand{\rspace}{{\mathrm{space}}}
\newcommand{\Fspa}{{\Fh^\rspace}}
\newcommand{\Ftime}{{\Fh^\rtime}}
\newcommand{\tht}{\vartheta} 
\newcommand{\deKspa}{{\partial^\rspace K}}
\newcommand{\deKtime}{{\partial^\rtime K}}
\newcommand{\hp}{_{hp}}
\newcommand\Vhp{v\hp}
\newcommand\Shp{\bsigma\hp}
\newcommand*{\Norm}[1]{\left\|#1\right\|}
\newcommand{\DG}{_{\mathrm{DG}}}
\newcommand{\DGp}{_{\mathrm{DG^+}}}
\newcommand{\bVp}{{\bV\hp}}
\newcommand{\vs}{{(v,\bsigma)}}
\newcommand{\vsh}{{(v\hp,\bsigma\hp)}}
\newcommand{\wt}{{(w,\btau)}}
\newcommand{\wth}{{(w\hp,\btau\hp)}}
\DeclareMathOperator{\spn}{span}

\newcommand{\mi}{{\boldsymbol{i}}}
\newcommand{\mj}{{\boldsymbol{j}}}
\newcommand{\mk}{{\boldsymbol{k}}}
\providecommand{\ind}{\lambda}

\newcommand\refb{b_J}
\newcommand{\QT}{{\mathbb{Q\!T}}}
\newcommand{\QU}{{\mathbb{Q\!U}}}
\newcommand{\QW}{{\mathbb{Q\!W}}}
\newcommand{\PP}{{\mathbb{Y}}}

\usepackage{comment}
\usepackage{floatrow}
\usepackage{csvsimple} 
\usepackage{siunitx} 
\usepackage{tabularx}
\usepackage{cleveref} 
\usepackage{booktabs}

\definecolor{pscol}{rgb}{0,0.6,0}

\pgfplotscreateplotcyclelist{paulcolors}{%
teal,every mark/.append style={solid,fill=teal},mark=*,very thick,mark size=3pt\\%
orange,every mark/.append style={solid,fill=orange},mark=square*,very thick,mark size=3pt\\%
violet,every mark/.append style={solid,fill=violet},mark=diamond*,very thick,mark size=3pt\\%
teal,densely dashed,every mark/.append style={solid,fill=teal},mark=*,very thick,mark size=3pt\\%
orange,densely dashed,every mark/.append style={solid,fill=orange},mark=square*,very thick,mark size=3pt\\%
violet,densely dashed,every mark/.append style={solid,fill=violet},mark=diamond*,very thick,mark size=3pt\\%
teal,dash dot,every mark/.append style={solid,fill=teal},mark=*,very thick,mark size=3pt\\%
orange,dash dot,every mark/.append style={solid,fill=orange},mark=square*,very thick,mark size=3pt\\%
violet,dash dot,every mark/.append style={solid,fill=violet},mark=diamond*,very thick,mark size=3pt\\%
}
\pgfplotsset{
    discard if not/.style 2 args={
        x filter/.code={
            \edef\tempa{\thisrow{#1}}
            \edef\tempb{#2}
            \ifx\tempa\tempb
            \else
                
            \fi
        }
    }
}
\pgfplotsset{compat=1.14}

\newcommand{\logLogSlopeTriangle}[5]
{

    \pgfplotsextra
    {
        \pgfkeysgetvalue{/pgfplots/xmin}{\xmin}
        \pgfkeysgetvalue{/pgfplots/xmax}{\xmax}
        \pgfkeysgetvalue{/pgfplots/ymin}{\ymin}
        \pgfkeysgetvalue{/pgfplots/ymax}{\ymax}

        \pgfmathsetmacro{\xArel}{#1}
        \pgfmathsetmacro{\yArel}{#3}
        \pgfmathsetmacro{\xBrel}{#1-#2}
        \pgfmathsetmacro{\yBrel}{\yArel}
        \pgfmathsetmacro{\xCrel}{\xArel}

        \pgfmathsetmacro{\lnxB}{\xmin*(1-(#1-#2))+\xmax*(#1-#2)} 
        \pgfmathsetmacro{\lnxA}{\xmin*(1-#1)+\xmax*#1} 
        \pgfmathsetmacro{\lnyA}{\ymin*(1-#3)+\ymax*#3} 
        \pgfmathsetmacro{\lnyC}{\lnyA-#4*(\lnxA-\lnxB)}
        \pgfmathsetmacro{\yCrel}{\lnyC-\ymin)/(\ymax-\ymin)} 

        \coordinate (A) at (rel axis cs:\xArel,\yArel);
        \coordinate (B) at (rel axis cs:\xBrel,\yBrel);
        \coordinate (C) at (rel axis cs:\xCrel,\yCrel);

        \draw[#5]   (A)-- node[pos=0.5,anchor=north] {}
                    (B)-- 
                    (C)-- node[pos=0.5,anchor=west] {#4}
                    cycle;
    }
}

\begin{document}
\title{A space--time quasi-Trefftz DG method\\ for the wave equation with piecewise-smooth coefficients}
\author{Lise-Marie Imbert-G\'erard\footnote{Department of Mathematics, University of Arizona, USA (lmig@math.arizona.edu)}, 
Andrea Moiola\footnote{Department of Mathematics, University of Pavia, Italy (andrea.moiola@unipv.it)}, 
Paul Stocker\footnote{Institute for Numerical and Applied Mathematics, University of G\"ottingen, Germany
(p.stocker@math.uni-goettingen.de)\!
}}
\maketitle
\begin{abstract}
Trefftz methods are high-order Galerkin schemes in which all discrete functions are elementwise solution of the PDE to be approximated.
They are viable only when the PDE is linear and its coefficients are piecewise-constant.
We introduce a ``quasi-Trefftz'' discontinuous Galerkin method for the discretisation of the acoustic wave equation with piecewise-smooth material parameters: the discrete functions are elementwise approximate PDE solutions.
We show that the new discretisation enjoys the same excellent approximation properties as the classical Trefftz one, and prove stability and high-order convergence of the DG scheme.
We introduce polynomial basis functions for the new discrete spaces and describe a simple algorithm to compute them.
The technique we propose is inspired by the generalised plane waves previously developed for time-harmonic problems with variable coefficients; it turns out that in the case of the time-domain wave equation under consideration the quasi-Trefftz approach allows for polynomial basis functions.


\end{abstract}

\section{Introduction}

The {\em Trefftz methods} are a class of Galerkin schemes for the approximation of linear partial differential equations.
Their distinctive property is that the restrictions to mesh elements of all test and trial functions are particular solutions of the underlying PDE.
The variational formulation weakly enforces interelement continuity and initial/boundary conditions.
They are named after the seminal work of Erich Trefftz \cite{Trefftz1926}.
The main advantage of Trefftz schemes over more classical ones is the higher accuracy for comparable numbers of degrees of freedom.

Trefftz methods have proved particularly successful for wave propagation in time-harmonic regime; see e.g.\ \cite{TrefftzSurvey} for a survey of the scalar case.
Trefftz methods are often formulated in a {\em discontinuous Galerkin} (DG) framework.
DG methods are a popular choice for time-domain wave propagation, due to their flexibility, efficiency and simplicity; see e.g.\
\cite{AMMQ16,Jo93,MoRi05,MuScSc18,WienersxtDGWave2019}.

Trefftz discretisations of time-dependent PDEs are intrinsically {\em space--time} methods (as opposed to space semi-discretisations and time-stepping): for the test and trial functions to be solution of the PDE they need to be functions of both space and time variables.
Trefftz DG schemes developed for time-domain (acoustic, electromagnetic and elastic) wave problems include
interior penalty (IP-DG) \cite{BGL2016}, hybrid DG (involving Lagrange multipliers on mesh interfaces) \cite{PFT09,WTF14}, and versions related to the ``ultra-weak variational formulation'' \cite{SpaceTimeTDG,MoPe18,BCDS20,KretzschmarPhD,KSTW2014,EKSW15,StockerSchoeberl}.
In all cases, a sensible choice of the DG numerical fluxes allows to write space--time Trefftz DG schemes as simply as standard ``DG-in-space+time-stepping'' schemes.
In particular, there is no need to solve huge global space--time linear systems but implicit (and, on suitable meshes, even explicit, \cite{StockerSchoeberl}) time-stepping is possible.
Numerical experiments on a wide range of academic test cases have shown excellent properties in terms of approximation and convergence rates 
\cite{BGL2016,PFT09,WTF14,SpaceTimeTDG,BCDS20,KretzschmarPhD,KSTW2014,EKSW15,StockerSchoeberl},
dissipation \cite{BGL2016,KretzschmarPhD,StockerSchoeberl,EKSW15}, dispersion \cite{KretzschmarPhD,EKSW15}, conditioning \cite{SpaceTimeTDG}, and even parallelism \cite{StockerSchoeberl}.

Since a sufficiently rich family of local exact solutions of the PDE is needed, almost always, {\em Trefftz schemes require PDE coefficients to be elementwise constant}.
However, many relevant wave propagation problems take place in a smoothly varying medium: classical examples are well-known in aeroacoustics, underwater acoustics, plasma physics, biomedical imaging, etc.
Approximation of smooth coefficients by piecewise-constant ones is not a viable strategy because it immediately spoils high-order convergence, which is one of the strongest reasons to opt for a Trefftz approach.
In the case of time-harmonic acoustic wave propagation (Helmholtz equation) Trefftz methods were adapted to smoothly varying coefficients with the introduction of {\em generalised plane waves} (GPWs) in \cite{ImbertGeradDespres2014}.
GPWs are not exact PDE solutions but rather ``solutions up to a given order'', in the sense of Taylor polynomials.
GPWs extend the accuracy property of Trefftz schemes to a much wider setting (provably for $h$-convergence \cite{IG15}, so far only numerically for $p$-convergence).
The critical point to construct GPWs relies on the choice of an ansatz, mimicking the oscillatory behaviour of plane waves, while allowing for more degrees of freedom.
However this is, in a sense, due more to the nature of the Helmholtz equation than to the GPW idea in itself.

In the present paper, inspired by the GPW idea, we propose an {\em extension of the space--time Trefftz DG scheme for the acoustic wave equation of \cite{MoPe18} to the case of smoothly-varying material parameters}.
Since the Galerkin basis functions are solution of the PDE up to a given order (with respect to the mesh size), the scheme is referred to as {\em quasi-Trefftz DG} method.
A surprising outcome is that test and trial basis functions can be taken as polynomials, and their coefficients can be computed with a simple iteration, which is initialised by assigning their values at a given time.
Their computation uses the first Taylor-expansion terms of the
material coefficient functions $\rho(\bx),G(\bx)$.
Here, $\rho$ and $G$ are two piecewise-smooth scalar coefficients are associated to the space and the time differential operators (or to the vector and the scalar equations in the first-order system), respectively, see \eqref{eq:IBVP}--\eqref{eq:IBVP_U}.
The problem wavespeed is $c=(\rho G)^{-1/2}$.
The Trefftz-DG method for the slightly simpler case where $\rho=1$ and the only variable coefficient is $G=c^{-2}$ is introduced and analysed in the first arXiv version of this manuscript.

The definition, the algorithm for the basis construction, and the analysis of the polynomial quasi-Trefftz discrete space properties are the main novel contributions of the present paper, and will be detailed in the next section.

\subsection{Outline of the paper and main contributions}

We give several preliminary definitions and notation concerning
the initial boundary value problems to be discretised in \S\ref{s:IBVP},
the space--time meshes in \S\ref{s:Mesh},
the numerical parameters needed for the definition of the DG scheme in \S\ref{s:Fluxes},
the mesh-dependent norms used in the error analysis in \S\ref{s:Norms},
along with 
multi-index notation, Taylor polynomials, wave operators and anisotropic weighted norms in \S\ref{s:DefNot}.

The variational formulation of the quasi-Trefftz DG method is introduced in \S\ref{s:DG}.
Remark~\ref{rem:DiffMoPe18} compares this formulation with some closely related ones appeared in \cite{MoPe18,BMPS20,MoRi05,StockerSchoeberl}.
Well-posedness, stability and quasi-optimality of the quasi-Trefftz DG scheme are described in \S\ref{s:WellP}, heavily relying on \cite{MoPe18}. This section also briefly lists several related results such as sharper bounds under more restrictive assumptions, energy dissipation, as well as error bounds on interfaces and partial cylinders.

The polynomial, local, discrete, quasi-Trefftz space $\QU^p(K)$ for the (smooth-wavenumber) second-order wave equation is defined in \S\ref{s:QUlocal}, $p$ standing for the polynomial degree of the basis functions on a mesh element $K$.
Proposition~\ref{prop:Approx} shows that for an appropriate choice of $p$ all wave equation solutions are approximated by this space with high orders in the (space--time) element size.
This is a fully-explicit, high-order, $h$-convergence result; on the other hand $p$-convergence results (i.e.\ regarding convergence for increasing polynomial degrees) on general elements are not available, neither for the Trefftz DG for the constant-coefficients wave equation \cite{MoPe18}, nor for the GPW-based DG scheme for the Helmholtz equation \cite{IG15}.
These best-approximation estimates lead to convergence bounds for the quasi-Trefftz DG scheme in \S\ref{s:QWglobal}.
In particular, for a suitable choice of the numerical parameters entering the DG formulation, we obtain the same orders of convergence as in the constant-coefficient case \cite[\S6]{MoPe18}, even if we require stronger solution regularity (see Remark~\ref{rem:CvsH}).

For the method to be practical, of course one needs to be able to explicitly compute the basis functions: we describe a family of bases in \S\ref{s:Basis}.
Given any basis of the classical polynomial space in $n$ real variables, Algorithms~\ref{algo:Basis1D} and \ref{algo:BasisND} give a simple recipe for the computation of a corresponding quasi-Trefftz basis (in $n=1$ and $n>1$ space dimensions, respectively).

Sections \ref{s:QUlocal}--\ref{s:Basis} focus on quasi-Trefftz schemes for solutions of the second-order wave equation.
However the DG scheme~\eqref{eq:DG} applies more generally to the first-order acoustic wave equation.
Thus \S\ref{s:fobasis} briefly describes another quasi-Trefftz discrete space, suited for the acoustic first-order system, together with the recipe for the computation of its basis.
In the constant-coefficient case the two classes of discrete spaces were proposed and analysed in \cite[\S6.1--6.2]{MoPe18}.

In Section \ref{sec:num} we illustrate the results of several numerical experiments for the implementation of the quasi-Trefftz DG method in NGSolve\footnote{The code is available online at \url{https://github.com/PaulSt/NGSTrefftz}}.
In particular, we briefly discuss the dependence on the penalty parameters and the orders of convergence, we compare the proposed method against standard polynomial and Trefftz DG schemes, and consider both prismatic and ``tent-pitched'' meshes, corresponding to implicit and semi-explicit time-stepping respectively.

\subsection{Extensions and future work}

We envisage that the quasi-Trefftz method proposed here can be extended to elastic and electromagnetic wave propagation in heterogeneous materials, and more generally to a wider class of hyperbolic or Friedrichs systems.
(For constant-coefficient examples of space--time Trefftz DG schemes for elastodynamics, electromagnetics and kinetic equations/transport models see \cite{BCDS20}, \cite{EKSW15,KretzschmarPhD} and \cite{BDM20}, respectively.)
The proposed approach might be effective also for the approximation of PDEs whose nature changes in the computational domain, exemplified by the Euler--Tricomi equation 
($\partial_{x}^2u+x\partial_{y}^2u=0$), 
used for applications in transonic flows and plasma physics.

The numerical analysis performed here is only a first step towards the establishment of a more comprehensive theory of quasi-Trefftz polynomial schemes.
More work is needed to address refined approximation estimates in Sobolev norms (see Remark~\ref{rem:CvsH}), the treatment of less regular solutions (e.g.\ with corner singularities), the proof of error bounds in mesh-independent norms, the analysis of dispersion and dissipation properties.
A very significant and challenging extension is the analysis of the approximation properties for increasing polynomial degrees, i.e.\ the $p$-convergence.
The construction of non-polynomial quasi-Trefftz spaces could be relevant, for example, in order to efficiently approximate solutions that are localised in frequency.

\section{Model problem}\label{s:IBVP}
We consider the following initial boundary value problem (IBVP) for the first-order acoustic wave equation:
\begin{align}\label{eq:IBVP}
\left\{\begin{aligned}
&\nabla v+\rho\partial_t\bsigma = \bzero &&\iin Q,\\
&\nabla\cdot\bsigma+G\partial_t v = 0 &&\iin Q,\\
&v(\cdot,0)=v_0, \quad \bsigma(\cdot,0)=\bsigma_0 &&\oon \Omega,\\
&v=g_D &&\oon \GD\times [0,T],\\
&\bsigma\cdot\bn_\Omega^x=g_N &&\oon \GN\times [0,T],\\
&\tht  v-\bsigma\cdot \bn_\Omega^x=g_R &&\oon \GR\times [0,T].
\end{aligned}
\right.
\end{align}
Here
\begin{align*}
n\in \IN&\text{ is the physical space dimension,}\\
\Omega\subset\IR^n & \text{ is an open, bounded, Lipschitz polytope,}\\
T>0 & \text{ is the final time,}\\
Q=\Omega\times(0,T)&\text{ is the space--time cylinder,}\\
\vs:Q\to\IR\times\IR^n &\text{ are the unknown fields (e.g.\ acoustic pressure and velocity),}\\
\bn^x_\Omega\in\IR^n&\text{ is the outward pointing normal unit vector on }\deO,\\
\GD,\GN,\GR&\text{ are a partition of $\deO$, one or two of them may be empty,}\\
v_0\in \LtO,\bsigma_0\in\LtO^2&\text{ are the initial conditions,}\\
g_D,g_N,g_R&\text{ are Dirichlet, Neumann and Robin boundary data, respectively,}\\
0<\tht\in L^\infty(\GR\times[0,T])&\text{ is a impedance parameter with the \rv{same physical units} of $(G/\rho)^{1/2}$,}\\
\nabla,\nabla\cdot&\text{ are the gradient and divergence operators in the space variable $\bx$ only,}\\
\partial_t&\text{ is the time derivative,}\\
0<\rho,G\in L^\infty\OO  &\text{ are the material coefficients, independent of time and piecewise-smooth,}\\
c:=(\rho G)^{-1/2}
&\text{ is the wavespeed.}
\end{align*}
\rv{When considering $\rho,G,c$ over $Q$, we extend them constant in time.}
With the same notation and assumptions, the corresponding IBVP for the second-order (scalar) wave equations reads as
\begin{align}\label{eq:IBVP_U}
\left\{\begin{aligned}
&-\nabla\cdot\Big(\frac1\rho \nabla u\Big)+ G\partial_t^2u=0&&\iin Q,\\
& u(\cdot,0)=u_0, \quad \partial_t u (\cdot,0)=u_1 &&\oon \Omega,\\
&\partial_t u=g_D &&\oon \GD\times [0,T],\\
&-\frac1\rho\bn_\Omega^x\cdot\nabla u=g_N &&\oon \GN\times [0,T],\\
&\tht \partial_t u+\frac1\rho\bn_\Omega^x\cdot\nabla u=g_R  &&\oon \GR\times [0,T].
\end{aligned}
\right.
\end{align}
If $\rho\bsigma_0$ is a gradient, the two IBVPs are seen to be equivalent by relating the unknowns as $v=\partial_t u$, $\bsigma=-\frac1\rho\nabla u$, and the initial conditions as $v_0=u_1$ and $\bsigma_0=-\frac1\rho\nabla u_0$.

If the model under consideration is meant to describe acoustic waves in a fluid, then $u,\rho,(\rho G)^{-1}$ represent acoustic pressure, mass density and the partial derivative of the pressure with respect to the density at constant entropy, respectively.
If the model describes TE (or TM) modes in electromagnetics, $\rho$ and $G$ represent electric permittivity $\epsilon$ and magnetic permeability $\mu$ (or viceversa), and $u$ a component of the magnetic (or electric) field.

The well-posedness of IBVP \eqref{eq:IBVP} with $\rho=1$ in Bochner spaces is briefly discussed in \cite[\S2.2]{BMPS20}.
In two space dimensions $d=2$, the regularity of the solution in corner-weighted Sobolev space of Kondrat'ev type is investigated in detail in \cite{KokPlam2004,LuTu15} and used in the convergence analysis of DG schemes in, e.g., \cite{MuSc15,MullerPhD,BMPS20}.

\section{Discontinuous Galerkin discretisation}

In this section we closely follow \cite[\S3--5]{MoPe18}; we extend assumptions, notation, definitions and results to the case of piecewise-smooth $\rho,G$ and non-Trefftz discretisations.

\subsection{Mesh assumptions and notation}\label{s:Mesh}

The space--time domain $Q$ is subdivided in a non-overlapping mesh $\calT_h$, where every element $K\in\calT_h$ is an $n+1$-dimensional Lipschitz polytope.
We assume that $\rho|_K,G|_K\in C^\infty(K)$ for all $K\in\calT_h$, 
and that each face ($F=\conj{K_1}\cap\conj{K_2}$ \rv{or $F=\conj{K_1}\cap\partial Q$} with positive $n$-dimensional measure for some $K_1,K_2\in\calT_h$) is an $n$-dimensional polytope.
We denote by $(\bn^x_F,n^t_F)\in\IR^{n+1}$ the unit normal vector orthogonal to a mesh face $F$, with $n^t_F\ge0$ and $|\bn^x_F|^2+(n^t_F)^2=1$.
Recalling that $c=(\rho G)^{-1/2}$, we assume that each face $F$ is either
\begin{align}\label{eq:HorVerFaces}
\text{space-like, i.e. }&\rv{n^t_F>0 \text{ and }}|\bn^x_F|\sup_{(\bx,t)\in F} c(\bx)\le n^t_F, \\
\text{or time-like, i.e. }&n^t_F=0. \nonumber
\end{align}
\rv{The traces of $c$ from both sides of space-like faces coincide because $c=c(\bx)$ is independent of time: 
all discontinuities of $c$ are 
captured by the time-like faces of the mesh.
}
A space-like face $ F$ lies below (i.e.\ in the past of) the cone of dependance of each of its points; its slope (when seen as the graph the function $\bx\mapsto t$ such that $(\bx,t)\in F$) is bounded by $1/c(\bx)$. A time-like face is a union of segments parallel to the time axis.
The class of meshes includes both Cartesian-product meshes such as those of \cite{BMPS20} ($\bn^x_F=0$, $n^t_F=1$ on all space-like faces) and tent-pitched meshes such as those of \cite{StockerSchoeberl,MoRi05} (all faces are space-like, $n^t_F\approx \frac c{\sqrt{1+c^2}}$); see two examples plotted in Figure~\ref{fig:mesh}.

We choose a ``centre point'' $(\bx_K,t_K)\in K$ for each mesh element $K\in\calT_h$, for example the barycentre, which will be used in the proof of the approximation estimates and to define the basis functions.
We define a radius and a ``weighted radius'' of each element as
\begin{equation}\label{eq:rK}
r_K:=\sup_{(\bx,t)\in K}\abs{(\bx,t)-(\bx_K,t_K)},
\qquad
r_{K,c}:=\sup_{(\bx,t)\in K}\abs{\big(\bx,c(\bx)t\big)-\big(\bx_K,c(\bx_K)t_K\big)},
\end{equation}
with $\abs{\cdot}$ the Euclidean distance in $\IR^{n+1}$.

We use the following notation for the mesh skeleton and its parts:
\begin{align*}
\Fh&:=\bigcup\nolimits_{K\in\calT_h}\partial K,\qquad
\Fspa:= \bigcup\{F\subset  \mathcal F_h \text{ space-like face} \},\qquad
\Ftime:= \bigcup\{F\subset  \mathcal F_h \text{ time-like face} \},\\
\FO&:=\Omega\times\{0\},\hspace{12mm} \FT:=\Omega\times\{T\},\\ 
\FD&:=\GD\times[0,T],\qquad
\FN:=\GN\times[0,T],\qquad
\FR:=\GR\times[0,T]. 
\end{align*}
We use standard DG notation for averages $\mvl\cdot$, space normal jumps $\jmp{\cdot}_\bN$ and time (full) jumps $\jmp{\cdot}_t$ of piecewise-continuous scalar and vector fields on internal faces:
on $F=\deK_1\cap\deK_2$, for $K_1,K_2\in\calT_h$,
\begin{align*}
\mvl{w}&:=\frac{w_{|_{K_1}}+w_{|_{K_2}}}2,\qquad
&\mvl{\btau}&:=\frac{\btau_{|_{K_1}}+\btau_{|_{K_2}}}2, \\
\jmp{w}_\bN&:= w_{|_{K_1}}\bn_{K_1}^x+w_{|_{K_2}} \bn_{K_2}^x,\qquad
&\jmp{\btau}_\bN&:= \btau_{|_{K_1}}\cdot\bn_{K_1}^x+\btau_{|_{K_2}} \cdot\bn_{K_2}^x,\\
\jmp{w}_t&:= w_{|_{K_1}}n_{K_1}^t+w_{|_{K_2}} n_{K_2}^t=(w^--w^+)n^t_F,\qquad
&\jmp{\btau}_t&:= \btau_{|_{K_1}} n_{K_1}^t+\btau_{|_{K_2}} n_{K_2}^t= (\btau^--\btau^+)n^t_F.
\end{align*}
The superscripts $+/-$ denote the traces on a space-like face taken from the mesh elements placed ``after'' ($+$) and ``before'' ($-$) the face itself, according to the time direction.
We introduce a piecewise-constant function $\gamma$ defined on $\Fspa\cup\FO\cup\FT$, measuring how close to characteristic cones the space-like mesh faces are:
\begin{align}\label{eq:gamma}
\gamma:= \frac{\N{c}_{C^0(F)}|\bn^x_F|}{n^t_F}\;\oon F
\subset\Fspa, \qquad \gamma:=0\;\oon \FO\cup\FT.
\end{align}
We define a ``space-like interface'' as a connected union of space-like faces $\Sigma\subset \Fspa\cup\FO\cup\FT$ that is the graph of a Lipschitz-continuous function $f_\Sigma:\conj \Omega\to [0,T]$. By \eqref{eq:HorVerFaces}, the Lipschitz constant of $f_\Sigma$ in $\bx\in\Omega$ will be at most $c^{-1}(\bx)$. 
The unit normal vector on $\Sigma$ is denoted $(\bn^x_\Sigma,n^t_\Sigma)$.

\subsection{DG flux and penalisation parameters}\label{s:Fluxes}
We fix three ``numerical flux parameter'' functions on portions of the mesh skeleton, and two ``volume penalisation coefficients'':
$$
\alpha\in L^\infty\left(\Ftime\cup\FD\right), \qquad
\beta\in L^\infty\left(\Ftime\cup\FN\right),\qquad
\delta\in L^\infty\left(\FR\right),\qquad
\mu_1,\mu_2\in L^\infty(Q).
$$
We assume that all these are uniformly positive and bounded:
$$
\alpha,\beta,\delta,\mu_1,\mu_2>0, 
\qquad \N{\delta}_{L^\infty(\FR)}<1,$$
$$\N{\alpha^{-1}}_{L^\infty(\Ftime\cup\FD)},
\N{\beta^{-1}}_{L^\infty(\Ftime\cup\FN)},
\N{\delta^{-1}}_{L^\infty(\FR)},
\N{\mu_1^{-1}}_{L^\infty(Q)}, 
\N{\mu_2^{-1}}_{L^\infty(Q)}<\infty.
$$
We also define the values
\begin{equation}\label{eq:mupm}
\mu_{K+}:=\max\big\{\N{\mu_1}_{L^\infty(K)},\N{\mu_2}_{L^\infty(K)}\big\},\quad
\mu_{K-}:=\max\big\{\N{\mu_1^{-1}}_{L^\infty(K)},\N{\mu_2^{-1}}_{L^\infty(K)}\big\},\quad \forall K\in\calT_h.
\end{equation}
{To obtain optimal convergence rates, the volume penalty parameters $\mu_1$ and $\mu_2$ need to be scaled proportionally to the local element size $r_{K,c}$: we assume this in \eqref{eq:muChoice}.
    For dimensional consistency, $\alpha$ and $\beta^{-1}$ must have the same \rv{physical units as} $Gc$ and of the impedance parameter $\vartheta$: we assume this in \eqref{eq:FluxChoice}.
}

\subsection{DG formulation}\label{s:DG}

Let $\bVp\Th$ be a {closed (e.g.\ finite-dimensional)} subspace of the broken Sobolev space
$$
\bH\Th:=\prod_{K\in\calT_h}
\big(H^1(K)\times H^1(K)^n \big).
$$
We consider the following variational formulation:
\begin{align}
\text{Seek}&\quad (\Vhp,\Shp)\in \bVp (\calT_h)\nonumber\\
\text{such that}&\quad
\calA(\Vhp,\Shp; w ,\btau )=\ell( w ,\btau )\quad 
\forall ( w ,\btau )\in \bVp (\calT_h), \label{eq:DG}\\
\text{where}&\nonumber\\\nonumber
\calA(\Vhp,&\Shp; w ,\btau ):=
-\sum_{K\in\calT_h} \int_K\bigg(\Vhp\Big(\nabla\cdot\btau+ G\partial_t w \Big) +\Shp\cdot\Big(\rho\partial_t\btau +\nabla w  \Big)\bigg)\di V\\
+&\int_{\Fspa}\big( G\Vhp^-\jmp{w}_t+\rho\Shp^-\cdot\jmp{\btau}_t+\Vhp^-\jmp{\btau}_\bN+\Shp^-\cdot\jmp{w}_\bN\big)\di S
+\int_\FT ( G\Vhp  w +\rho\Shp \cdot\btau )\di \bx
\nonumber\\
+&\int_\Ftime\!\! \big( \mvl{\Vhp}\jmp{\btau }_\bN+\mvl{\Shp}\cdot\jmp{ w }_\bN
+\alpha\jmp{\Vhp}_\bN\cdot\jmp{ w }_\bN+ \beta\jmp{\Shp}_\bN\jmp{\btau }_\bN
\big)\di S
\nonumber\\
+&\int_\FD \big(\Shp\cdot\bn_\Omega^x\, w +\alpha \Vhp w   \big) \di S+\int_\FN\big(\Vhp(\btau\cdot\bn_\Omega^x)
+\beta(\Shp\cdot\bn_\Omega^x)(\btau\cdot\bn_\Omega^x)\big)\di S
\nonumber\\
+&\int_\FR\Big((1-\delta)\tht \Vhp w+(1-\delta)v\hp(\btau\cdot\bn_\Omega^x)
+\delta(\Shp\cdot\bn_\Omega^x) w+\frac{\delta}\tht(\Shp\cdot\bn_\Omega^x) (\btau\cdot\bn_\Omega^x)\Big)\di S\nonumber\\
+&\sum_{K\in\calT_h} \!\int_K \!\Big(\mu_1 \frac1G\big( \nabla\cdot\Shp+ G\partial_t \Vhp\big)
\big( \nabla\cdot\btau+ G\partial_t w\big) 
+ \mu_2\frac1\rho\big(\nabla\Vhp+\rho\partial_t\Shp\big)\cdot\big(\nabla w+\rho\partial_t\btau\big)\Big) \di V,
\nonumber\\
\ell( w ,\btau )&:=
\int_\FO (  G v_0 w  +\rho\bsigma_0\cdot \btau )\di \bx
\nonumber\\
+&\int_\FD g_D\big(\alpha  w -\btau\cdot\bn_\Omega^x\big)\di S
+\int_\FN g_N \big(\beta\,\btau\cdot\bn_\Omega^x-w\big)\di S
+\int_\FR g_R \Big((1-\delta)w-\frac{\delta}\tht\,\btau\cdot\bn_\Omega^x\Big)\di S.
\nonumber
\end{align}
Noting that all terms involving $\rho$ and $G$ are integrated by parts with respect to the time variable only, the derivation in \cite[\S4]{MoPe18} shows that the formulation \eqref{eq:DG} is consistent: 
\begin{align}\nonumber
&{\text{let } \vs \text{ with }
v\in H^1(Q),\;
\bsigma \in H^1\big(0,T;L^2\OO^n\big)\cap L^2\big(0,T;H(\mathrm{div},\Omega)\big),\;
\bsigma|_K\in H^1(K)^n \;\forall K\in\calT_h,
}\\
&\text{{be} solution of \eqref{eq:IBVP}, then } \quad
\calA(v,\bsigma;w,\btau)=\ell\wt \quad \forall\wt\in \bH\Th.
\label{eq:VSregularity}
\end{align}

\begin{rem}\label{rem:DiffMoPe18}
The differences between \eqref{eq:DG} and \cite[equation~(7)]{MoPe18} are the following: 
(i) we allow position-dependent and possibly discontinuous material parameters $\rho$ and $G$ (\cite[equation~(7)]{MoPe18} is written in terms of $c$ only);
(ii) we allow fields that are not local solution of the PDE system (i.e.\ our method is not Trefftz);
(iii) as a consequence we have a volume term in $\calA(\cdot,\cdot)$, ensuring consistency;
(iv) we have a further stabilisation/penalisation volume term (the term involving $\mu_1,\mu_2$).
This term can be understood as a Galerkin--least squares (GLS) correction.
The formulation \eqref{eq:DG} exactly corresponds to the numerical fluxes $\hat v_{hp},\hat \bsigma_{hp}$ in \cite[pp.~396--397]{MoPe18}, except for those on $\FR$ where we have absorbed in $\tht$ the dependence on the material parameters.

The formulation of \cite{MoPe18} has been studied also in \cite{StockerSchoeberl} and extended to the non-Trefftz case, with piecewise-constant coefficient, in \cite{BMPS20} (with tensor-product and sparse polynomial bases).
With appropriate choices of the numerical flux parameters the present formulation is a special case of that in \cite{MoRi05} (see the comparison in \cite[Rem.~4]{MoPe18}).
\end{rem}

Although the variational problem \eqref{eq:DG} couples the discrete solution on all mesh elements in $Q$, the structure of the terms on $\Fspa$ (the space-like part of the mesh skeleton) allows to compute the solution $\vsh$ by solving a sequence of smaller linear systems; see \cite[p.~398]{MoPe18}.
E.g., if the elements of a quasi-uniform mesh can be grouped in $N$ ``time-slabs'' $\Omega\times(t_{i-1},t_i)$, with $0=t_0<t_1<\cdots<t_N=T$, $t_i-t_{i-1}\approx T/N$, then the discrete solution on each time-slab can be computed from the solution on the previous time-slab, solving $N$ linear systems of size $O(\dim\bV\hp\Th/N)$ each.
This is equivalent to an implicit time-stepping.

As in classical Trefftz methods, we only consider homogeneous IBVPs (i.e.\ with zero volume source term), which are frequently encountered in wave problems.
To treat non-homogeneous problems ($\nabla v+\rho\partial_t \bsigma=\mathbf\Phi$, $\nabla\cdot\bsigma+G\partial_t v=\psi$) one can proceed in two steps: (i) construct (possibly in parallel) elementwise solutions $(v_K^\#,\bsigma_K^\#)$ of the source problem with artificial local boundary conditions, and (ii) solve for the difference $(v-v_K^\#,\bsigma-\bsigma_K^\#)$ using the quasi-Trefftz DG scheme~\eqref{eq:DG}.
This technique has been developed and analysed for constant-coefficient time-harmonic problems in, e.g., \cite{HuYuan18}.

\subsection{Mesh-dependent norms}\label{s:Norms}

We define two mesh- and flux-dependent norms on $\bH\Th$:
\begin{align}\label{eq:DGnorm}
\Tnorm{(w ,\btau )}^2\DG&:=
\frac12 \Norm{\Big(\frac{1-\gamma}{n^t_F}\Big)^{1/2}G^{1/2}\jmp{w }_t}_{L^2(\Fspa)}^2
+\frac12\Norm{\Big(\frac{1-\gamma}{n^t_F}\Big)^{1/2}\rho^{1/2}\jmp{\btau }_t\rule{0pt}{3mm}}_{L^2(\Fspa)^n}^2 \\\nonumber
&\quad
+\frac12\Norm{G^{1/2}w }^2_{L^2(\FO\cup\FT)} 
+\frac12\Norm{\rho^{1/2}\btau}^2_{L^2(\FO\cup\FT)^n} 
\\&\quad\nonumber
+\Norm{\alpha^{1/2}\jmp{w }_\bN}_{L^2(\Ftime)^n}^2
+\Norm{\beta^{1/2}\jmp{\btau }_\bN}_{L^2(\Ftime)}^2\\\nonumber
&\quad+\Norm{\alpha^{1/2}w }_{L^2(\FD)}^2
+\Norm{\beta^{1/2} \btau\cdot\bn_\Omega^x }_{L^2(\FN)}^2\\\nonumber
&\quad+\Norm{\big((1-\delta)\tht\big)^{1/2} w }_{L^2(\FR)}^2
+\Norm{\Big(\frac{\delta}\tht\Big)^{1/2} \btau\cdot\bn_\Omega^x }_{L^2(\FR)}^2 \\
&\quad+\sum_{K\in\calT_h}\left(\Norm{\mu_1^{1/2}G^{-1/2}(\nabla\cdot\btau+ G\partial_t w) }^2_{L^2(K)} 
+\Norm{\mu_2^{1/2}\rho^{-1/2}(\nabla w+\rho\partial_t\btau)}^2_{L^2(K)^n}\right);
\nonumber\\
\nonumber
\Tnorm{(w ,\btau )}^2\DGp&:=
\Tnorm{(w ,\btau )}^2\DG
+2\Norm{\Big(\frac{n^t_F}{1-\gamma}\Big)^{1/2}G^{1/2} w^-}_{L^2(\Fspa)}^2
+2\Norm{\Big(\frac{n^t_F}{1-\gamma}\Big)^{1/2}\rho^{1/2}\btau^-}_{L^2(\Fspa)^n}^2  
\\&\quad\nonumber
+\Norm{\beta^{-1/2}\mvl{w }}_{L^2(\Ftime)}^2
+\Norm{\alpha^{-1/2}\mvl{\btau }}_{L^2(\Ftime)^n}^2
\\&\quad\nonumber
+\Norm{\alpha^{-1/2}\btau\cdot\bn_\Omega^x }_{L^2(\FD)}^2
+\Norm{\beta^{-1/2} w }_{L^2(\FN)}^2
\\&\quad\nonumber
+\sum_{K\in\calT_h}\left(\Norm{\mu_1^{-1/2}G^{1/2}w}^2_{L^2(K)}
+\Norm{\mu_2^{-1/2}\rho^{1/2}\btau}^2_{L^2(K)}\right).
\end{align}
These are \textit{norms} on the broken Sobolev space $\bH\Th$ defined on the mesh $\calT_h$. 
Indeed, $\Tnorm{(w,\btau)}\DG=0$ for $(w,\btau)\in\bH(\calT_h)$ implies that $(w,\btau)$ is solution of the IBVP with zero initial and boundary conditions, so $(w,\btau)=(0,\bzero)$ by the well-posedness of the IBVP itself; see \cite[Lemma~4.1]{SpaceTimeTDG}.

As in \cite[\S5.3]{MoPe18}, we define the energy of a field $\wt\in\bH\Th$ on a space-like interface $\Sigma$ as
\begin{equation}\label{eq:Energy}
\calE(\Sigma;w,\btau):=
\int_\Sigma \Big(w\btau\cdot\bn^x_\Sigma+\frac12(G w^2+\rho|\btau|^2)n^t_\Sigma\Big)\di S.
\end{equation}

\subsection{Well-posedness, stability, quasi-optimality}\label{s:WellP}

Integration by part on a mesh element gives for any field $\wt\in\bH\Th$
\begin{align}\label{eq:IBP}
&\int_K\bigg(w\Big(\nabla\cdot\btau+ G\partial_t w \Big) +\btau\cdot\Big(\nabla w+\rho\partial_t\btau\Big)\bigg)\di V
- \int_{\deK}
\bigg(w\btau\cdot\bn^x_K+\frac{1}2\Big( G{w^2}+\rho |\btau|^2\Big)n^t_K\bigg)\di S
=0.
\end{align}

The results of \cite[\S5.2]{MoPe18} hold also in the current, slightly extended, setting and are summarised in the following theorem.

\begin{theorem}\label{thm:WP}
The bilinear form $\calA$ is coercive in $\Tnorm{\cdot}\DG$ norm and continuous in 
$\Tnorm{\cdot}\DGp$--$\Tnorm{\cdot}\DG$ norms, and the linear functional $\ell$ is continuous:
\begin{align}\label{eq:Coercivity}
&\calA(w,\btau; w ,\btau )\ge
\Tnorm{(w,\btau)}\DG^2,
\\
\nonumber
&\abs{\calA(v,\bsigma; w ,\btau )}\le C_{\!\calA}\Tnorm{(v,\bsigma)}\DGp\Tnorm{(w,\btau)}\DG,
\quad \text{where }\\&
\hspace{20mm} C_{\!\calA}:=
\begin{cases}
2, &\text{if }\FR=\emptyset,\\
2\max\Big\{\Norm{\frac{1-\delta}\delta}_{L^\infty(\FR)}^{1/2},
\big\|\frac\delta{1-\delta}\big\|_{L^\infty(\FR)}^{1/2}\Big\}
&\text{if }\FR\ne\emptyset,
\end{cases}
\label{eq:Continuity}
\\
&\nonumber
\abs{\ell(w,\btau)}
\le\Big(2\Norm{G^{1/2}v_0}_{L^2(\FO)}^2+2\Norm{\rho^{1/2}\bsigma_0}_{L^2(\FO)}^2
+2\Norm{\alpha^{1/2}g_D}_{L^2(\FD)}^2
\\&\hspace{20mm}
+2\Norm{\beta^{1/2}g_N}_{L^2(\FN)}^2
+\Norm{\tht^{-1/2}g_R}_{L^2(\FR)}^2
\Big)^{1/2} \Tnorm{(w,\btau)}\DGp
\rv{,\qquad\forall \wt,\vs\in\bH\Th}
.
\nonumber
\end{align}
The variational problem \eqref{eq:DG} admits a unique solution $(v\hp,\bsigma\hp)\in \bVp\Th$, for any choice of $\bVp\Th$.
\rv{If the solution $\vs$ is as in \eqref{eq:VSregularity}, then t}he discrete solution satisfies the error bound
\begin{align}\label{eq:QuasiOpt}
\mmbox{\Tnorm{(v-\Vhp,\bsigma-\Shp)}\DG\le  (1+C_{\!\calA}) \inf_{(w,\btau)\in\bVp\Th} \Tnorm{(v-w,\bsigma-\btau)}\DGp.}
\end{align}
Moreover, if $g_D=g_N=0$ (or the corresponding parts $\FD,\FN$ of the boundary are empty) then
\begin{align*}
&\Tnorm{(v\hp,\bsigma\hp)}\DG\le
\bigg(2\Norm{G^{1/2}v_0}_{L^2(\FO)}^2+2\Norm{\rho^{1/2}\bsigma_0}_{L^2(\FO)}^2
+\Norm{\tht^{-1/2}g_R}_{L^2(\FR)}^2 \bigg)^{1/2}.
\end{align*}
\end{theorem}

Of the differences between the methods in \S\ref{s:DG} and in \cite{MoPe18} listed in Remark~\ref{rem:DiffMoPe18}: 
(i) is unimportant for the proof of Theorem~\ref{thm:WP} as the terms involving $\rho$ and $G$ are integrated by parts in time only;
(ii) does not affect the theorem as the Trefftz property is replaced by the presence of the first volume term in \eqref{eq:DG};
the term described in (iii) is taken care by the identity \eqref{eq:IBP};
the term of (iv) coincides with the new term in the $\Tnorm{\cdot}\DG$ norm.

The error bound \eqref{eq:QuasiOpt} slightly differs from the quasi-optimality result (C\'ea lemma) in classical FEM analysis in that the norm ($\Tnorm{\cdot}\DGp$) at the right-hand side is stronger than that ($\Tnorm{\cdot}\DG$) at the left-hand side.
This mismatch is typical of the DG formulation employed here,
not only for hyperbolic equations (as in \cite{MoPe18,BMPS20,SpaceTimeTDG}) but also for the analogous discretisation of the Helmholtz equation, see \cite[\S2.2.1]{TrefftzSurvey}.

Under more restrictive assumptions, slightly stronger results are possible.
On a mesh made of Carte\-sian-product elements, or more generally if $\bn_F^x=\bzero$ on all faces $F\subset\Fspa$, then the coercivity inequality \eqref{eq:Coercivity} is an equality: $\calA(w,\btau; w ,\btau )=\Tnorm{(w,\btau)}\DG^2$ for all $(w,\btau)\in \bH\Th$.
If $g_D=g_N=0$ (or the corresponding parts $\FD,\FN$ of the boundary are empty) then 
the $\Tnorm{(w,\btau)}\DGp$ norm at the right-hand side of the bound on $|\ell(w,\btau)|$ can be substituted by $\Tnorm{(w,\btau)}\DG$.

The bound \eqref{eq:QuasiOpt} allows to control the DG error only in the $\Tnorm{\cdot}\DG$ norm, which involves jumps on internal faces.
However a simple adaptation of the proof allows to control the $L^2$ norm of the traces on space-like interfaces of the error.
Let $\Sigma$ be a space-like interface, as defined in \S\ref{s:Mesh}.
Assume that $\bVp\Th=\{\wt\in\bVp\Th,\supp\wt\subset\{(\bx,t),0\le t\le f_\Sigma(\bx) \}\}\oplus \{\wt\in\bVp\Th,\supp\wt\subset\{(\bx,t),f_\Sigma(\bx)\le t\le T\}\}$ (i.e.\ that the discrete functions are indeed discontinuous across $\Sigma$).
Then \cite[Prop.~1]{MoPe18} gives that 
\begin{align*}
&\calE(\Sigma; v-\Vhp^-,\bsigma-\Shp^-)
\le \frac52\N{(1-\gamma)^{-1}}_{L^\infty(\Sigma)}(1+C_{\!\calA})^2
\inf_{(w,\btau)\in\bVp\Th} \Tnorm{(v-w,\bsigma-\btau)}^2\DGp.
\end{align*}

If $g_D=g_N=0$ and $\GR=\emptyset$ the IBVP \eqref{eq:IBVP} preserves energy: $\calE(\FT;v,\bsigma)=\calE(\FO;v,\bsigma)$.
The DG scheme dissipates energy: $\calE(\FT;v\hp,\bsigma\hp)\le\calE(\FO;v\hp,\bsigma\hp)$; the dissipation can be quantified in terms of the jumps on the faces and the residual in the mesh elements with the same technique of \cite[\S5.3]{MoPe18} and \cite[Rem.~5.7]{BMPS20}.

For any space-like interface $\Sigma$, the results of Theorem~\ref{thm:WP} can be localised to the partial space--time cylinder $Q_\Sigma=\{(\bx,t),\bx\in\Omega, 0<t<f_\Sigma(\bx)\}$, by proceeding as in \cite[\S5.3]{BMPS20}.

In order to reduce the computational cost of the system assembly, one might consider the case without volume penalisation term, i.e.\ with $\mu_1=\mu_2=0$
(on the other hand, the first volume term in $\calA(\cdot,\cdot)$ is needed for consistency).
Indeed, we observe numerically that, at least in some examples (\S\ref{sec:numcoeff}), this choice does not spoil errors and convergence rates. 
Well-posedness for $\mu_1=\mu_2=0$ is shown in \cite{BMPS20} for full polynomial spaces; however, this technique cannot be directly used to prove the well-posedness of the quasi-Trefftz DG scheme because the crucial condition \cite[(5.4)]{BMPS20} is generally not satisfied by the quasi-Trefftz discrete spaces defined in the next section.

\section{Quasi-Trefftz space}

In this section, we present the first extension of Generalized Plane Waves to a time-dependent problem, focusing on two main aspects: properties of the resulting function spaces, and explicit construction of the basis fucntions.
\subsection{Definitions and notation}\label{s:DefNot}
We use standard multi-index notation for partial derivatives and monomials, adapted for space--time fields:
for $\mi=(\mi_\bx,i_t)=(i_{x_1},\ldots,i_{x_n},i_t)
\in \IN^{n+1}_0$,
$D^\mi f := 
\partial_{x_1}^{i_{x_1}}\cdots\partial_{x_n}^{i_{x_n}}\partial_{t}^{i_t} f$ and
$(\bx,t)^\mi=\bx^{\mi_\bx}t^{i_t}
=x_1^{i_{x_1}}\cdots x_n^{i_{x_n}}t^{i_t}$.
We use the canonical basis of $\IR^n$, namely $\{\be_k\in\IR^n,1\leq k\leq  n, (\be_k)_l=\delta_{kl}\}$.
We recall the Leibniz product rule for multi-indices:
\begin{equation}\label{def:Leibniz}
D^\mi(f\tilde f)
=\sum_{\substack{\mj\in\IN_0^{n+1},\;\mj\le\mi}}
\binom\mi\mj D^\mj fD^{\mi-\mj} \tilde f,
\quad\text{where}\quad
\binom\mi\mj:=\frac{\mi!}{\mj!(\mi-\mj)!}
=\binom{i_{x_1}}{j_{x_1}}\cdots
\binom{i_{x_n}}{j_{x_n}}\binom{i_t}{j_t},
\end{equation}
$\mi!:=i_{x_1}!\cdots i_{x_n}!i_t!$ and $\mj\le\mi$ means that the inequality holds component-wise.
The length of a multi-index is $|\mi|=|\mi_\bx|+i_t:=i_{x_1}+\cdots+i_{x_n}+i_t$.
For any field $f\in C^{m}(K)$, denote the Taylor polynomial of order $m+1$ (and polynomial degree at most~$m$) centered at $(\bx_K,t_K)$ by 
\begin{align*}
T^{m+1}_K[f](\bx,t):=
\sum_{|\mi|\le m}\frac1{\mi!}
(\bx-\bx_K)^{\mi_\bx}(t-t_K)^{i_t}
D^\mi f(\bx_K,t_K).
\end{align*}
It follows that
\begin{equation}
\label{eq:coefsTP}
D^\mi T^{m+1}_K[f](\bx_K,t_K)=
\left\{
\begin{array}{ll}
D^\mi f(\bx_K,t_K) &\text{ if }|\mi|\le m,\\
0 &\text{ if }|\mi|>m.
\end{array}
\right.
\end{equation}
Lagrange's form of the Taylor remainder \cite[Cor.~3.19]{Cal10} is the following: if $f$ has $m+1$ continuous derivatives in a neighbourhood of the segment $S$ with extremes $(\bx_K,t_K)$ and $(\bx,t)$, then
\begin{equation}\label{eq:TaylorRem}
\exists (\bx_*,t_*)\in S \quad\text{such that}\quad
f(\bx,t)-T^{m+1}_K[f](\bx,t)=
\sum_{|\mj|=m+1}\frac1{\mj!}(\bx-\bx_K)^{\mj_\bx}(t-t_K)^{j_t}D^\mj f(\bx_*,t_*).
\end{equation}
For elementwise-smooth, positive, spatial functions $G,\rho:\Omega\to\IR$, we denote the variable-coefficient second-order wave operator
\begin{equation*}
(\square_{\rho,G} f)(\bx,t):= \nabla\cdot\Big(\frac1{\rho(\bx)}\nabla f\Big)(\bx,t)-G(\bx)\partial_t^2 f(\bx,t).
\end{equation*}
We denote the partial derivatives of $\frac1\rho$ and $G$ evaluated at an element centre as
\begin{align}\label{eq:Gg}
&\zeta_{\mi_\bx}:=\frac1{\mi_\bx!}D^{(\mi_\bx,0)}\Big(\frac1{\rho(\bx_K)}\Big),\qquad
g_{\mi_\bx}:=\frac1{\mi_\bx!}D^{(\mi_\bx,0)}G(\bx_K),
\quad \text{so that}\\
&\frac1{\rho(\bx)}=\sum_{\mi_\bx\in \IN_0^n}  (\bx-\bx_K)^{\mi_\bx}\zeta_{\mi_\bx},\qquad
G(\bx)=\sum_{\mi_\bx\in \IN_0^n}  (\bx-\bx_K)^{\mi_\bx}g_{\mi_\bx}.
\nonumber
\end{align}

We underline the value of a particular partial derivative that will come into play in the definition of the quasi-Trefftz space:
for $f\in C^{|\mi|+2}(K)$,
$$
D^\mi(\square_{\rho,G} f)
=\sum_{k=1}^n D^{\mi+(\be_k,0)}\left(\frac1{\rho} D^{(\be_k,0)} f\right)+D^{(\mi_\bx,i_t+2)}\Big(Gf\Big),
$$
\begin{equation}\label{eq:DiSquare}
\begin{aligned}
\Rightarrow\quad(D^\mi\square_{\rho,G} f)(\bx_K,t_K)
=&\sum_{k=1}^n\sum_{\mj_\bx\le \mi_\bx+\be_k} \frac{(\mi_\bx+\be_k)!}{\mj_\bx!} \zeta_{\mi_\bx+\be_k-\mj_\bx} D^{(\mj_\bx+\be_k,i_t)}f (\bx_K,t_K)
\\&
- \sum_{\mj_\bx\le \mi_\bx}\frac{\mi_\bx!}{\mj_\bx!} 
g_{\mi_\bx-\mj_\bx}D^{(\mj_\bx,i_t+2)}f (\bx_K,t_K).
\end{aligned}
\end{equation}
This is obtained using Leibniz formula \eqref{def:Leibniz} and noting that only terms with $j_t=i_t$ contribute since $G$ and $\rho$ are independent of time.

We use standard notation for local $C^m$ norms and seminorms, introduce wavespeed-weighted seminorms $C^m_c$, and extend local spaces to global spaces in the piecewise-smooth case: 
for $m\in\IN_0$
\begin{align}\nonumber
\N{f}_{C^0(K)}:=&\sup_{(\bx,t)\in K}|f(\bx,t)|,&&
\abs{f}_{C^m(K)}:=\max_{|\mi|=m}\N{D^\mi f}_{C^0(K)},&&
\abs{f}_{C^m_c(K)}:=\max_{|\mi|=m}\N{c^{-i_t}D^\mi f}_{C^0(K)},
\\
C^m\Th:=&\prod_{K\in\calT_h}C^m(K),&&
\abs{f}_{C^m\Th}:=\max_{K\in\calT_h} \abs{f|_K}_{C^m(K)},&&
\abs{f}_{C^m_c\Th}:=\max_{K\in\calT_h} \abs{f|_K}_{C^m_c(K)}.
\label{eq:Cm}
\end{align}
{The $C^m_c$ seminorms are scaled with the wavespeed $c$ to ensure that all terms compared by the $\max$ have the same physical units (the unit of $f$ times [space]$^{-m}$), similarly to the anisotropic Sobolev norms in \cite[eq.~(37)]{MoPe18}.
In particular, for a time-independent $f$, these seminorms are independent of the wavespeed.
}

\subsection{Local quasi-Trefftz space and approximation properties}\label{s:QUlocal}

We define the ``{quasi-Trefftz}'' space for the second-order wave equation on a mesh element $K\in\calT_h$ as
\begin{equation}\label{eq:QU}
\mmbox{\QU^p(K):=\big\{
f\in\mathbb{P}^p(K) \mid\; D^\mi\square_{\rho,G} f(\bx_K,t_K)=0,\; \forall \mi\in \IN^{n+1}_0,\; |\mi|\le p-2
\big\},
\qquad p\in \IN.}
\end{equation}
This is the space of degree-$p$ space--time polynomials $f$, such that the Taylor polynomial of their image by the wave operator $\square_{\rho,G} f$ vanishes at the element centre $(\bx_K,t_K)$ up to order $p-2$.
From \eqref{eq:DiSquare}, the space $\QU^p(K)$ is well-defined if $\rho\in C^{\max\{p-1,1\}}$ and $G\in C^{\max\{p-2,0\}}$ in a neighbourhood of $\bx_K$. 
For $p=1$, we simply have $\QU^1(K) = \mathbb{P}^1(K)$.

\begin{rem}
Compare the above definition to the 'standard' Trefftz space. 
We define the polynomial Trefftz space for the second-order wave equation with constant parameters inside the mesh element $K$ as 
\begin{equation*}
\mathbb{U}^p(K):=\big\{u\in \mathbb{P}^p(K): \square_{\rho,G} u=0 \;\iin K\big\}.
\end{equation*}
For constant $\rho,G$, the quasi-Trefftz space $\QU^p(K)$ is equal to this space, see Remark \ref{rem:Dimension}.
\end{rem}

The next proposition shows that smooth solutions of the wave equation are approximated in $\QU^p(K)$ with optimal convergence rate with respect to the element radius $r_K$ (recall $r_K\le\diam K$ from the definition \eqref{eq:rK}).
By ``optimal'' we mean that the rate is equal to the rate offered by the full polynomial space $\IP^p(K)$.
We give two approximation estimates: one in classical $C^m$ seminorms and one in their weighted version $C^m_c$ defined in \eqref{eq:Cm}, with $r_{K,c}$ in place of $r_K$.
We will use the latter bound in the convergence analysis of the DG scheme in \S\ref{s:QWglobal}.

\begin{prop}\label{prop:Approx}
Let $u\in C^{p+1}(K)$ be solution of $\square_{\rho,G}u=0$, with $\rho\in C^{\max\{p-1,1\}}(K)$ and $G\in C^{\max\{p-2,0\}}(K)$.

Then the Taylor polynomial $T^{p+1}_K[u]\in \QU^p(K)$.

Moreover, if $K$ is star-shaped with respect to $(\bx_K,t_K)$,
with $r_K$ and $r_{K,c}$ as defined in \eqref{eq:rK} while $q\in\mathbb N_0$ satisfies $q\le p$, then
\begin{align}\label{eq:Approx}
\begin{aligned}
&\inf_{P\in\QU^p(K)}\abs{u-P}_{C^q(K)}
\le \frac{(n+1)^{p+1-q}}{(p+1-q)!} r_K^{p+1-q} \abs{u}_{C^{p+1}(K)},\\
&\inf_{P\in\QU^p(K)}\abs{u-P}_{C^q_c(K)}
\le \frac{(n+1)^{p+1-q}}{(p+1-q)!} r_{K,c}^{p+1-q} \abs{u}_{C^{p+1}_c(K)}.
\end{aligned}
\end{align}
\end{prop}
\begin{proof}
Since $T^{p+1}_K[u]$ is polynomial of degree $p$, in order to show that it belongs to $\QU^p(K)$ we only need to verify that $D^\mi\square_{\rho,G} T^{p+1}_K[u](\bx_K,t_K)=0$, for all $|\mi|\le p-2$.
From the identity \eqref{eq:DiSquare}, this quantity is a linear combination of the partial derivatives of order at most equal to $|\mi|+2\le p$ of the Taylor polynomial at $(\bx_K,t_K)$, which according to \eqref{eq:coefsTP} coincide with the corresponding partial derivatives of $u$:
$$
D^\mi\square_{\rho,G} T^{p+1}_K[u](\bx_K,t_K)
=D^\mi\square_{\rho,G} u (\bx_K,t_K).
$$
Since $\square_{\rho,G}u=0$ in $K$, these partial derivatives vanish, hence $T^{p+1}_K[u]\in \QU^p(K)$. 

We prove the inequality in \eqref{eq:Approx} involving the weighted norms $C^m_c(K)$ using the norm definition \eqref{eq:Cm}, the identity $D^\mi T^{p+1}_K[u]=T^{p+1-|\mi|}_K[D^\mi u]$
for $|\mi|\le p$ from \cite[eq.~(3.5)]{AndreaPhD}, and Taylor's theorem \eqref{eq:TaylorRem}:
\begin{align*}
\inf_{P\in\QU^p(K)}\abs{u-P}_{C^q_c(K)}
&\le\abs{u-T^{p+1}_K[u]}_{C^q_c(K)}
\\
&=\max_{\mi\in\IN_0^{n+1},\; |\mi|=q}\N{c^{-i_t}D^\mi (u-T^{p+1}_K[u])}_{C^0(K)}
\\
&=\max_{\mi\in\IN_0^{n+1},\; |\mi|=q}
\N{c^{-i_t}(D^\mi u-T^{p+1-q}_K[D^\mi u])}_{C^0(K)}
\\
&\overset{\eqref{eq:TaylorRem}}\le \max_{\substack{\mi\in\IN_0^{n+1},\:|\mi|=q}}
\sum_{\substack{|\mj|=p+1-q}}
\frac1{\mj!}\N{\big((\bx,ct)-(\bx_K,ct_K)\big)^\mj
c^{-i_t-j_t}D^{\mi+\mj}u(\bx,t)}_{C^0(K)}
\\
&\le \frac{(n+1)^{p+1-q}}{(p+1-q)!} r_{K,c}^{p+1-q} \abs{u}_{C^{p+1}_c(K)}.
\end{align*}
In the last step we used 
$\sum_{|\mj|=p+1-q}\frac1{\mj!}= \frac{(n+1)^{p+1-q}}{(p+1-q)!}$, \cite[p.~198]{AndreaPhD}.
The first bound in \eqref{eq:Approx} follows from the same chain of inequalities, after dropping all powers of $c$.
\end{proof}

Bound \eqref{eq:Approx} gives approximation rates with respect to the mesh size ($h$-convergence) but is unsuitable for proving convergence for increasing polynomial degrees ($p$-convergence): while the coefficient in the bound is infinitesimal for $p\to\infty$, in general, the seminorm $\abs{u}_{C^{p+1}(K)}$ is not bounded in the same limit.

\begin{rem}
In general, unlike full polynomial spaces, quasi-Trefftz spaces with increasing $p$ are not nested, i.e.\ $\QU^p(K)\not\subset\QU^{p+1}(K)$.
To see this: consider $f(\bx,t)=x_1^2+t^2\in\IP^2$, $\rho=1$ and $G(\bx)=1+x_1$ in a neighbourhood $K$ of $(\bx_K,t_K)=(\bzero,0)$, then
$f$ satisfies $\square_{\rho,G} f=2-2(1+x_1)=-2x_1$, so $\square_{\rho,G}f(\bx_K,t_K)=0$ and $\partial_x\square_{\rho,G} f(\bx_K,t_K)=-2$, therefore $f\in \QU^2(K)\setminus\QU^3(K)$.

On the other hand, $\QU^p(K)\cap\IP^{p-1}(K)\subset \QU^{p-1}(K)$.
Moreover, if $\rho$ is constant, $\QU^1(K)=\IP^1(K)\subset\QU^2(K)=\{f\in\IP^2(K): \Delta f-G(\bx_K,t_K)\partial_t^2 f=0\}$.

Removing the polynomials that conflict with nestedness reduces the convergence order.
Indeed, for the example above we have 
\begin{align*}
\QU^2=\spn\big\{1,\:x,\:t,\:xt,\:x^2+t^2\big\} \quad \text{ and } \quad
\QU^3=\spn\big\{1,\:x,\:t,\:xt,\:x^2-xt^2+t^2,\:3xt^2+x^3,\:t^3+3x^2t\big\}. &
\end{align*}
If we take $\QU_*^2=\spn\{1,x,t,xt\}$, removing from $\QU^2$ the only element that is not in $\QU^3$, then we cannot have $|f-P|_{C^0}\sim r^3$ and $|f-P|_{C^2}\sim r$ in a neighbourhood of $(0,0)$ for any $P\in\QU^2_*$, because $f_{xx},f_{tt}\ne0$ and $P_{xx}=P_{tt}=0$.
\end{rem}

\begin{rem}
One could define a more general version of the space $\QU^{p}$ by imposing the vanishing of the derivatives up to an arbitrary order:
\begin{equation*}
\QU^{p,q}(K):=
\big\{f\in\mathbb{P}^p(K) \mid D^\mi\square_{\rho,G} f(\bx_K,t_K)=0,\ \forall |\mi|\le q-2\big\},
\qquad p,q\in\IN. 
\end{equation*}
For these spaces we have the inclusions $\QU^{p,q+1}(K)\subset\QU^{p,q}(K)\subset\QU^{p+1,q}(K)$ and $\QU^{p,p}(K)=\QU^p(K)$.
However, let us now motivate why the choice $q=p$ is preferable. 
\begin{itemize}
\item For $q<p$ the space $\QU^{p,q}(K)$ is larger than $\QU^p(K)$, but since $\QU^{p,q}(K)\subset\IP^p(K)$ it does not offer better $h$-convergence rates than those showed in Proposition~\ref{prop:Approx} for $\QU^p(K)$.
Moreover, it does not serve as a generalisation of a Trefftz space any more.
Indeed, in the case of constant $\rho$ and $G$ the inclusion
$\IU^p(K):=\{f\in \IP^p(K): \square_{\rho,G} f=0 \iin K\}\subset\QU^{p,q}(K)$
is always true, nevertheless the identity
$\IU^p(K)=\QU^{p,q}(K)$
holds if and only if $q\geq p$.

\item For $q>p$ the space $\QU^{p,q}(K)$ is too small and loses his favorable approximation properties.
Indeed, take $n=1$, $i_x=p-1$, $i_t=0$ and $\rho=1$.
Then, for a solution $f$ of $\square_{1,G} f=0$ we have that 
\begin{align*}
\partial_x^{ i _x}\partial_t^{ i _t}\square_{1,G} T^{p+1}_{K}[f](x_K,t_K)
&\overset{\eqref{eq:DiSquare}}=
\left(\partial_x^{p+1} T^{p+1}_{K} [f]
-\sum_{j_x=0}^{ i _x}\frac{ i _x!}{j_x!}
g_{ i _x-j_x}\partial_x^{j_x}\partial_t^{2} T^{p+1}_{K}[f]\right)(x_K,t_K)
\\&
=-\sum_{j_x=0}^{p-2}\frac{ i _x!}{j_x!}
g_{ i _x-j_x}\partial_x^{j_x}\partial_t^{2} T^{p+1}_{K}[f](x_K,t_K)
\hspace{13mm}
\big(T^{p+1}_{K} [f]\in\IP^p(K)\big)
\\&
\overset{\eqref{eq:coefsTP}}=-\sum_{j_x=0}^{p-2}\frac{ i _x!}{j_x!}
g_{ i _x-j_x}\partial_x^{j_x}\partial_t^{2} f(x_K,t_K)
\hspace{23mm}
\\&
\overset{\eqref{eq:DiSquare}}= 
\Big(
\underbrace{\partial_x^{ i _x}\partial_t^{ i _t}\square_{1,G} f}_{=0}
-\partial_x^{p+1} f
+g_{0}\partial_x^{p-1} \partial_t^2 f\Big)(x_K,t_K)
\\&=\partial_x^{p-1}\big((g_0-G)\partial_t^2f\big)(x_K,t_K)
\hspace{35mm}(\partial_x^2f=G\partial_t^2f)
\\&
\neq 0,\ \text{in general.}
\end{align*}
Therefore, $T^{p+1}_{K}[f]\notin \QU^{p,p+1}(K)$, which contradicts the essential property used in the proof of Proposition~\ref{prop:Approx} to prove the approximation properties of the space.

Moreover, for $q>p$ the dimension of $\QU^{p,q}(K)$ depends on the functions $\rho$ and $G$ (is equal to $\dim\IU^p(K)$ for constant $\rho$, $G$ and smaller in general), while we see in the following that $\dim\QU^p(K)$ is independent of~$G$.
\end{itemize}
Hence, the choice $q=p$ yields 
the smallest subspace of $\IP^p(K)$ in this class that offers the same $h$-convergence rates of $\IP^p(K)$ itself, when approximating solutions of $\square_{\rho,G} u=0$.
\end{rem}

\subsection{Global quasi-Trefftz space and DG convergence bounds}\label{s:QWglobal}

We use the local spaces $\QU^p(K)$ to define a discrete space for the DG scheme of \S\ref{s:DG}.
Recall that $\QU^p(K)$ was constructed for the second-order scalar wave equation, while the DG scheme addresses the first-order system.
A global quasi-Trefftz discrete space can be defined as
\begin{align}\label{eq:QWglobal}
\mmbox{\QW^p\Th:=\left\{
\wt\in\bH\Th: \; w|_K=\partial_t u,\;\btau |_K=-\frac1\rho \nabla u, \; u\in \QU^{p+1}(K)
\right\}, \quad p\in\IN_0.}
\end{align}
The elements of $\QW^p\Th$ are products of vector polynomials of degree at most $p$ and the datum function~$\frac1\rho$.
For general $\rho$ they are non-polynomial fields, however, as we explain in \S\ref{s:Basis}, a quasi-Trefftz implementation only requires the definition and the manipulation of a basis of $\QU^{p+1}(K)$, for each $K$, which is a space of polynomials.

Following again \cite{MoPe18}, for each element $K\in\calT_h$ we introduce a notation for the space-like and the time-like parts of its boundary and three related coefficients:
\begin{align}
\nonumber
\rho_K:=&\inf_K \rho,
\\
\nonumber
\deKspa:=&\deK\cap(\Fspa\cup\FO\cup\FT),\qquad\qquad \deKtime:=\deK\cap(\Ftime\cup\FD\cup\FN\cup\FR),
\\
\xi_K^\rtime:=&\max\Big\{
\N{2c\rho\alpha}_{L^\infty(\deK\cap(\Ftime\cup\FD))}
+\N{c\rho/\beta}_{L^\infty(\deK\cap(\Ftime\cup\FN))},\nonumber\\
\label{eq:xi}
&\hspace{10mm}\N{2\beta/(c\rho)}_{L^\infty(\deK\cap(\Ftime\cup\FN))}
+\N{1/(c\rho\alpha)}_{L^\infty(\deK\cap(\Ftime\cup\FD))},\\
&\hspace{10mm}\N{(1-\delta) c\vartheta}_{L^\infty(\deK\cap(\FR))},\qquad
\N{\delta/(c\vartheta)}_{L^\infty(\deK\cap(\FR))}\Big\}
,\nonumber\\
\xi_K^\rspace:=&\N{n^t_K\big(2(1-\gamma)^{-1}+1\big)}_{L^\infty(\deKspa)},
\nonumber\\
\xi_K:=&\max\{\xi_K^\rtime,\xi_K^\rspace\},
\nonumber
\end{align}
with $\gamma$ as defined in \eqref{eq:gamma}.
The dimensionless coefficients $\xi_K^\bullet$ measure the impact of the choices made in a concrete implementation of the DG scheme -- in terms of the numerical flux parameters and the element shapes -- on the convergence bounds of Theorem~\ref{thm:Convergence}.
If $\rho,G\in C^0\OO$ and
\begin{equation}\label{eq:FluxChoice}
\alpha=\beta^{-1}=(\rho c)^{-1}
=Gc{=\sqrt{\frac G\rho}}, \qquad \delta=\frac{c^2\vartheta^2}{1+c^2\vartheta^2},
\end{equation}
then $\xi_K^{\mathrm{time}}=3$ 
while $\xi_K^{\mathrm{space}}$ only depends on the maximal slope of the space-like faces of $K$ and on $c$. 
If all faces of $K$ are either aligned with or perpendicular to the time axis, i.e.\ $n_t\in\{0,1\}$, then $\xi_K^\rspace=3$ as well.

We measure mesh regularity by fixing a dimensionless parameter $\eta>0$ such that
\begin{equation}\label{eq:ShapeReg}
r_{K,c}\Big(|\deK^\rspace|\N{c}_{C^0(K)}^{-1}+|\deK^\rtime|\Big)\le\eta|K|\qquad \forall K\in\calT_h.
\end{equation}
Remark~\ref{rem:EtaCuboid} gives more details about $\eta$ in the case of cuboidal elements.

The next theorem gives error bounds for the quasi-Trefftz space--time DG scheme \eqref{eq:DG}.
In particular, the $\Tnorm{\cdot}\DG$ norm of the Galerkin error converges in the mesh size (measured by $r_{K,c}$ as in \eqref{eq:rK}) with rate $p+1/2$.

\begin{theorem}\label{thm:Convergence}
Let $u\in C^1(\conj Q)\cap C^{p+2}\Th$, for some $p\in\IN_0$, be solution of the IBVP \eqref{eq:IBVP_U} with
$\rho\in C^0\OO\cap C^{\max\{p-1,1\}}\Th$ and $G\in C^0\OO\cap C^{\max\{p-2,0\}}\Th$  
and $\vs=(\partial_t u,-\frac1\rho \nabla u)$ be the corresponding solution to the IBVP \eqref{eq:IBVP}.
Let $\vsh$ be the solution of the DG formulation \eqref{eq:DG} with the discrete space 
$\bVp\Th=\QW^p\Th$. 
Assume that each mesh element $K$ is star-shaped with respect to its centre point $(\bx_K,t_K)$.
Then, 
\begin{align}\label{eq:Convergence}
&\frac12\N{c^{-1}(v-v\hp)}_{L^2(\FT)}+\frac12\N{\bsigma-\bsigma\hp}_{L^2(\FT)^n}\\
&\le \Tnorm{\vs-\vsh}\DG
\nonumber\\\nonumber
&\le(1+C_{\!\calA})\frac{(n+1)^p}{p!}
\bigg(\sum_{K\in\calT_h}|K|
\bigg[\Big(\eta\N{c}_{C^0(K)}\,\xi_K+\mu_{K-}r_{K,c}\Big)\frac{(n+1)^3 r_{K,c}}{\rho_K(p+1)^2}
\\&\hspace{30mm}
+\mu_{K+}\N{c^2\rho^{-1}}_{C^0(K)}\Big(2(n+1) +\frac{(n+1)^2}{(p+1)^2}\abs{\rho}_{C^1(K)}^2\rho_K^{-2}r_{K,c}^2 \Big)\bigg]
 r_{K,c}^{2p} \abs{u}_{C^{p+2}_c(K)}^2\bigg)^{1/2}.
 \nonumber
\end{align}
The values of $C_{\!\calA}$, $\mu_{K\pm}$, $\rho_K$, $\xi_K$ and $\eta$ are defined in equations \eqref{eq:Continuity}, \eqref{eq:mupm}, \eqref{eq:xi} and \eqref{eq:ShapeReg}, respectively.
If the numerical flux parameters are set according to \eqref{eq:FluxChoice} and all space-like faces are perpendicular to the time axis, then $\xi_K=3$.

If moreover the volume penalty parameters are chosen as
\begin{equation}\label{eq:muChoice}
\mu_1|_K=\mu_2|_K = r_{K,c} \N{c}_{C^0(K)}^{-1}  \qquad \forall K\in\calT_h,
\end{equation}
then the right-hand side of the estimate \eqref{eq:Convergence} can be bounded by
\begin{align}\label{eq:Convergence2}
(1+C_{\!\calA})\frac{|Q|^{1/2}(n+1)^{p+3/2}}{p!}
\sup_{K\in\calT_h}\bigg(\frac{\N{c}_{C^0(K)}^{1/2}}{\rho_K^{1/2}}
\bigg[\frac{\eta\,\xi_K}{(p+1)^2}+2+\frac{\abs{\rho}_{C^1(K)}^2r_{K,c}^2}{(p+1)^2\rho_K^2}\bigg]^{1/2}
r_{K,c}^{p+1/2} \abs{u}_{C^{p+2}_c(K)}\bigg).
\end{align}
\end{theorem}
\begin{proof}
We first recall that the $\Tnorm{\cdot}\DGp$ norm in \eqref{eq:DGnorm} differs from the analogous one in \cite{MoPe18} only by the presence of the volume terms, the coefficients $\rho,G$ and the incorporation of $c$ in $\tht$ on $\FR$. Then, the first step in the proof of \cite[Thm.~2]{MoPe18} allows to control the $\Tnorm{\cdot}\DGp$ norm of any (possibly discontinuous) 
$\wt\in C^1\Th$ by local norms:
\begin{align*}
&\Tnorm{(w,\btau)}\DGp^2\\
&\le \sum_{K\in\calT_h}\!\bigg[
\xi_K^\rspace
\Big(\N{G^{1/2} w}^2_{L^2(\deKspa)}\!+\N{\rho^{1/2}\btau}^2_{L^2(\deKspa)^{n}}\Big)
\\&
\hspace{13mm}
+\xi_K^\rtime\Big(\N{c^{1/2}G^{1/2}w}^2_{L^2(\deKtime)}
+\N{c^{1/2}\rho^{1/2}\btau}^2_{L^2(\deKtime)^{n}}\Big)
\\&
\hspace{13mm}
+\Norm{\mu_1^{-1/2}G^{1/2} w}^2_{L^2(K)}
+\Norm{\mu_2^{-1/2}\rho^{1/2}\btau}^2_{L^2(K)^n}\\&
\hspace{13mm}
+\Norm{\mu_1^{1/2}G^{-1/2} (\nabla\cdot\btau+ G\partial_t w)}^2_{L^2(K)}
+\Norm{\mu_2^{1/2}\rho^{-1/2}(\nabla w+\rho\partial_t\btau)}^2_{L^2(K)^n}\bigg]
\\&
\le \sum_{K\in\calT_h}\bigg[
\Big(\xi_K^\rspace|\deKspa|+\xi_K^\rtime\N{c}_{C^0(K)}|\deKtime|+
\mu_{K-}|K|\Big)
\Big(\N{G^{1/2} w}^2_{C^0(K)}+ n\N{\rho^{1/2}\btau}^2_{C^0(K)^{n}}\Big)
\\&\hspace{3mm}+|K|\mu_{K+}
\Big(\N{G^{1/2} \partial_t w}^2_{C^0(K)}+n\N{\rho^{-1/2}\nabla w}^2_{C^0(K)^n}
+n\N{\rho^{1/2}\partial_t \btau}^2_{C^0(K)^n}+\N{G^{-1/2} \nabla\cdot\btau}^2_{C^0(K)}\Big)
\bigg],
\end{align*}
recalling that $\rho G=c^{-2}$.
For $\wt=(\partial_t z,-\frac1\rho\nabla z)$ we control the norms at the right-hand side with the weighted $C^m_c(K)$ norms of $z$ defined in \eqref{eq:Cm}:
\begin{align*}
&\N{G^{1/2}w}^2_{C^0(K)}+ n\N{\rho^{1/2}\btau}^2_{C^0(K)^{n}}
\le \rho_K^{-1}\Big(\N{c^{-1}\partial_t z}^2_{C^0(K)}+ n\N{\nabla z}^2_{C^0(K)^{n}}\Big)
\le\rho_K^{-1}(n+1)\abs{z}_{C^1_c(K)}^2,
\\
&\N{G^{1/2}\partial_t w}^2_{C^0(K)}+n\N{\rho^{-1/2}\nabla w}^2_{C^0(K)^n}
+n\N{\rho^{1/2}\partial_t \btau}^2_{C^0(K)^n}+\N{G^{-1/2}\nabla\cdot\btau}^2_{C^0(K)}
\\&\le
\N{G^{1/2}\partial^2_t z}^2_{C^0(K)}+2n\N{\rho^{-1/2}\nabla u_t}^2_{C^0(K)^n}
+\N{G^{-1/2}\frac1\rho\Delta z}^2_{C^0(K)}
+\N{G^{-1/2}\frac1{\rho^2}\nabla \rho\cdot\nabla z}^2_{C^0(K)}
\\&\le
\N{c^2 \rho^{-1}}_{C^0(K)}\Big(
\N{c^{-2}\partial^2_t z}^2_{C^0(K)}
+2n\N{c^{-1}\nabla u_t}^2_{C^0(K)^n}
+\N{\Delta z}^2_{C^0(K)}
+\abs{\rho}_{C^1(K)}^2\rho_K^{-2}
\abs{z}^2_{C^1_c(K)}\Big)
\\&\le
\N{c^2 \rho^{-1}}_{C^0(K)}\Big(2(n+1)\abs{z}^2_{C^2_c(K)}
+\abs{\rho}_{C^1(K)}^2\rho_K^{-2}\abs{z}^2_{C^1_c(K)}\Big).
\end{align*}
Now we use the quasi-optimality \eqref{eq:QuasiOpt}, the relation between the discrete spaces $\QU^{p+1}(K)$ and $\QW^p\Th$, the assumption $\vs=(\partial_t u,-\frac1\rho \nabla u)$, the local best-approximation bound \eqref{eq:Approx}, the definition of $\eta$ \eqref{eq:ShapeReg}:
\begin{align*}
&\frac1{(1+C_{\!\calA})^2}\Tnorm{\vs-\vsh}\DG^2\\
&\overset{\eqref{eq:QuasiOpt}}\le\inf_{\wth\in\QW^p\Th}\Tnorm{\vs-\wth}\DGp^2\\
&\overset{\eqref{eq:QWglobal}}=
\inf_{u\hp\in\prod_{K\in\calT_h}\QU^{p+1}(K)}\Tnorm{\big(\partial_t(u-u\hp),-\frac1\rho \nabla(u-u\hp)\big)}\DGp^2
\\&
\le\inf_{u\hp\in\prod_{K\in\calT_h}\QU^{p+1}(K)}\sum_{K\in\calT_h}
\bigg[\Big(\xi_K^\rspace|\deKspa|+\xi_K^\rtime\N{c}_{C^0(K)}|\deKtime|+\mu_{K-} |K|\Big)
\frac{n+1}{\rho_K} \abs{u-u\hp}_{C^1_c(K)}^2
\\&\hspace{30mm}
+|K|\mu_{K+}\N{c^2\rho^{-1}}_{C^0(K)} \Big(2(n+1) \abs{u-u\hp}_{C^2_c(K)}^2
+\abs{\rho}_{C^1(K)}^2\rho_K^{-2}\abs{u-u\hp}^2_{C^1_c(K)}\Big)\bigg]
\\
&\overset{\eqref{eq:Approx}}\le 
\frac{(n+1)^{2p}}{(p!)^2} 
\sum_{K\in\calT_h}
\bigg[\Big(\xi_K^\rspace|\deKspa|+\xi_K^\rtime\N{c}_{C^0(K)}|\deKtime|+\mu_{K-}|K|\Big)\frac{(n+1)^{3}r_{K,c}^2}{\rho_K(p+1)^2}
\\&\hspace{30mm}
+|K|\mu_{K+}\N{c^2\rho^{-1}}_{C^0(K)}\Big(2(n+1) +\frac{(n+1)^2}{(p+1)^2}\abs{\rho}_{C^1(K)}^2\rho_K^{-2}r_{K,c}^2 \Big)\bigg]
 r_{K,c}^{2p} \abs{u}_{C^{p+2}_c(K)}^2\\
&\overset{\eqref{eq:ShapeReg}}\le 
\frac{(n+1)^{2p}}{(p!)^2} 
\sum_{K\in\calT_h}|K|\bigg[
    \Big(\eta\N{c}_{C^0(K)}\xi_K+\mu_{K-}r_{K,c}\Big)\frac{(n+1)^3 r_{K,c}}{\rho_K(p+1)^2}
\\&\hspace{30mm}
+\mu_{K+}\N{c^2\rho^{-1}}_{C^0(K)}\Big(2(n+1) +\frac{(n+1)^2}{(p+1)^2}\abs{\rho}_{C^1(K)}^2\rho_K^{-2}r_{K,c}^2 \Big)\bigg]
 r_{K,c}^{2p} \abs{u}_{C^{p+2}_c(K)}^2.
\end{align*}
Under assumption \eqref{eq:muChoice} the last expression is bounded by
\begin{align*}
&\frac{(n+1)^{2p}}{(p!)^2} 
\sum_{K\in\calT_h}|K|\frac{\N{c}_{C^0(K)}}{\rho_K}\bigg[({\eta\,\xi_K+1})\frac{(n+1)^3}{(p+1)^2}
+2(n+1)+\frac{(n+1)^2}{(p+1)^2}\frac{\abs{\rho}_{C^1(K)}^2r_{K,c}^2}{\rho_K^2}\bigg]
 r_{K,c}^{2p+1} \abs{u}_{C^{p+2}_c(K)}^2
 \\
&\le\frac{|Q|(n+1)^{2p+3}}{(p!)^2} 
\sup_{K\in\calT_h}\bigg(\frac{\N{c}_{C^0(K)}}{\rho_K}
\bigg[\frac{\eta\,\xi_K}{(p+1)^2}+2+\frac{\abs{\rho}_{C^1(K)}^2r_{K,c}^2}{(p+1)^2\rho_K^2}\bigg]
 r_{K,c}^{2p+1} \abs{u}_{C^{p+2}_c(K)}^2\bigg),
\end{align*}
where we used that $\sum_{K\in\calT_h}|K|=|Q|$.
Estimates \eqref{eq:Convergence} and \eqref{eq:Convergence2} follow taking square roots.
\end{proof}

Theorem~\ref{thm:Convergence} immediately extends to quasi-Trefftz discrete spaces $\bVp\Th$ with different polynomial degrees in each mesh elements.

\begin{rem}[Relevance of $\mu_2$]\label{rem:mu2}
When the discrete space is taken as $\bV\hp\Th=\QW^p\Th$, the choice of the parameter $\mu_2$ is immaterial because of the vanishing of the term in $\calA(\cdot;\cdot)$ it multiplies.
On the other hand, the assertion of Theorem~\ref{thm:Convergence} holds also for the space $\bV\hp\Th=\QT^p\Th$ (defined below in \S\ref{s:fobasis}), and in this case $\mu_2$ needs to be chosen as in \eqref{eq:muChoice}.
\end{rem}

\begin{rem}[Error analysis in $C^m$ and Sobolev spaces]\label{rem:CvsH}
In Proposition~\ref{prop:Approx} and Theorem~\ref{thm:Convergence} we study approximation properties of quasi-Trefftz functions and convergence rates of the corresponding DG scheme using $C^m(K)$-type spaces.
To allow the treatment of less regular solutions, a study using classical Sobolev norms $H^m(K)$ (possibly weighted with the wavespeed similarly to \eqref{eq:Cm} above and \cite[eq.~(37)]{MoPe18}) is needed.
While the bulk of the proof of Theorem~\ref{thm:Convergence} immediately extends to this case, the best-approximation bound of Proposition~\ref{prop:Approx}, which is the key ingredient, does not.
Indeed, the discrete spaces $\QU^p(K)$ and $\QW^p\Th$ are defined from pointwise expansions of the solutions, while typical approximation estimates in Sobolev norms (such as Bramble--Hilbert theorem) require some integral averaging over a subset of the element~$K$.
This technique is not helpful in our setting: ``averaged Taylor polynomials'' \cite[p.~421]{MoPe18} of solutions of variable-coefficient PDEs are not quasi-Trefftz polynomials.

Moreover, the technique used in \cite[\S6]{MoPe18} to prove Trefftz best-approximation estimates in Sobolev spaces for piecewise-constant coefficients relies on an affine transformation of $K$ whose pull-back transforms the wave equation into its copy with unit speed (see the proof of
\cite[Corollary~3]{MoPe18}).
The element $K$ is assumed to be star-shaped with respect to the ellipsoid that is the counterimage of a ball under this transformation; the parameters defining this ellipsoid have an important role in the approximation bounds.
In the setting considered here, instead, the corresponding transformation of $K$ is not affine, the set obtained is not an ellipsoid but a more complicated shape, and the pull-back would not preserve polynomials.
All these reasons prevent a straightforward extension of the theory of \cite[\S6.1.2, 6.2.3]{MoPe18} to the smooth wavespeed case.

On the other hand, the regularity theory for the wave equation shows that point singularities generated by domain corners and by discontinuities in the derivatives of the material coefficients do not propagate in space, \cite{LuTu15,MullerPhD,MuSc15}.
Thus we expect that smooth initial and boundary data give piecewise-smooth solutions, to which the error analysis of this section applies.
A complete regularity theory for piecewise-smooth domains and coefficients is still missing, to our knowledge.

{In \S\ref{sec:numl} we study numerically an IBVP with initial datum $u_0\in H^2\OO\setminus C^2\OO$, approximated with $\bVp\Th=\QW^0\Th$.
We observe linear convergence in the final-time error (see also Remark~\ref{rem:Optimal}) suggesting that the result of Theorem~\ref{thm:Convergence} holds also with $u\in C^p\Th$ replaced by $u\in H^p\Th$.
}
\end{rem}

\begin{rem}[Rate optimality {in DG and final-time norms}]\label{rem:Optimal}
Despite the use of $C^m$ spaces in the analysis, the convergence rates {in $\Tnorm{\cdot}\DG$ norm} for a given polynomial degree are optimal: compare the term $r_{K,c}^{p+1/2}$ in \eqref{eq:Convergence2} and the term $h_K^{m_K+1/2}$ in \cite[eq.~(46)]{MoPe18} (with $h_K\approx2r_{K,c}$, $m_K=p$).
On the other hand, the solution regularity required is stronger.

{On the contrary, the convergence rates at final time (i.e.\ in $\N{\cdot}_{L^2(\FT)}$ norm) proved in Theorem~\ref{thm:Convergence} are suboptimal by half power of the mesh size.
In Figures~\ref{fig:tent1} (left) and \ref{fig:tent2} (left) we observe numerically that this is only a shortcoming of the proof: 
on quasi-uniform meshes (either Cartesian-product or tent-pitched),
$\N{c^{-1}(v-v\hp)}_{L^2(\FT)}+\N{\bsigma-\bsigma\hp}_{L^2(\FT)^n}= \calO(r_{K,c}^{p+1})$ as opposed to the rate $\calO(r_{K,c}^{p+1/2})$ expected from \eqref{eq:Convergence} and \eqref{eq:Convergence2}.
In order to prove optimal-rate bounds, a suitable duality argument needs to be devised; in the setting of space--time Trefftz-DG schemes, a special duality argument for error bounds on mesh-independent norms has been devised in \cite[\S5.4]{MoPe18} but it does not immediately allow for $\N{\cdot}_{L^2(\FT)}$ norms.}

A more sophisticated treatment of corner singularity in polygonal domains using weighted Sobolev spaces is done in \cite{BMPS20}.
\end{rem}

\begin{rem}[Value of $\eta$ for a cuboid]\label{rem:EtaCuboid}
Condition \eqref{eq:ShapeReg} is a condition on the ``chunkiness'' of the mesh elements.
For instance, if all elements are translated copies of the Cartesian product $(0,L_\bx)^n\times(0,L_t)$ between a space segment/square/cube and a time interval, the centres $(\bx_K,t_K)$ are their barycentres, and $c$ is constant in $K$, then 
\begin{align*}
|K|=L_\bx^n L_t, \qquad |\deK^\rspace|=2L_\bx^n, \qquad  |\deK^\rtime|=2nL_\bx^{n-1}L_t, \qquad 
r_{K,c}=\frac12\sqrt{nL_\bx^2+c^2L_t^2}
\end{align*}
and $\eta$ can be taken as
$\eta_{\text{cuboid}}:=2n^{\frac32}(\frac{L_\bx}{cL_t}+\frac{cL_t}{L_\bx})$ since
\begin{align*}
2n^{\frac32}\left(\frac{L_\bx}{cL_t}\!+\!\frac{cL_t}{L_\bx}\right) \!\ge 
\frac{n^{\frac32}(L_\bx+cL_t)^2}{cL_\bx L_t}\!\ge
\frac{\frac12\sqrt{nL_\bx^2+c^2L_t^2}(2L_\bx^n+2ncL_\bx^{n-1}L_t)}{cL_\bx^n L_t}
=\frac{r_{K,c}(|\deK^\rspace|+c|\deK^\rtime|)}{c|K|}.
\end{align*}
Thus $\eta$ is minimal for ``isotropic'' cuboids with $L_\bx=cL_t$.
\end{rem}

\subsection{Basis functions}\label{s:Basis}

Here, we describe the construction of basis functions for $\QU^p(K)$; those for $\QW^p\Th$ can then obtained simply by taking their appropriate partial derivatives.

To construct a basis of the quasi-Trefftz space $\QU^p(K)$ space, we first choose two polynomial bases in the space variable only:
$$
\left\{\widehat b_J\right\}_{J=1,\dots,\binom{p+n}n}
\text{ basis for }\mathbb P^p(\mathbb R^n), \quad \text{ and } \quad
\left\{\widetilde b_J\right\}_{J=1,\dots,\binom{p-1+n}n}
\text{ basis for }\mathbb P^{p-1}(\mathbb R^{n}).
$$
Their total cardinality is
$$
N(n,p):=\binom{p+n}n+\binom{p-1+n}n =\frac{(p-1+n)!\:(2p+n)}{n!\:p!}.
$$
We define the following $N(n,p)$ elements of $\QU^p(K)$:
\begin{align}\label{eq:bJ}
\left\{\refb\in\QU^p(K) \mid 
\begin{array}{l}
\refb(\cdot,t_K)=\widehat b_J \;\tand \;\partial_t \refb(\cdot,t_K)=0 \tfor J\leq \binom{p+n}n,
\\ 
\refb(\cdot,t_K)=0\; \tand \;\partial_t \refb(\cdot,t_K)=\widetilde b_{J-\binom{p+n}n}
\tfor \binom{p+n}n < J
\end{array}
\right\}_{J=1,\dots,N(n,p).}
\end{align}
In the rest of this section we show that the elements $b_J$ are well-defined, that they can be computed with a simple algorithm, and that they constitute a basis of $\QU^p(K)$.

Constructing the $J$-th basis function of \eqref{eq:bJ} is equivalent to finding the set of coefficients $(a_\mk=a_{\mk_\bx,k_t})_{\mk\in\IN_0^{n+1},|\mk|\le p}$ such that the polynomial 
$$
\refb(\bx,t):=\sum_{\mk\in\IN_0^{n+1},|\mk|\leq p} a_\mk (\bx-\bx_K)^{\mk_\bx} (t-t_K)^{k_t}
,\qquad a_\mk=\frac1{\mk!}D^\mk \refb(\bx_K,t_K),
$$
fulfills the following system of equations
\begin{align}\label{eq:systembasis}
\begin{cases}
D^\mi  \square_{\rho,G} \refb(\bx_K,t_K)=0 
&\qquad  J=1,\ldots,N(n,p),\; \mi \in\IN_0^{n+1},\; |\mi|\le p-2,\\
\refb(\cdot,t_K)=\widehat b_J \tand  \partial_t \refb(\cdot,t_K)=0
&\qquad J=1,\ldots,\binom{p+n}n,\\ 
\refb(\cdot,t_K)=0 \tand  \partial_t \refb(\cdot,t_K)=\widetilde b_J
&\qquad J=\binom{p+n}n +1,\ldots,N(n,p).
\end{cases}
\end{align}
The second and the third sets of equations assign the values of all $a_{\mk_\bx,0}$ and $a_{\mk_\bx,1}$.
From~\eqref{eq:DiSquare} and since
$D^\mi ((\bx-\bx_K)^{\mk_\bx} (t-t_K)^{k_t})\ne0$ at $(\bx_K,t_K)$ if and only if $\mi =\mk$, the first equation in \eqref{eq:systembasis} -- corresponding to the $\big(\partial_\bx^{\mi _\bx}\partial_t^{i_t}
 \square_{\rho,G} \refb\big) (\bx_K,t_K)$ term -- reads
\begin{align}\label{eq:systembasis2}
\sum_{l=1}^n\sum_{\mj_\bx\le \mi_\bx+\be_l} 
 (j_{x_l}+1)(\mi_\bx+\be_l)!i_t! \zeta_{\mi_\bx+\be_l-\mj_\bx}  a_{\mj_\bx+\be_l,i_t}
- \sum_{\mj_\bx\le \mi_\bx}  \mi_\bx!  (i_t+2)! g_{\mi_\bx-\mj_\bx} a_{\mj_\bx,i_t+2}
=0.
\end{align}
This equation is used to compute the element with $\mj_\bx=\mi_\bx$ in the last sum, since $g_\bzero=G(\bx_K)>0$ its coefficient is non-zero:
\begin{equation}
a_{\mi_\bx,i_t+2}
=
-\sum_{\mj_\bx<\mi_\bx} \frac{g_{\mi_\bx-\mj_\bx}}{g_\bzero} a_{\mj_\bx,i_t+2}
+
\sum_{l=1}^n\sum_{\mj_\bx\le \mi_\bx+\be_l} 
\frac{(i_{x_l}+1)(j_{x_l}+1)\,\zeta_{\mi_\bx+\be_l-\mj_\bx} }{(i_t+2)(i_t+1)\,g_\bzero} a_{\mj_\bx+\be_l,i_t},
\label{eq:BasisFormula}
\end{equation}
where the strict inequality $\mj<\mi$ in the summation means that $\mj\le\mi$ and $\mj\ne\mi$.
We compute these values iteratively.
We need to make sure that at every step of the iterations we use values already computed.

We note that the parameter functions $\rho$ and \rv{$G$} enters the computation of $\refb$ only through their Taylor coefficients at $(\bx_K,t_K)$, i.e.\ the $\zeta_\mi$ and $g_\mi $ defined in \eqref{eq:Gg}.

One could also write the equations \eqref{eq:systembasis2} as a linear system, where the right-hand side vector is given by the known values of $a_{\mk_\bx,0}$, $a_{\mk_\bx,1}$, and solve it.
However the recursive implementation appears simpler.

The next two sections describe in detail the iterative algorithm to compute the coefficients in the cases $n=1$ and $n>1$, respectively.
Possible choices of the space-only bases $\left\{\widehat b_J\right\}$, $\left\{\widetilde b_J\right\}$ are described in the numerical results section \ref{sec:num}.

\begin{prop}\label{prop:Basis}
The polynomials $\{\refb\}_{J=1,\ldots,N(n,p)}$ defined by \eqref{eq:systembasis} (and computable with the algorithms of \S\ref{s:Basis1D}--\ref{s:BasisND}) constitute a basis for the space $\QU^p(K)$.
\end{prop}
\begin{proof}
The algorithms described in \S\ref{s:Basis1D}--\ref{s:BasisND} show that \eqref{eq:systembasis} uniquely defines the $N(n,p)$ polynomials $\refb$.
The first set of conditions in \eqref{eq:systembasis} ensures that $\refb\in\QU^p(K)$.
The traces on $\{t=t_K\}$ of these polynomials ensure that they are linearly independent: if $\sum_{J=1}^{N(n,p)}c_J\refb=0$ then, by \eqref{eq:systembasis}
$\sum_{J=1}^{\binom{p+n}n}c_J\widehat b_J=0$ and 
$\sum_{J=\binom{p+n}n+1}^{N(n,p)}c_J\widetilde b_J=0$, so $c_J=0$ for all $J$ because $\left\{\widehat b_J\right\}$ and $\left\{\widetilde b_J\right\}$ are assumed to be linearly independent.

Relation \eqref{eq:BasisFormula} holds not only for the elements of the basis, but for the monomial expansion of any element of $\QU^p(K)$.
Then Algorithms \ref{algo:Basis1D} and \ref{algo:BasisND} (which simply apply \eqref{eq:BasisFormula} following a precise ordering of the multi-indices $\mi$) show that the Taylor coefficients in $(\bx_K,t_K)$ of any  $f\in\QU^p(K)$ are uniquely determined by the coefficients $a_{\mi_\bx,0}$ and $a_{\mi_\bx,1}$, and hence by $f(\cdot,t_K)$ and $\partial_t f(\cdot,t_K)$.
Since $f(\cdot,t_K)\in \IP^p(\IR^n)$ and $\partial_t f(\cdot,t_K)\in \IP^{p-1}(\IR^n)$, $f$ is linear combination of the $\{\refb\}$.
Thus this set spans $\QU^p(K)$ and we conclude the proof.
\end{proof}
From the proposition it follows that the conditions in the definition \eqref{eq:QU} of $\QU^p(K)$ are linearly independent:
$$
\dim\big(\IP^p(K)\big)-\#\{\mi \in\IN_0^{n+1}\mid |\mi |\leq p-2\}=\begin{pmatrix}p+n+1\\ n+1\end{pmatrix}-\begin{pmatrix}p+n-1\\ n+1\end{pmatrix} 
=N(n,p)=\dim\big(\QU^p(K)\big).$$

\begin{rem}\label{rem:Dimension}
Proposition \ref{prop:Basis} implies that 
$$
\dim\big(\QU^p(K)\big)=N(n,p)
=\begin{cases}
2p+1 & n=1\\ (p+1)^2 & n=2\\ \frac16(p+2)(p+1)(2p+3) & n=3
\end{cases}
\quad=\calO_{p\to\infty}(p^n).
$$
For large polynomial degrees $p$, the dimension of the quasi-Trefftz space is much smaller than the dimension of the full space--time polynomial space of the same degree: $\dim(\IP^p(K))=\binom{p+n+1}{n+1}=\calO_{p\to\infty}(p^{n+1})$ (recall that $K\subset\IR^{n+1}$).

Proposition \ref{prop:Approx} shows that both spaces $\QU^p(K)$ and $\IP^p(K)$ have comparable $h$-approximation properties when the function to be approximated is solution of the (variable-coefficient) wave equation.
This is the main advantage offered by Trefftz and quasi-Trefftz schemes: same approximation power for much fewer degrees of freedom.

The dimension of $\QU^p(K)$ is equal to the dimension of the Trefftz space of the same degree for the constant-coefficient wave equation \cite[\S6.2.1]{MoPe18}, for the
Laplace equation (i.e.\ the space of harmonic polynomials in $\IR^{n+1}$ of degree $\le p$) and for the Helmholtz equation \cite[\S3]{TrefftzSurvey} (the space of circular/spherical and plane waves in $\IR^{n+1}$ with the same approximation order).

Similarly, $\dim(\QW^p(K))=\dim(\QU^{p+1}(K))-1=\calO_{p\to\infty}(p^n)$.
\end{rem}

\subsubsection{The construction of the basis functions: the case \texorpdfstring{$n=1$}{n=1}}\label{s:Basis1D}

We first describe the one-dimensional case for the sake of clarity.

For each basis function $\refb$ we need to compute the coefficients $a_{ k _x, k _t}$, $ k _x, k _t\in\IN_0$, $ k _x+ k _t\le p$; they are represented by the dots constituting a triangular shape in the plan of indices $(k_x,k_t)\in\mathbb N_0^2$, as represented on Figure~\ref{fig:IndexTriangle}.
We recall that the coefficients $\{a_{ i _x,0},0\le i _x\le p\}$, and $\{a_{ i _x,1}, 0\le i _x\le p-1\}$, represented by the shaded area in the figure, are known from the choice of the ``Cauchy data'' bases $\left\{\widehat b_J\right\}$, $\left\{\widetilde b_J\right\}$.
Formula \eqref{eq:BasisFormula} allows to compute $a_{ i _x, i _t+2}$ from similar coefficients $a_{j_x, i _t+2}$ with $j_x< i _x$ and from 
coefficients $a_{ j _x+1, i _t}$ for $j_x\leq i_x+1$.
This suggests to proceed ``diagonally'': i.e.\ to compute the values $a_{ k _x, k _t}$ for $ k _x+ k _t=\ell$ increasingly from $\ell=2$ to $\ell=p$.
On each of these diagonals (in gray in the figure) we compute the values of $a_{ k _x,\ell- k _x}$ for decreasing $ k _x$.
This means that we perform a double loop: in terms of the graphical representation, the external loop moves away from the origin ($\nearrow$) and the inner loop moves from the $ k _x$ axis to the $ k _t$ axis ($\nwarrow$).
This procedure is described in Algorithm \ref{algo:Basis1D}.

\begin{figure}[htb]\centering
\begin{tikzpicture}[scale=.8]
\fill[yellow!50!white](-.3,-.3)--(-.3,1.3)--(6.3,1.3)--(6.3,-.3)--(-.3,-.3);
\draw[lightgray](2,0)--(0,2);
\draw[lightgray](3,0)--(0,3);
\draw[lightgray](4,0)--(0,4);
\draw[lightgray](5,0)--(0,5);
\draw[->](0,0)--(7,0); 
\draw[->](0,0)--(0,7); 
\draw[thick](0,6)--(6.3,-.3);
\draw(-.5,6)node{$p$};
\draw[thick](3,0)circle(.1);\draw[thick](4,0)circle(.1);
\draw[thick](5,0)circle(.1);\draw[thick](6,0)circle(.1);
\draw[thick](2,1)circle(.1);\draw[thick](3,1)circle(.1);
\draw[thick](4,1)circle(.1);\draw[thick](5,1)circle(.1);
\draw[thick](1,2)circle(.1);\draw[thick](2,2)circle(.1);
\draw[thick](3,2)circle(.1);
\draw[thick](0,0)circle(.1);\draw[thick](1,0)circle(.1);\draw[thick](0,1)circle(.1);
\draw[thick](2,0)circle(.1);\draw[thick](1,1)circle(.1);\draw[thick](0,2)circle(.1);
\draw[thick,fill](4,2)circle(.15);
\draw(5,2.3)node{$a_{ i _x+2, i _t}$};
\draw[thick,fill](3,2)circle(.15);
\draw[thick,fill](2,2)circle(.15);
\draw[thick,fill](1,2)circle(.15);
\draw[thick](0,3)circle(.1);\draw[thick](1,3)circle(.1);
\draw[thick](2,3)circle(.1);\draw[thick](3,3)circle(.1);
\draw[thick,fill](0,4)circle(.15);
\draw[thick,fill](1,4)circle(.15);
\draw(3.1,4.3)node{$a_{ i _x, i _t+2}$};
\draw[thick,fill](2,4)circle(.15);\draw[ultra thick,red](2,4)circle(.3);
\draw[thick](0,5)circle(.1);\draw[thick](1,5)circle(.1);
\draw[thick](0,6)circle(.1);
\draw(7,-.4)node{$ k _x$};
\draw(-.4,7)node{$ k _t$};
\draw[decorate, decoration={brace,amplitude=5}] (7.6,1.3) -- (7.6,-.3);
\draw(9,.8)node{Cauchy};
\draw(9,.2)node{data};
\draw[dashed](2,0)--(2,4); \draw(2,-.5)node{$ i _x$};
\draw[dashed](0,2)--(4,2); \draw(-.5,2)node{$ i _t$};
\end{tikzpicture}
\caption{Graphical representation of the algorithm used to compute a quasi-Trefftz basis function in the case $n=1$ (and $p=6$), see \S\ref{s:Basis1D}.
The function $\refb$ is defined by the coefficients $a_{ k _x, k _t}$ corresponding to the small circles $\circ$.
The coefficients corresponding to the dots in the shaded area ($ k _t\in\{0,1\}$) are given by the ``Cauchy data'', the second and third set of equations in \eqref{eq:systembasis}.
The first equation in \eqref{eq:systembasis} for $( i _x, i _t)=(2,2)$ relates the seven nodes depicted with a black dot $\bullet$ and is used to compute $a_{ i _x, i _t+2}$ (explicitly with \eqref{eq:BasisFormula}), corresponding to the node surrounded by a red circle \red{\textbigcircle}.
All these coefficients (in the non-shaded region) are computed with formula \eqref{eq:basisrec} in a double loop: first across diagonals $\nearrow$, and then along diagonals $\nwarrow$.}
\label{fig:IndexTriangle}
\end{figure}
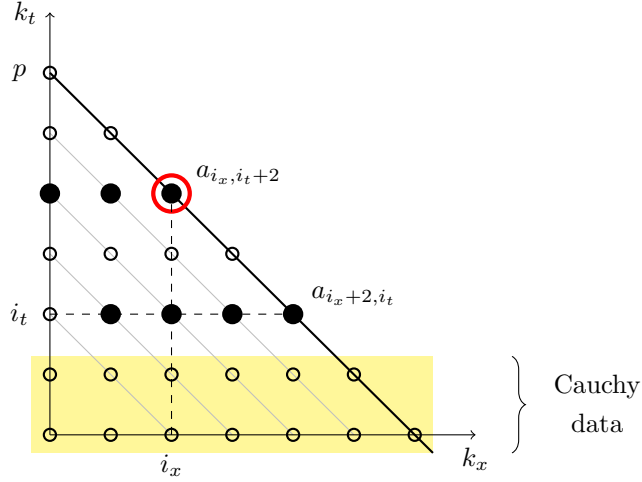

\begin{algorithm}[H]
{\sc Algorithm}\\
\SetAlgoLined
Data: $(g_m)_{m\in\IN_0}$, $(\zeta_m)_{m\in\IN_0}$, $x_K$, $t_K$, $p$.\\
Choose polynomial bases $\left\{\widehat b_J\right\}$, $\left\{\widetilde b_J\right\}$, fixing coefficients $a_{ k _x,0}$, $a_{ k _x,1}$.\\
For each $J=1,\ldots,N(n,p)$ (i.e.\ for each basis function), we construct $\refb$ as follows:\\
\For{$\ell=2$ to $p$\quad (loop across diagonals $\nearrow$)\quad}{
\For{$ i _t=0$ to $\ell-2$\quad (loop along diagonals $\nwarrow$)\quad}{
set $ i _x=\ell- i _t-2$ and compute
\begin{equation}\label{eq:basisrec}
a_{i_x,i_t+2}
=
-\sum_{j_x<i_x} \frac{g_{i_x-j_x}}{g_\bzero} a_{j_x,i_t+2}
+ \sum_{j_x\le i_x+1}
\frac{(i_{x}+1)(j_{x}+1)\,\zeta_{i_x+1-j_x}}{(i_t+2)(i_t+1)\,g_\bzero} 
\;a_{j_x+1,i_t}
\end{equation}
}}
$\displaystyle\refb(x,t)
=\sum_{0< k _x+ k _t\leq p} a_{ k _x, k _t} (x-x_K)^{ k _x} (t-t_K)^{ k _t}$
\vspace{2mm}
\caption{The algorithm for the construction of $\refb$ in the case $n=1$, \S\ref{s:Basis1D}.}
\label{algo:Basis1D}
\end{algorithm}

\subsubsection{The construction of the basis functions: the case \texorpdfstring{$n>1$}{n>1}}\label{s:BasisND}

Algorithm \ref{algo:BasisND} extends Algorithm \ref{algo:Basis1D} to the general case $n>1$.
The main novelty is that for each value of $\ell=|\mi_\bx|+i_t+2$ and of $i_t$ there are several coefficients $a_{\mi_\bx,i_t}$ to be computed, exactly one for each $\mi_\bx\in\IN_0^n$ with $|\mi_\bx|=\ell-2-i_t$, thus a further inner loop over $\mi_\bx$ is needed.
Each coefficient of the innermost loop can be computed independently of the others.

Figure \ref{fig:IndexPyramid} depicts the dependence between these coefficients, represented as integer-coordinate points in the $(\mi_\bx,i_t)$ space for $n=2$.
The general coefficient, indicated by the red diamond, is computed with \eqref{eq:basisrecNd} as linear combination of the coefficients corresponding to the black dots.
The structure of the algorithm ensures that, when a coefficient $a_{\mi_\bx,i_t}$ is computed, all the coefficients needed for the right-hand side of \eqref{eq:basisrecNd} have already been computed.

\newcommand{\DDT}{\draw[dashed,thick]}
\newcommand{\NodeThreeD}[3]{\draw (#1,#2,#3) node[circle,fill,inner sep=2pt] {};
\DDT(#1,#2,0)--(#1,#2,#3);\DDT(0,#2,0)--(#1,#2,0);\DDT(#1,0,0)--(#1,#2,0);
\DDT(#1,0,#3)--(#1,#2,#3);\DDT(0,0,#3)--(#1,0,#3);\DDT(#1,0,0)--(#1,0,#3);
\DDT(0,#2,#3)--(#1,#2,#3);\DDT(0,#2,0)--(0,#2,#3);\DDT(0,0,#3)--(0,#2,#3);}
\begin{figure}[htb]\centering
\resizebox{.49\linewidth}{!}{
\begin{tikzpicture}
[rotate around x=-90,rotate around y=0,rotate around z=-30,grid/.style={very thin,gray}]
        \foreach \x in {0,1,...,7}
        \foreach \y in {0,1,...,7}
        \foreach \z in {0,1,...,7}
        {
            \draw[grid,lightgray] (\x,0,0) -- (\x,7,0);
            \draw[grid,lightgray] (0,\y,0) -- (7,\y,0);
            \draw[grid,lightgray] (0,\y,0) -- (0,\y,7);
            \draw[grid,lightgray] (0,0,\z) -- (0,7,\z);
        }
\draw[fill=yellow,opacity=.5] (0,0,0)--(0,7,0)--(7,0,0);
\draw[fill=yellow,opacity=.4] (0,0,1)--(0,6,1)--(6,0,1);
        \draw [->] (0,0) -- (8,0,0) node [right] {$k_{x_1}$};
        \draw [->] (0,0) -- (0,8,0) node [above] {$k_{x_2}$};
        \draw [->] (0,0) -- (0,0,8) node [below left] {$k_t$};
\draw(0,0)node[left]{$\bzero$};
\draw[gray] (4,0,3)--(0,4,3);
        \draw[fill=blue,opacity=.5] (0,7,0)--(0,0,7)--(7,0,0);
\NodeThreeD313;
        \draw[] (3,1,3) node[color=red,diamond,fill] {};
\NodeThreeD{5}{1}{1};
\NodeThreeD331;
        \draw[] (0,0,3) node[circle,fill,inner sep=2pt] {};
        \draw[] (1,0,3) node[circle,fill,inner sep=2pt] {};
        \draw[] (2,0,3) node[circle,fill,inner sep=2pt] {};
        \draw[] (3,0,3) node[circle,fill,inner sep=2pt] {};
        \draw[] (0,1,3) node[circle,fill,inner sep=2pt] {};
        \draw[] (1,1,3) node[circle,fill,inner sep=2pt] {};
        \draw[] (2,1,3) node[circle,fill,inner sep=2pt] {};
        \draw[] (0,2,1) node[circle,fill,inner sep=2pt] {};
        \draw[] (0,3,1) node[circle,fill,inner sep=2pt] {};
        \draw[] (1,2,1) node[circle,fill,inner sep=2pt] {};
        \draw[] (1,3,1) node[circle,fill,inner sep=2pt] {};
        \draw[] (2,0,1) node[circle,fill,inner sep=2pt] {};
        \draw[] (2,1,1) node[circle,fill,inner sep=2pt] {};
        \draw[] (2,2,1) node[circle,fill,inner sep=2pt] {};
        \draw[] (2,3,1) node[circle,fill,inner sep=2pt] {};
        \draw[] (3,0,1) node[circle,fill,inner sep=2pt] {};
        \draw[] (3,1,1) node[circle,fill,inner sep=2pt] {};
        \draw[] (3,2,1) node[circle,fill,inner sep=2pt] {};
        \draw[] (4,0,1) node[circle,fill,inner sep=2pt] {};
        \draw[] (4,1,1) node[circle,fill,inner sep=2pt] {};
        \draw[] (5,0,1) node[circle,fill,inner sep=2pt] {};
        \draw[] (0,1,1) node[circle,fill,inner sep=2pt] {};
        \draw[] (1,0,1) node[circle,fill,inner sep=2pt] {};
        \draw[] (1,1,1) node[circle,fill,inner sep=2pt] {};
\draw (3,0,0) node[below left] {$i_{x_1}$};
\draw (5,0,0) node[below left] {$i_{x_1}+2$};
\draw (7,0,0) node[below left] {$\ell$};
\draw (0,0,1) node[left] {$i_t$};
\draw (0,0,3) node[left] {$i_t+2$};
\draw (0,0,7) node[left] {$\ell$};
\end{tikzpicture}
}
\caption{
Representation of Algorithm \ref{algo:BasisND} to compute the coefficient $a_{\mi_\bx,i_t}$ of a quasi-Trefftz basis function $\refb$ for $n=2$, $p\ge5$, $\ell=7$ and $(\mi_\bx,i_t)=(3,1,1)$.
The coefficient $a_{3,1,3}$, represented by the large red diamond $\red{\blacklozenge}$, is computed with formula \eqref{eq:basisrecNd} from the twenty-six coefficients indicated by the black dots $\bullet$.
The two yellow triangles in the planes $k_t=0$ and $k_t=1$ indicate the coefficients whose values are given by the ``initial conditions'' $\widehat b_J$ and $\widetilde b_J$ in the second and third equation of \eqref{eq:systembasis}.
To ensure that the right-end side of \eqref{eq:basisrecNd} is well-defined for every $(\mi_\bx,i_t)$, 
Algorithm \ref{algo:BasisND} computes the coefficients first looping through triangles parallel to the one depicted in blue (which corresponds to stage $\ell=7$ of the loop), then through horizontal planes, and finally along the horizontal segments determined by the intersection between the two planes.
}
\label{fig:IndexPyramid}
\end{figure}
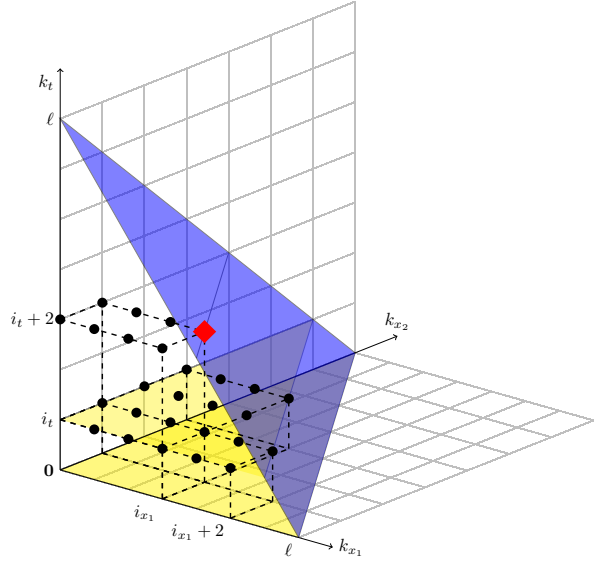

\begin{algorithm}[H]
{\sc Algorithm}\\
\SetAlgoLined

Data: $(g_m)_{m\in\IN_0}$, $\bx_K$, $t_K$, $p$.\\
Choose polynomial bases $\left\{\widehat b_J\right\}$, $\left\{\widetilde b_J\right\}$, fixing coefficients $a_{\mk_\bx,0}$, $a_{\mk_\bx,1}$.\\
For each $J=1,\ldots,N(n,p)$ (i.e.\ for each basis function), we construct $\refb$ as follows:\\
\For{$\ell=2$ to $p$\qquad (loop across $\{|\mi_\bx|+i_t=\ell-2\}$ hyperplanes, $\nearrow$)}{
    \For{$ i _t=0$ to $\ell-2$\qquad (loop across constant-time hyperplanes $\uparrow$)}{
        \For{ $\mi_\bx$ with $|\mi_\bx|=\ell- i _t-2$\qquad
}{
\begin{equation}\label{eq:basisrecNd}
a_{\mi_\bx,i_t+2}
=-\sum_{\mj_\bx<\mi_\bx} \frac{g_{\mi_\bx-\mj_\bx}}{g_\bzero} a_{\mj_\bx,i_t+2}
+ \sum_{l=1}^n\sum_{\mj_\bx\le\mi _\bx+\be_l}
\frac{ (i_{x_l}+1)(j_{x_l}+1)\, \zeta_{\mi_\bx+\be_l-\mj_\bx}}{(i_t+2)(i_t+1)\,g_\bzero} 
a_{\mj_\bx+\be_l,i_t},
\end{equation}
        }
    }
}
$\displaystyle\refb(\bx,t)
=\sum_{|\mk_\bx|+ k _t\leq p} a_{\mk_\bx, k _t} (\bx-\bx_K)^{\mk_\bx} (t-t_K)^{k _t}$
\vspace{2mm}
\caption{The algorithm for the construction of $\refb$ in the general case.}
\label{algo:BasisND}
\end{algorithm}

\subsection{Quasi-Trefftz discrete spaces for the first-order problem}\label{s:fobasis}

The quasi-Trefftz space $\QW^p\Th$ was defined in \eqref{eq:QWglobal} from derivatives of solutions to the second-order wave equation.
Thus it offers high-order approximation properties only for solutions $\vs$ to IBVPs \eqref{eq:IBVP} related to a solution $u$ of the second-order IBVPs \eqref{eq:IBVP_U} by the relation $\vs=(\partial_t u,-\nabla u)$.
We briefly describe a larger discrete space suitable to approximate the general first-order IBVP \eqref{eq:IBVP}.

For $p\in\IN_0$ and any mesh element $K\in\calT_h$, we set
\begin{align}\label{eq:QT}
\mmbox{\begin{aligned}
\QT^p(K):=&
\left\{\wt\in\IP^p(K)^{n+1} \Big|
\begin{array}{l}
D^\mi (\nabla w+\rho\partial_t\btau)(\bx_K,t_K) = \bzero \\
D^\mi (\nabla\cdot\btau+ G\partial_t w)(\bx_K,t_K) = 0 
\end{array}
\quad\forall \mi\in \IN^{n+1}_0, |\mi|\le p-1\right\},
\\
\QT^p\Th:=&\prod_{K\in\calT_h}\QT^p(K).
\end{aligned}}
\end{align}
It is easy to check that if $\rho$ is constant then $\QW^p\Th\subset\QT^p\Th$.
This implies that the convergence results of Theorem~\ref{thm:Convergence} hold also with $\QT^p\Th$ in place of $\QW^p\Th$.

In the case of variable $\rho$, the proof of Proposition~\ref{prop:Approx} immediately extends to the space $\QT^p(K)$, leading to the following result.
\begin{prop}
For $\rho,G\in C^{\max\{p-1,0\}}$, $\vs\in C^{p+1}(K)^{n+1}$ with 
$\nabla v+\rho\partial_t\bsigma = \bzero$ and $\nabla\cdot\bsigma+ G\partial_t v= 0$ in $K$, star-shaped with respect to $(\bx_K,t_K)$,
the (vector) Taylor polynomial $\wt=T^{p+1}_K[\vs]$ belongs to $\QT^p(K)$ and 
$$
\max\bigg\{ \frac{|v-w|_{C^q_c(K)}}{\abs{v}_{C^{p+1}_c(K)}}, 
\frac{|\sigma_1-\tau_1|_{C^q_c(K)}}{\abs{\sigma_1}_{C^{p+1}_c(K)}}, \ldots,
\frac{|\sigma_n-\tau_n|_{C^q_c(K)}}{\abs{\sigma_n}_{C^{p+1}_c(K)}}\bigg\}
 \le\frac{(n+1)^{p+1-q}}{(p+1-q)!} r_{K,c}^{p+1-q} 
 \qquad 0\le q\le p.
$$
The same holds for the unweighted $\abs{\cdot}_{C^q}$ seminorms (with $r_K$ in place of $r_{K,c}$).
\end{prop}
Adapting the proof of Theorem~\ref{thm:Convergence} we obtain the following statement.
\begin{prop}
Let $\vs\in C^0(\conj Q)\cap C^{p+1}\Th$, for some $p\in\IN_0$, be solution of the IBVP \eqref{eq:IBVP} with
$\rho,G\in C^0\OO\cap C^{\max\{p-1,1\}}\Th$.
Let $\vsh$ be the solution of the DG formulation \eqref{eq:DG} with the discrete space 
$\bVp\Th=\QT^p\Th$. 
Assume that each mesh element $K$ is star-shaped with respect to its centre point $(\bx_K,t_K)$.
Then, 
\begin{align*}
&\Tnorm{\vs-\vsh}\DG
\\
&
\le(1+C_{\!\calA})\frac{(n+1)^{p}}{p!} 
\bigg(\sum_{K\in\calT_h}|K|\bigg[
n\frac{(n+1)^2}{(p+1)^2}\Big(\eta\N{c}_{C^0(K)}\xi_K+\mu_{K-}r_{K,c}\Big)r_{K,c}
+2n\mu_{K+}\N{c}_{C^0(K)}^2 
\bigg]\\
&\hspace{40mm}\cdot\Big(\N{G}_{C^0(K)}  \N{v}_{C^{p+1}_c(K)}^2+\N{\rho}_{C^0(K)}  \N{\bsigma}_{C^{p+1}_c(K)^n}^2\Big)
r_{K,c}^{2p}\bigg)^{1/2}.
\end{align*}
If moreover the volume penalty parameters are chosen as in \eqref{eq:muChoice},
then the right-hand side of this estimate can be bounded by
\begin{align*}
&(1+C_{\!\calA})
\frac{|Q|(n+1)^{p+3/2}}{p!}
\\
&\cdot\sup_{K\in\calT_h}\bigg(\N{c}_{C^0(K)} \Big(\frac{\eta\xi_K}{(p+1)^2}+2\Big)
\Big(\N{G}_{C^0(K)}  \N{v}_{C^{p+1}_c(K)}^2+\N{\rho}_{C^0(K)}  \N{\bsigma}_{C^{p+1}_c(K)^n}^2\Big)\bigg)^{1/2}
r_{K,c}^{p+1/2}.
\end{align*}
\end{prop}

Given any basis $\left\{\widetilde b_J(\bx)\right\}_{J=1,\ldots,\binom{p+n}n}$ of $\IP^p(\IR^n)$, we can define a basis for $\QT^p(K)$ as
\begin{align*}
\Bigg\{\bb_{J,l}(\bx,t)\in\QT^p(K) \;\text{ such that }\;
\begin{aligned}
&\bb_{J,0}(\bx,t_K)=\left(\widetilde b_J(\bx), \bzero\right),\\
&\bb_{J,l}(\bx,t_K)=\left(0, \widetilde b_J(\bx)\be_l\right), \; l=1,\ldots, n
\end{aligned}
\Bigg\}_{J=1,\ldots,\binom{p+n}n;\; l=0,\ldots,n}.
\end{align*}

To compute explicitly a basis element $\bb_{J,l}$ from $\widetilde b_J$, we expand it in monomials:
\begin{align*}
\bb_{J,l}(\bx,t)&= \sum_{\substack{\mk\in\IN_0^{n+1},\: |\mk|\le p}}
\ba_{\mk}(\bx-\bx_K)^{\mk_\bx} (t-t_K)^{k_t},\quad  l=0,\dots,n, 
\qquad\tfor \big\{\ba_{\mk}=\ba_\mk(J,l)\big\}_{|\mk|\le p}\in\IR^{n+1}.
\end{align*}
We index the components of the field $\bb(\bx,t)=\bb_{J,l}(\bx,t)$ from $b^0(\bx,t)$ to $b^n(\bx,t)$, and write similarly
$\ba_\mk=(a_\mk^0,\ldots,a_\mk^n)$.
Space--time multi-indices are split as previously in space and time parts $\mk=(\mk_\bx,k_t)$.
Then the conditions corresponding respectively to $D^\mi \Big(\partial_{x_\lambda} b^0+\rho\partial_t b^\ind\Big)(\bx_K,t_K)$ for $\ind$ from $1$ to $n$, namely the components of the vector-valued constraint, and $D^\mi \Big(\sum_{\ind=1}^n \partial_{x_\ind}b^\ind+ G\partial_t b^0\Big)(\bx_K,t_K) $, namely the scalar-valued constraint, in the definition \eqref{eq:QT} of $\QT^p(K)$ can be written in terms of coefficients as
\begin{align*}
0=&
(\mi_\bx+\be_\ind)!i_t!a^0_{(\mi_\bx+\be_\ind,i_t)}+\sum_{\mj_\bx\leq\mi_\bx}\mi_\bx!(i_t+1)!\zeta_{\mi_\bx-\mj_\bx}a^\ind_{(\mj_\bx,i_t+1)},  \qquad  \ind=1,\dots,n,
\\
0=&
\sum_{\ind=1}^n (\mi_\bx+\be_\ind)!i_t!a^\ind_{(\mi_\bx+\be_\ind,i_t)} 
+\sum_{\mj_\bx\leq\mi_\bx} \mi_\bx!(i_t+1)! g_{\mi_\bx-\mj_\bx} a^0_{(\mj_\bx,i_t+1)}.
\end{align*}
Then $\bb=\bb_{J,l}\in\QT^p(K)$ if and only if its coefficients satisfy the recurrence relations
\begin{align*}
&a^0_{(\mi_\bx,i_t+1)}=-\sum_{\ind=1}^n \frac{i_{x_\ind}+1}{g_\bzero(i_t+1)}a^\ind_{(\mi_\bx+\be_\ind,i_t)}  
- \sum_{\mj_\bx<\mi_\bx}  \frac{g_{\mi_\bx-\mj_\bx}}{g_\bzero} a^0_{(\mj_\bx,i_t+1)},\\
&a^\ind_{(\mi_\bx,i_t+1)}  = -\frac{i_{x_\ind}+1}{(i_t+1)\zeta_\bzero} a_{(\mi_\bx+\be_\ind,i_t)}^0
- \sum_{\mj_\bx<\mi_\bx}  \frac{\zeta_{\mi_\bx-\mj_\bx}}{\zeta_\bzero} a^\lambda_{(\mj_\bx,i_t+1)},
\qquad  \ind=1,\dots,n.
\end{align*}
The coefficients $a^\ind_{(\mk_\bx,0)}$, $\ind=0,\ldots,n$, $|\mk_\bx|\le p$, are known from the comparison with the space-only basis element $\widetilde b_J$.
All the other coefficients $a^\ind_{\mk}$ can be computed with a double loop: first over $|\mk|=1,\ldots,p$, and then over $k_t=1,\ldots,|\mk|$, similarly to Algorithms~\ref{algo:Basis1D}--\ref{algo:BasisND}.
The procedure is described in Algorithm \ref{algo:BasisFO}.

It is possible to verify that the $\bb_{J,l}$ constitute a basis of $\QT^p(K)$ following the lines of the proof of Proposition~\ref{prop:Basis}.
It follows that 
\begin{align*}
\dim\big(\QT^p(K)\big)
=(n+1)\binom{p+n}n 
=\frac{(n+1)(p+1)}{2p+2+n}\big(\dim\big(\QW^p(K)\big)+1\big)
=\calO_{p\to\infty}(p^n).
\end{align*}

\begin{algorithm}[H]
{\sc Algorithm}\\
\SetAlgoLined

Data: $(g_m)_{m\in\IN_0}$, $\bx_K$, $t_K$, $p$.\\
Choose polynomial basis $\{\widetilde b_J\}$, fixing coefficients $a_{\mk_\bx,0}^\ind$.\\
For each $J=1,\ldots,N(n,p)$ and $l=0,\ldots,n$, we construct $\bb_{J,l}$ as follows:\\
\For{$\ell=1$ to $p$\qquad (loop across $\{|\mi_\bx|+i_t=\ell-1\}$ hyperplanes, $\nearrow$)}{
    \For{$ i _t=0$ to $\ell-1$\qquad (loop across constant-time hyperplanes $\uparrow$)}{
        \For{ $\mi_\bx$ with $|\mi_\bx|=\ell- i _t-1$\qquad
}{
\begin{align*}
&a^0_{(\mi_\bx,i_t+1)}=-\sum_{\ind=1}^n \frac{i_{x_\ind}+1}{g_\bzero(i_t+1)}a^\ind_{(\mi_\bx+\be_\ind,i_t)}  
- \sum_{\mj_\bx<\mi_\bx}  \frac{g_{\mi_\bx-\mj_\bx}}{g_\bzero} a^0_{(\mj_\bx,i_t+1)},\\
&a^\ind_{(\mi_\bx,i_t+1)}  = -\frac{i_{x_\ind}+1}{(i_t+1)\zeta_\bzero} a_{(\mi_\bx+\be_\ind,i_t)}^0
- \sum_{\mj_\bx<\mi_\bx}  \frac{\zeta_{\mi_\bx-\mj_\bx}}{\zeta_\bzero} a^\lambda_{(\mj_\bx,i_t+1)},
\qquad  \ind=1,\dots,n.
\end{align*}
        }
    }
}
$\displaystyle
\bb_{J,l}(\bx,t)= \sum_{\substack{\mk\in\IN_0^{n+1},\: |\mk|\le p}}
\ba_{\mk}(\bx-\bx_K)^{\mk_\bx} (t-t_K)^{k_t}$
\vspace{2mm}
\caption{The algorithm for the construction of $\bb_{J,l}$ in the general case.}
\label{algo:BasisFO}
\end{algorithm}

\section{Numerical experiments}\label{sec:num}

We present some numerical test results in one and two space dimensions.
The quasi-Trefftz DG scheme has been implemented in NGSolve, see \cite{ngsolve}.  
Except for the last example, we consider the initial boundary value problem \eqref{eq:IBVP_U} with Dirichlet boundary conditions only, i.e.\ $\GN=\GR=\emptyset$.
In \Cref{sec:numcoeff,sec:nummesh,sec:numspace,sec:numgauss} we will assume $\rho\equiv 1$ and test our method with the following wavespeed parameters $ G$, exact solutions $u$, and space--time domain $Q$: 
\begin{subequations}\label{eq:exsol3x}\begin{align}
\label{eq:exsolAiry1D}
n=1,&\quad G(x)=x+1,\; &&
u(x,t)=\mathrm{Ai}(-x-1)\cos(t), && Q=(0,5)^2,\\
\label{eq:exsolAiry2D}
n=2,&\quad G(x_1,x_2)=x_1+x_2+1,\; &&
u(x_1,x_2,t)=\mathrm{Ai}(-x_1-x_2-1)\cos(\sqrt{2}t), && Q=(0,1)^3,\\
\label{eq:exsolPower2D}
n=2,&\quad G(x_1,x_2)=(x_1+x_2+1)^{-2},\; &&
u(x_1,x_2,t)=(x_1+x_2+1)^a \ee^{-\sqrt{2}\sqrt{a(a-1)}t}, && Q=(0,1)^3.
\end{align}\end{subequations}
Here $\mathrm{Ai}$ is the Airy function, which fulfills $\mathrm{Ai}''(x)=x\mathrm{Ai}(x)$, and we choose $a=2.5$ in \eqref{eq:exsolPower2D}.
The corresponding wavespeeds $c(\bx)$ range respectively in the intervals $[\sqrt{1/6},1]\approx[0.41,1]$, $[\sqrt{1/3},1]\approx[0.58,1]$ and $[1,3]$ for the three problems \eqref{eq:exsol3x}.
Then, the solution of the first-order wave equation is given by $\vs=(\partial_t u,-\nabla u)$.
In \S\ref{sec:numrho} we consider $\rho$ to be non-constant, testing with the exact solution given by
\begin{subequations}\begin{align}
n=1,&\quad G(x)=x^2-2,\; \rho(x)=x^{-2}&&
u(x,t)=\frac{\sin(x)-x\cos(x)}{x^2}\cos(t), && Q=(2,3)\times(0,1),\tag{\ref{eq:exsol3x}d}
\label{eq:exsolBessel}
\end{align}\end{subequations}
with the first-order system solution given by $\vs=(\partial_t u,-\frac1\rho\nabla u)$.

To construct the quasi-Trefftz basis we pre-compute coefficients of {$\frac1\rho$ and $G$}'s Taylor expansion \eqref{eq:Gg} at the centre of each mesh element.
We choose a monomial basis (scaled according to the element size) for $\left\{\widehat b_J\right\}$, and $\left\{\widetilde b_J\right\}$, as input of Algorithms~\ref{algo:Basis1D}--\ref{algo:BasisND}. 
This is motivated by experiments described in \cite[\S6.3]{StockerSchoeberl}, where monomials, chosen as initial basis for the construction of the standard Trefftz space, outperformed Legendre and Chebyshev basis. 
Remarkably, if space--time mesh elements share the same centre in space, namely $\bx_K$, then the coefficients of the quasi-Trefftz basis functions are identical on both elements, therefore they can be computed once and used on both elements.

\medskip
The section continues as follows.
In \S\ref{sec:numcoeff} we compare different choices for the penalisation coefficients.
The quasi-Trefftz discretisation is compared against a full polynomial space and a standard Trefftz space in \S\ref{sec:numspace}.
In \S\ref{sec:nummesh} we use a special type of space--time meshes allowing for semi-explicit time-stepping: tent-pitched meshes.
We show snapshots of the numerical approximation of a Gaussian pulse traveling through a heterogeneous medium in \S\ref{sec:numgauss}.
{Finally, we study the behaviour of the quasi-Trefftz scheme for a problem with variable mass density $\rho$ in \S\ref{sec:numrho} and for a problem with a coarse solution ($v_0,\bsigma_0\notin C^1\OO$) in \S\ref{sec:numl}.
}

\subsection{Volume penalisation and numerical flux parameters}\label{sec:numcoeff}
In this experiment we consider different combinations of the numerical flux parameters $\alpha,\beta$ and the volume penalisation coefficient $\mu_1$.
We recall that for $\QW^p\Th$ the choice of the parameter $\mu_2$ is irrelevant (see Remark~\ref{rem:mu2}).
Furthermore, we use Dirichlet boundary conditions, thus $\delta$ does not appear.
We compare the choices for the parameters given in \eqref{eq:FluxChoice} and \eqref{eq:muChoice} against setting them to zero.
{(The case $\alpha=\beta=0$ was investigated in \cite{KretzschmarPhD,EKSW15,KSTW2014} for the Maxwell equations with constant coefficients, even though in this case the DG scheme is not guaranteed to be well-posed.)}
We fix $p=4$ and a sequence of Cartesian meshes in 1+1 dimensions with square space--time mesh elements 
$K=(\bx_K-\frac h2,\bx_K+\frac h2)\times(t_K-\frac h2,t_K+\frac h2)$, and compare against the exact solution \eqref{eq:exsolAiry1D} (which can be seen in Figure~\ref{fig:mesh}).

\begin{table}[!ht]\centering
\begin{tabular}{lllllllll}
\toprule  
\multicolumn{9}{c}{$\mu_1= 0$, $p=4$,  problem \eqref{eq:exsolAiry1D}} \\
\midrule
& \multicolumn{2}{c}{$\alpha=0,\ \beta=0$} & \multicolumn{2}{c}{$\alpha=c^{-1},\ \beta=0$} & \multicolumn{2}{c}{$\alpha=0,\ \beta=c$} & \multicolumn{2}{c}{$\alpha=c^{-1},\ \beta=c$}\\
\midrule$ h$
& DG-error & rate & DG-error & rate & DG-error & rate & DG-error & rate \\
\midrule
\begin{filecontents*}{./cartairy1gppw.csv}
cond,dgerrorccrc,dgrateccrc,errorccrc,h,hnr,ndof,p,rateccrc,time,errorcnrc,dgerrorcnrc,ratecnrc,dgratecnrc,errorncrc,dgerrorncrc,ratencrc,dgratencrc,errornnrc,dgerrornnrc,ratennrc,dgratennrc,errorccn,dgerrorccn,rateccn,dgrateccn,errorcnn,dgerrorcnn,ratecnn,dgratecnn,errorncn,dgerrorncn,ratencn,dgratencn,errornnn,dgerrornnn,ratennn,dgratennn
0.0,3.104288327631245e-06,0.0,2.736342400244769e-07,0.125,3,17600,4,0.0,1.5522217750549316,2.758707444638012e-07,2.508572327928873e-06,0.0,0.0,2.0096435868323628e-07,2.769187545627369e-06,0.0,0.0,2.0675269183853905e-07,2.086929939650528e-06,0.0,0.0,2.6774202414466944e-07,3.0775556513165745e-06,0.0,0.0,2.629768253906863e-07,2.466649136826712e-06,0.0,0.0,1.9174118855142188e-07,2.7310891502430928e-06,0.0,0.0,1.8629243190453085e-07,2.024677936073648e-06,0.0,0.0
0.0,1.3766650667281346e-07,4.495013052522828,8.441890816132277e-09,0.0625,4,70400,4,5.018538786426566,6.975704669952393,8.402055793617588e-09,1.1001842014539372e-07,5.037106293645134,4.511049538542613,6.093067040747363e-09,1.2237202996522994e-07,5.0436272357183976,4.5001170144016465,6.075664185071068e-09,9.02480012447644e-08,5.088720171405148,4.531343377890303,8.377638953401193e-09,1.373233109225704e-07,4.998156077980533,4.486136493220957,8.222345252475815e-09,1.0943758497856786e-07,4.9992419073475975,4.494372313166539,5.99006168070321e-09,1.2186288100403482e-07,5.000445611337752,4.4861457501422,5.787447710908418e-09,8.936474135628435e-08,5.008498002220149,4.5018428958452334
0.0,6.095628417519614e-09,4.497258833536431,2.625064290998132e-10,0.03125,5,281600,4,5.0071415085351445,41.960684537887566,2.587922386458677e-10,4.849199447267442e-09,5.020876110339908,4.50385468813787,1.883130899472668e-10,5.413922171782751e-09,5.015963421500926,4.498455886908171,1.835293659380505e-10,3.959760783424334e-09,5.0489593036078695,4.510409798326214,2.618515519210313e-10,6.091311517010872e-09,4.99972265011625,4.494679849112763,2.568018512403161e-10,4.841956377603782e-09,5.0008224435210815,4.49837440689554,1.8725927371882887e-10,5.407081885318206e-09,4.9994617860194985,4.494264738031671,1.803426323587964e-10,3.948142227721333e-09,5.0041148636139265,4.500459864379127
0.0,2.696764805389702e-10,4.4984733509641055,8.192438773734182e-12,0.015625,6,1126400,4,5.001915960518116,156.30188179016113,8.046756515543959e-12,2.1415322919162009e-10,5.0072431610889865,4.5010312561743735,5.8653949378096284e-12,2.3948094244589483e-10,5.004761223986863,4.498691390123763,5.664520139743018e-12,1.7462740928587428e-10,5.0179133665150175,4.503061351676585,8.185438625272385e-12,2.696216522145968e-10,4.999545617947568,4.497744624235156,8.02766945901467e-12,2.1407137569615827e-10,4.999530577762021,4.499426276344211,5.854100919935744e-12,2.393844584953462e-10,4.999445740564521,4.497448808223579,5.634077074963639e-12,1.74479143768703e-10,5.000417374736382,4.5000474619172754
\end{filecontents*}
\csvreader[head to column names,filter ifthen=\equal{}{}]{./cartairy1gppw.csv}{}
{$2^{-\hnr}$ 
& \num[round-precision=2,round-mode=figures, scientific-notation=true]{\dgerrornnn} 
& \num[round-precision=3,round-mode=figures, scientific-notation=false]{\dgratennn} 
& \num[round-precision=2,round-mode=figures, scientific-notation=true]{\dgerrorcnn} 
& \num[round-precision=3,round-mode=figures, scientific-notation=false]{\dgratecnn} 
& \num[round-precision=2,round-mode=figures, scientific-notation=true]{\dgerrorncn} 
& \num[round-precision=3,round-mode=figures, scientific-notation=false]{\dgratencn} 
& \num[round-precision=2,round-mode=figures, scientific-notation=true]{\dgerrorccn} 
& \num[round-precision=3,round-mode=figures, scientific-notation=false]{\dgrateccn} 
\tabularnewline
}
\\\addlinespace[-\normalbaselineskip]\bottomrule
\end{tabular}
\caption{Errors committed by the quasi-Trefftz DG method for different combinations of the numerical flux parameters and vanishing volume penalisation coefficient.}
\label{tab:numcoeff1}
\end{table}

\begin{table}[!ht]\centering
\begin{tabular}{lllllllll}
\toprule  
\multicolumn{9}{c}{$\mu_1|_K= r_{K,c} \N{c}_{L^\infty(K)}^{-1}$, $p=4$,  problem \eqref{eq:exsolAiry1D}}\\
\midrule
& \multicolumn{2}{c}{$\alpha=0,\ \beta=0$} & \multicolumn{2}{c}{$\alpha=c^{-1},\ \beta=0$} & \multicolumn{2}{c}{$\alpha=0,\ \beta=c$} & \multicolumn{2}{c}{$\alpha=c^{-1},\ \beta=c$}\\
\midrule$ h$ 
& DG-error 
& rate & DG-error & rate & DG-error & rate & DG-error & rate \\
\midrule
\csvreader[head to column names,filter ifthen=\equal{}{}]{./cartairy1gppw.csv}{}
{$2^{-\hnr}$ 
& \num[round-precision=2,round-mode=figures, scientific-notation=true]{\dgerrornnrc} 
& \num[round-precision=3,round-mode=figures, scientific-notation=false]{\dgratennrc} 
& \num[round-precision=2,round-mode=figures, scientific-notation=true]{\dgerrorcnrc} 
& \num[round-precision=3,round-mode=figures, scientific-notation=false]{\dgratecnrc} 
& \num[round-precision=2,round-mode=figures, scientific-notation=true]{\dgerrorncrc} 
& \num[round-precision=3,round-mode=figures, scientific-notation=false]{\dgratencrc} 
& \num[round-precision=2,round-mode=figures, scientific-notation=true]{\dgerrorccrc} 
& \num[round-precision=3,round-mode=figures, scientific-notation=false]{\dgrateccrc} 
\tabularnewline
}
\\\addlinespace[-\normalbaselineskip]\bottomrule
\end{tabular}
\caption{Errors committed by the quasi-Trefftz DG method for different combinations of the numerical flux parameters and positive volume penalisation coefficient.}
\label{tab:numcoeff2}
\end{table}

The results are shown in \Cref{tab:numcoeff1,tab:numcoeff2}.
The errors are measured in the $\Tnorm{\cdot}\DG$ norm \eqref{eq:DGnorm}.
We observe optimal convergence in all cases, despite vanishing jump- or volume-penalisation term. 
Even though the volume penalisation term is needed for the well-posedness proof in Theorem~\ref{thm:WP}, in this example it is not necessary for the discrete problem to be well-posed and for the numerics to converge with optimal rate.
In this example, the choices suggested by the analysis (shown in the last column of \Cref{tab:numcoeff2}) result in a slightly larger error: this is because some of the terms on time-like faces in the $\Tnorm{\cdot}\DG$ norm vanish when $\alpha$ or $\beta$ are set to zero. 
Similar behaviours were observed for the wave equation in \cite[Fig.~6]{SpaceTimeTDG}, for the Helmholtz equation in \cite[Fig.~7--8]{GHP09} (concerning the flux parameters) and in \cite[\S5.1]{IGM17} (concerning the volume penalisation parameter).

The $\calO(h^{p+1/2})$ convergence rates observed coincide with those proved in the bound \eqref{eq:Convergence2}.
If, instead of using the $\Tnorm{\cdot}\DG$ norm, we measure the error at final time only, specifically in the $L^2(\Omega\times\{T\})$ norm for both the $v$ and the $\bsigma$ components, we obtain $\calO(h^{p+1})$ convergence rates (we do not report the values here), i.e.\ they are half a power higher than those in the $\Tnorm{\cdot}\DG$ norm.
The same half-order difference has been observed for the non-Trefftz version of the same method and $c=1$ in Table~1 of \cite{BMPS20}; see also the considerations after Proposition~6.5 therein.
Moreover, being the $L^2(\Omega\times\{T\})$ norm parameter-independent, the errors are slightly smaller for the flux parameter values suggested in \eqref{eq:FluxChoice}.

\subsection{Approximation properties of quasi-Trefftz spaces}\label{sec:numspace}

We compare the numerical error for different choices of the discretisation spaces:
the quasi-Trefftz space $\QW^p\Th$ of \eqref{eq:QWglobal}, 
the first-order derivatives $\PP^p\Th$ of the full polynomial space, and 
the Trefftz space $\IW^p\Th$, respectively defined by
\begin{align*}
\PP^p\Th:=&
\Big\{\wt\in\bH\Th: \; w|_K=\partial_t u,\;\btau |_K=-{\frac1{\rho(\bx_K)}}\nabla u, \; u\in \IP^{p+1}(K), \;\forall K \in\Th\Big\},
\qquad p\in\IN_0,
\\
\IW^p\Th:=&
\Big\{\wt\in\bH\Th: \; w|_K=\partial_t u,\;\btau |_K=-{\frac1{\rho(\bx_K)}}\nabla u, \; u\in \IP^{p+1}(K),\;\\
&\hspace{60mm} {-\nabla\cdot\Big(\frac1{\rho(\bx_K)}\nabla u\Big)+G(\bx_K)\,}\partial_t^2 u=0 \iin K,\forall K \in\Th\Big\}.
\end{align*}
The space $\IW^p\Th$ is the Trefftz space (as in \cite[\S6.2]{MoPe18}) for the approximated IBVP in which the
{parameters $\rho$ and $G$ are substituted by elementwise-constant approximants.}
We have $\dim\IW^p\Th=\dim\QW^p\Th$, {$\IW^p\Th\subset\PP^p\Th$, and, if $\rho$ is constant, $\QW^p\Th\subset\PP^p\Th$}.

We consider the problems \eqref{eq:exsolAiry2D} and \eqref{eq:exsolPower2D} in 2+1 dimensions and set initial and boundary conditions accordingly.
We use meshes that are Cartesian product between a spatial, quasi-uniform, unstructured, triangular mesh in $(0,1)^2$ with spatial meshwidth $h$, and a uniform mesh in time with time-step $h_t\approx h$.
Therefore all elements are right triangular prisms and all their sides have comparable lengths.
We set the volume penalisation and numerical flux parameters to the values in \eqref{eq:muChoice} and \eqref{eq:FluxChoice}, respectively.
The errors are measured in $\Tnorm{\cdot}\DG$ norm.
The results are displayed in \Cref{fig:numspace,fig:numspacedof}.

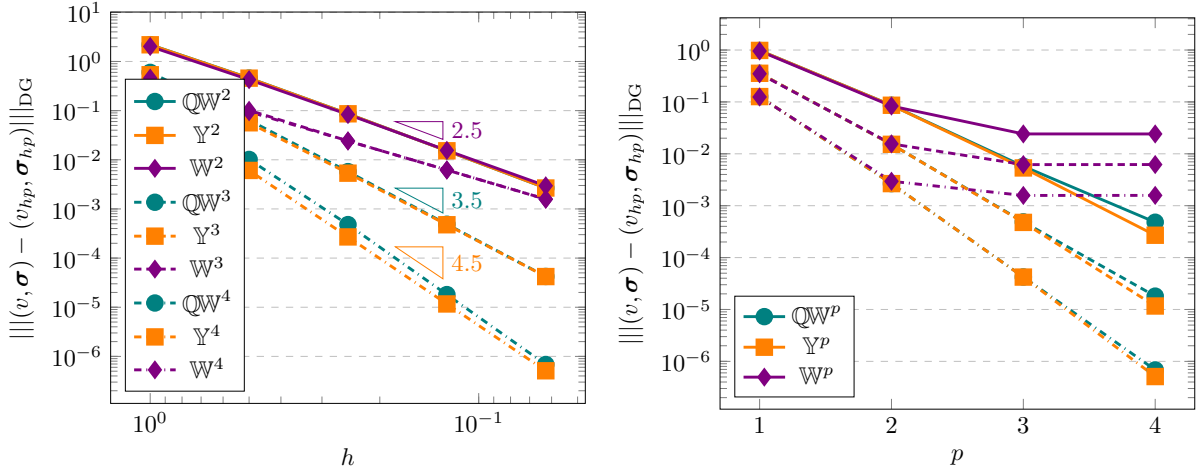
\begin{figure}[ht]
    \resizebox{.49\linewidth}{!}{
        \begin{tikzpicture}
            \begin{loglogaxis}
                [ xlabel={$h$},
                ylabel={$\Tnorm{\vs-\vsh}\DG$},
                ymajorgrids=true,
                grid style=dashed,
                legend pos=south west,
                title={},
                cycle list name=paulcolors,
                x dir=reverse,				
                ]
                \foreach \p in {2,3,4}{
                    \addplot+[discard if not={p}{\p}] table [x=h, y=dgerrorccrc, col sep=comma] {./cartquad2gppw.csv};
                    \addlegendentryexpanded{$\QW^\p$};
                    \addplot+[discard if not={p}{\p}] table [x=h, y=dgerrorccrc, col sep=comma] {./cartquad2pol.csv};
                    \addlegendentryexpanded{$\PP^\p$};
                    \addplot+[discard if not={p}{\p}] table [x=h, y=dgerrorccrc, col sep=comma] {./cartquad2trefftz.csv};
                    \addlegendentryexpanded{$\mathbb W^\p$};
                }
                \logLogSlopeTriangle{0.7}{0.1}{0.72}{2.5}{violet}; 
                \logLogSlopeTriangle{0.7}{0.1}{0.55}{3.5}{teal};
                \logLogSlopeTriangle{0.7}{0.1}{0.4}{4.5}{orange};
            \end{loglogaxis}
        \end{tikzpicture}
    }
    \resizebox{.49\linewidth}{!}{
        \begin{tikzpicture}
            \begin{semilogyaxis}
                [ xlabel={$p$},
                ylabel={$\Tnorm{\vs-\vsh}\DG$},
                ymajorgrids=true,
                grid style=dashed,
                legend pos=south west,
                title={},
                cycle list name=paulcolors,
                xtick={1,2,3,4}
                ]
                \foreach \hnr in {2,3,4}{
                    \addplot+[discard if not={hnr}{\hnr}] table [x=p, y=dgerrorccrc, col sep=comma] {./cartquad2gppw.csv};
                    \addplot+[discard if not={hnr}{\hnr}] table [x=p, y=dgerrorccrc, col sep=comma] {./cartquad2pol.csv};
                    \addplot+[discard if not={hnr}{\hnr}] table [x=p, y=dgerrorccrc, col sep=comma] {./cartquad2trefftz.csv};
                }
                    \addlegendentryexpanded{$\QW^p$};
                    \addlegendentryexpanded{$\PP^p$};
                    \addlegendentryexpanded{$\mathbb W^p$};
            \end{semilogyaxis}
        \end{tikzpicture}
    }
\caption{Comparison of different approximation spaces for problem \eqref{eq:exsolPower2D} as described in Section~\ref{sec:numspace}. 
Left panel: $h$-convergence.
Right panel: $p$-convergence; the three sets of curves correspond to $h=2^{-2},2^{-3},2^{-4}$.}
\label{fig:numspace}
\end{figure}

\Cref{fig:numspace} focuses on \eqref{eq:exsolPower2D}.
The left panel plots the error against the mesh size for different values of $p$: the quasi-Trefftz space and the full polynomial space show the same, optimal, rate of convergence $\calO(h^{p+1/2})$.  
The full polynomial space has a slightly smaller error throughout.
The standard Trefftz space, however, does not achieve convergence with the same rate, but the rate is instead limited by roughly $\calO(h^2)$; this is due to the low-order (piecewise-constant) approximation of $c$ in the construction of the basis functions.
The right panel of \Cref{fig:numspace} shows the error against the polynomial degree $p$ for mesh sizes $h=2^{-2},2^{-3},2^{-4}$.
We observe exponential convergence for both the quasi-Trefftz space and the full polynomial space.
As expected, the standard Trefftz space does not lead to convergence in $p$ because the approximation of $c$ does not improve with $p$-refinement.

\begin{figure}[ht]
    \resizebox{.49\linewidth}{!}{
        \begin{tikzpicture}
            \begin{loglogaxis}
                [ xlabel={\#dof},
                ylabel={$\Tnorm{\vs-\vsh}\DG$},
                ymajorgrids=true,
                grid style=dashed,
                legend pos=south west,
                title={},
                cycle list name=paulcolors,
                ]
                \foreach \hnr in {3}{
                    \addplot+[discard if not={hnr}{\hnr}] table [x=ndof, y=dgerrorccrc, col sep=comma] {./cartairy2gppw.csv};
                    \addplot+[discard if not={hnr}{\hnr}] table [x=ndof, y=dgerrorccrc, col sep=comma] {./cartairy2pol.csv};
                    \addplot+[discard if not={hnr}{\hnr}] table [x=ndof, y=dgerrorccrc, col sep=comma] {./cartairy2trefftz.csv};
                }
                    \addlegendentryexpanded{$\QW^p$};
                    \addlegendentryexpanded{$\PP^p$};					
                    \addlegendentryexpanded{$\IW^p$};
                \foreach \hnr in {4}{
                    \addplot+[discard if not={hnr}{\hnr}] table [x=ndof, y=dgerrorccrc, col sep=comma] {./cartairy2gppw.csv};
                    \addplot+[discard if not={hnr}{\hnr}] table [x=ndof, y=dgerrorccrc, col sep=comma] {./cartairy2pol.csv};
                    \addplot+[discard if not={hnr}{\hnr}] table [x=ndof, y=dgerrorccrc, col sep=comma] {./cartairy2trefftz.csv};
				}
            \end{loglogaxis}
        \end{tikzpicture}
    }	
    \resizebox{.49\linewidth}{!}{
        \begin{tikzpicture}
            \begin{loglogaxis}
                [ xlabel={time (s)},
                ylabel={$\Tnorm{\vs-\vsh}\DG$},
                ymajorgrids=true,
                grid style=dashed,
                legend pos=south west,
                title={},
                cycle list name=paulcolors,
                ]
                \foreach \hnr in {4}{
                    \addplot+[discard if not={hnr}{\hnr}] table [x=time, y=dgerrorccrc, col sep=comma] {./cartairy2gppw.csv};
                    \addplot+[discard if not={hnr}{\hnr}] table [x=time, y=dgerrorccrc, col sep=comma] {./cartairy2pol.csv};
                    \addplot+[discard if not={hnr}{\hnr}] table [x=time, y=dgerrorccrc, col sep=comma] {./cartairy2trefftz.csv};
                }
                    \addlegendentryexpanded{$\QW^p$};
                    \addlegendentryexpanded{$\PP^p$};
                    \addlegendentryexpanded{$\IW^p$};
            \end{loglogaxis}
        \end{tikzpicture}
    }
\caption{Left panel: comparison of different approximation spaces in terms of numbers of degrees of freedom for problem \eqref{eq:exsolAiry2D}, as described in \S\ref{sec:numspace}.
The continuous lines correspond to a mesh with $h=2^{-3}$ and the dashed ones to $h=2^{-4}$.
The nodes in each line correspond to polynomial degrees $p=1,2,3,4$.
Right panel: the same errors (for $h=2^{-4}$) plotted against the computational time, including calculating the basis functions, assembly and solve.
Both plots show that the quasi Trefftz space $\QW^p\Th$ allows more efficient computations than the full polynomial space $\PP^p\Th$.}
    \label{fig:numspacedof}
\end{figure}
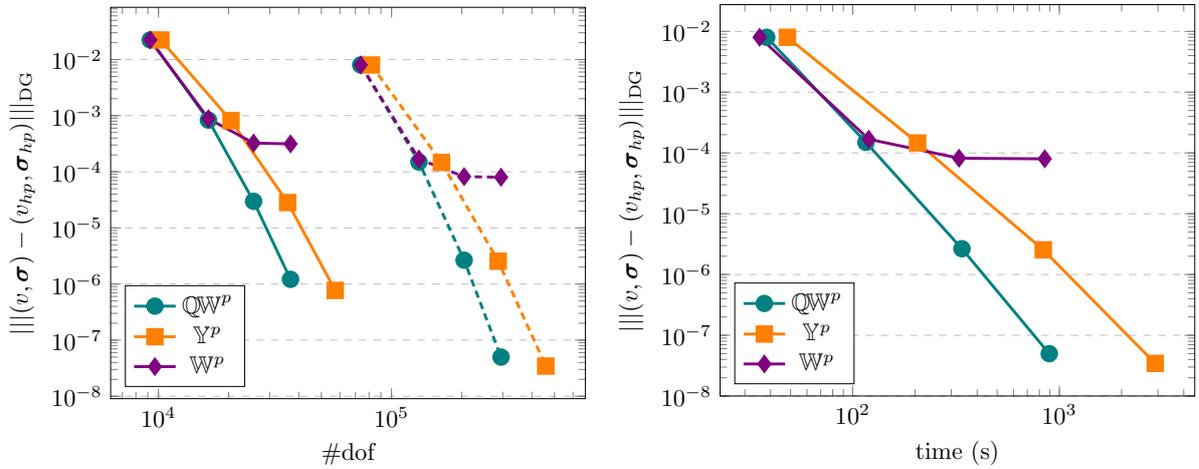

In \Cref{fig:numspacedof} we switch to problem \eqref{eq:exsolAiry2D} and plot (in the left panel) the error against the global number of degrees of freedom, on a fixed mesh, for increasing polynomial degrees $p$. 
The continuous and dashed lines correspond to two different mesh sizes, $h=2^{-3}$ and $h=2^{-4}$ respectively.
The right panel plots the same error against the computational time.
These plots illustrate the power of the quasi-Trefftz approach compared to the full polynomial approach, as discussed in Remark \ref{rem:Dimension}: for comparable numbers of degrees of freedom the quasi-Trefftz method can achieve much higher accuracy. 
In this example the accuracy improvement is up to about one and a half orders of magnitude, as observed when comparing the errors and the number of degrees of freedom for $\QW^4\Th$ and $\PP^3\Th$.

\subsection{Tent-pitched meshes}\label{sec:nummesh}

The meshes used in all numerical examples in \S\ref{sec:numcoeff}--\ref{sec:numspace} are Cartesian products between a mesh in space and one in time.
Thus the numerical solution has to be computed simultaneously for all the elements corresponding to the same time interval; this is analogous to an implicit time-stepping scheme.

We now discuss an alternative space--time meshing strategy: tent pitching. 
We call a mesh ``tent-pitched'' if all interior faces are space-like according to the definition in \eqref{eq:HorVerFaces}.
This implies that the numerical solution in a given element $K$ can be computed only from the numerical solutions on the elements that are adjacent to $K$ and lying ``before'' $K$, thanks to the causality constraint (represented in the DG formulation \eqref{eq:DG} by the use of the $v\hp^-$ and $\bsigma\hp^-$ traces on $\Fspa$).
The solution can be computed independently, and in parallel, in several mesh elements and the solution procedure resembles an explicit time-stepping.

\begin{figure}[ht!]\centering
\centering
\includegraphics[width=0.4\textwidth,clip, trim = 70 100 70 100]{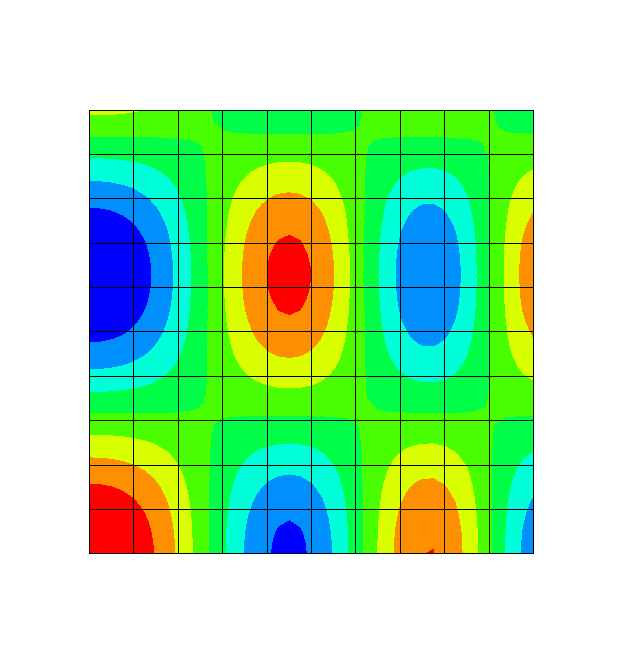}\qquad
\includegraphics[width=0.4\textwidth,clip, trim = 70 100 70 100]{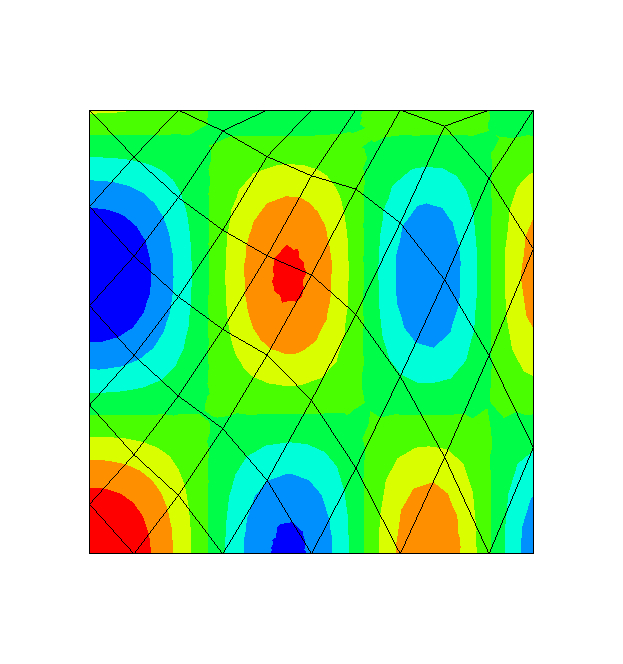}
\caption{A Cartesian-product mesh (left) and a tent-pitched mesh (right) on the domain $Q=(0,5)^2$.
Both show the solution $u$ of problem \eqref{eq:exsolAiry1D}.} 
\label{fig:mesh}
\end{figure}

An example of a 1+1-dimensional tent-pitched mesh on $Q=(0,5)\times(0,5)$ can be seen in \Cref{fig:mesh}.
This mesh is constructed for the wavespeed $c(x)=(1+x)^{-1/2}$ of problem \eqref{eq:exsolAiry1D}, thus the tents in the right part of the domain are allowed to be ``taller'' than those on the left, i.e.\ to have longer extension in the time direction, without violating the causality constraint of having slope bounded by $c^{-1}(x)$.

The algorithm used to produce the tent-pitched mesh used here can be found in \cite{GSW16}.
A closer look into the implementation of Trefftz functions on tent-pitched meshes is given in \cite{StockerSchoeberl}.

To optimize storage during the computations on a tent-pitched mesh we only need to store the solution furthest in time.
Therefore, in this section we measure the error at the final-time term in the definition of the DG norm \eqref{eq:DGnorm}:
\begin{equation}\label{eq:fterror}
\mathrm{error}(T)=\left( \N{\sqrt{G(\cdot)} \big(v(\cdot,T)-v\hp(\cdot,T)\big)}^2_{L^2(\Omega)}
+\N{{\sqrt{\rho(\cdot)}\big(}\bsigma(\cdot,T)-\bsigma\hp(\cdot,T){\big)}}^2_{L^2(\Omega)}\right)^{1/2},
\end{equation}
with final time $T=1$.
We use tent-pitched meshes in 2+1 dimensions to approximate 
problem \eqref{eq:exsolAiry2D}.

\pgfplotscreateplotcyclelist{paulcolors2}{%
teal,every mark/.append style={solid,fill=teal},mark=*,very thick,mark size=3pt\\%
orange,every mark/.append style={solid,fill=orange},mark=square*,very thick,mark size=3pt\\%
teal,densely dashed,every mark/.append style={solid,fill=teal},mark=*,very thick,mark size=3pt\\%
orange,densely dashed,every mark/.append style={solid,fill=orange},mark=square*,very thick,mark size=3pt\\%
teal,dash dot,every mark/.append style={solid,fill=teal},mark=*,very thick,mark size=3pt\\%
orange,dash dot,every mark/.append style={solid,fill=orange},mark=square*,very thick,mark size=3pt\\%
}

\begin{figure}[ht]
    \resizebox{.49\linewidth}{!}{
        \begin{tikzpicture}
            \begin{loglogaxis}
                [ xlabel={\#dof$^{-1/3}$},
                ylabel={error$(T)$},
                ymajorgrids=true,
                grid style=dashed,
                legend pos=south west,
                cycle list name=paulcolors2,
                x dir=reverse
                ]
            \foreach \p in {3,4}{
                    \addplot+[discard if not={p}{\p}] table [x=cndof, y=error, col sep=comma] {./tentairygppw2.csv};
                    \addlegendentryexpanded{$\QW^\p$ tents};
                }
            \foreach \p in {3,4}{
                    \addplot+[discard if not={p}{\p}] table [x=cndof, y=errorccrc, col sep=comma] {./cartairy2gppw.csv};
                    \addlegendentryexpanded{$\QW^\p$};
                }
                \logLogSlopeTriangle{0.6}{0.1}{0.7}{4}{teal};
                \logLogSlopeTriangle{0.6}{0.1}{0.35}{5}{orange};
        \end{loglogaxis}
    \end{tikzpicture}
}
    \resizebox{.49\linewidth}{!}{
        \begin{tikzpicture}
            \begin{loglogaxis}
                [ xlabel={time (s)},
                ylabel={error$(T)$},
                ymajorgrids=true,
                grid style=dashed,
                legend pos=south west,
                cycle list name=paulcolors2,
                ]
            \foreach \p in {3,4}{
                    \addplot+[discard if not={p}{\p}] table [x=time1, y=error, col sep=comma] {./tentairygppw2.csv};
                    \addlegendentryexpanded{$\QW^\p$ tents};
                }
            \foreach \p in {3,4}{
                    \addplot+[discard if not={p}{\p}] table [x=time, y=errorccrc, col sep=comma] {./cartairy2gppw.csv};
                    \addlegendentryexpanded{$\QW^\p$};
                }
        \end{loglogaxis}
    \end{tikzpicture}
}
\caption{The final-time error against the number of degrees of freedom (left) and against the computational time (right) for the sequential solution of Problem~\eqref{eq:exsolAiry2D} on tent-pitched (continuous lines) and Cartesian meshes (dashed lines).
}
\label{fig:tent1}
\end{figure}
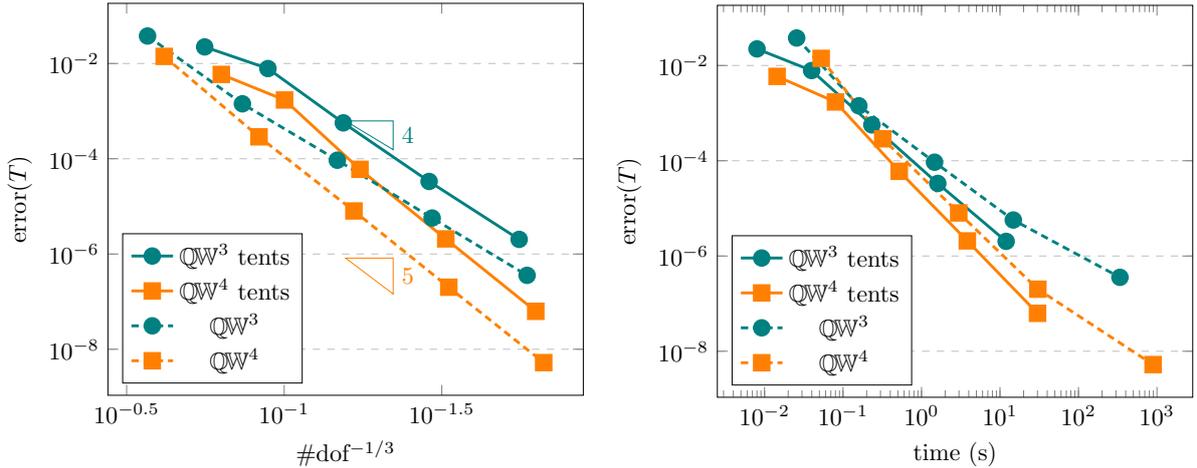

We first compare the error committed on tent-pitched meshes against that on Cartesian-product meshes (of the same kind of those in \S\ref{sec:numspace}).
For a fair comparison we plot the error in terms of the number of degrees of freedom, for varying mesh sizes. 
On the left panel of \Cref{fig:tent1} we observe optimal convergence rates of $\calO(\#\mathrm{dof}^{-(p+1)/3})$ for both meshing strategies, which corresponds to $\calO(h^{p+1})$.
The Cartesian-product mesh outperforms the tent-pitched mesh in terms of efficiency per degrees of freedom, due to the fact that we need more tent elements to cover the same space--time volume.
However, in terms of computational time, shown in the right panel of \Cref{fig:tent1}, the tents perform better since they do not require the solution of any large linear system, even though in this comparison the solution is only solved sequentially without any parallelisation.

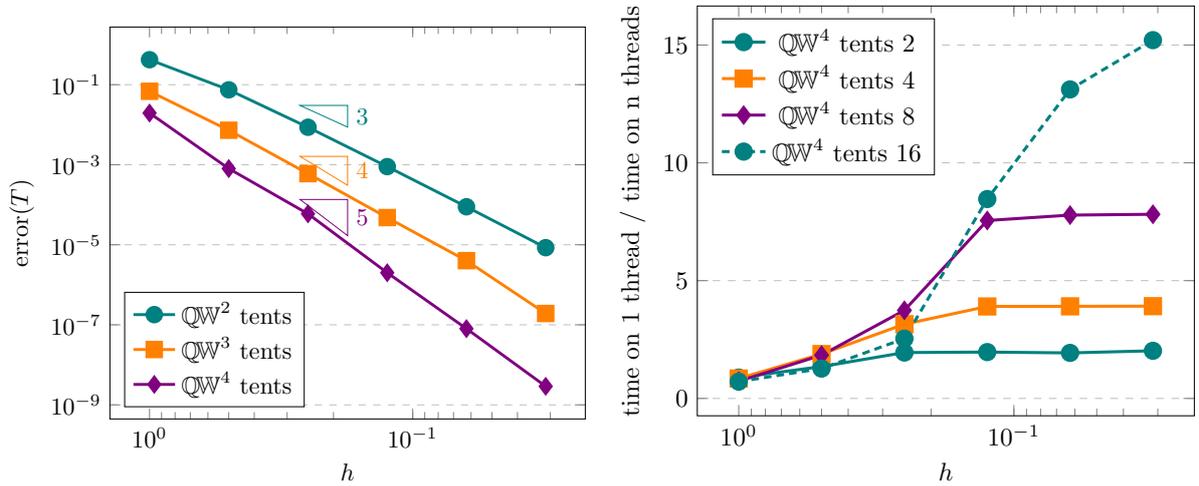
\begin{figure}[ht]
    \resizebox{.49\linewidth}{!}{
        \begin{tikzpicture}
            \begin{loglogaxis}
                [ xlabel={$h$},
                ylabel={error$(T)$},
                ymajorgrids=true,
                grid style=dashed,
                legend pos=south west,
                cycle list name=paulcolors,
                x dir=reverse
                ]
            \foreach \p in {2,3,4}{
                    \addplot+[discard if not={p}{\p}] table [x=h, y=error, col sep=comma] {./tentquadgppw2.csv};
                    \addlegendentryexpanded{$\QW^\p$ tents};
                }
                \logLogSlopeTriangle{0.5}{0.1}{0.8}{3}{teal};
                \logLogSlopeTriangle{0.5}{0.1}{0.67}{4}{orange};
                \logLogSlopeTriangle{0.5}{0.1}{0.56}{5}{violet};
        \end{loglogaxis}
    \end{tikzpicture}
}
\resizebox{.49\linewidth}{!}{
    \begin{tikzpicture}
        \begin{semilogxaxis}
            [ xlabel={$h$},
            ylabel={time on 1 thread / time on n threads},
            ymajorgrids=true,
            grid style=dashed,
            legend pos=north west,
            cycle list name=paulcolors,
            x dir=reverse
            ]
            \foreach \p in {4}{
            \foreach \time in {2,4,8,16}{
                \addplot+[discard if not={p}{\p}] table [x=h, y=timescale\time, col sep=comma] {./tentquadgppw2.csv};
                \addlegendentryexpanded{$\QW^\p$ tents \time};
            }
            }
        \end{semilogxaxis}
    \end{tikzpicture}
}
\caption{Left panel: the error on tent-pitched meshes for problem \eqref{eq:exsolPower2D}.
Right panel: the speedup in the computational time for 2 to 16 threads.
}
\label{fig:tent2}
\end{figure}

Next we study the effect of parallelisation.
We measure the speedup obtained by increasing the number of threads, i.e.\ the maximum number of elements on which the solution is computed independently in parallel.
Now we consider problem \eqref{eq:exsolPower2D}; the final-time error in terms of the mesh size is shown in the left panel of \Cref{fig:tent2}.
The speedup in the computational time for 2, 4, 8 and 16 threads is shown in the right panel of \Cref{fig:tent2}. 
We observe that the speedup factor is quite close to the number of threads.
The figure shows that increasing the number of threads is beneficial only for moderate mesh sizes, as otherwise there are not enough independent tents.
All timings were performed on a server with two Intel(R) Xeon(R) CPU E5-2687W v4, with 12 cores each.

\subsection{Gaussian pulse in a non-homogenous medium}\label{sec:numgauss}

We illustrate the propagation of a vertical Gaussian pulse traveling through a medium with wavespeed varying along the $x_2$-direction, $G(x_1,x_2)=1+x_2$.
The initial conditions are given by
\begin{equation*}
\bsigma_0(x_1,x_2) = \Big(-\frac{2x_1}{\delta^2} \ee^{-\frac{x_1^2}{\delta^2}},0\Big),
\qquad v_0(x) = 0 \qquad\oon \Omega=(0,1)^2,
\end{equation*}
setting $\delta=2^{-5}$.
We choose homogeneous Neumann boundary conditions, a tent-pitched mesh as discussed in the previous section, spatial mesh size $h=2^{-7}$ and polynomial degree $p=3$.
Snapshots of the solution are shown in \Cref{fig:pulse}.
At $T=0$ the initial condition is constant in $x_2$-direction.
In the next snapshot, at $T=0.25$, we can see the expected effects of the variable wavespeed: at the top of the domain, the wave travels faster than at the bottom.
At $T=0.5$ the wavefront on the top side reaches the right border.
In the last image, at $T=0.75$, we can see the wave being reflected from the right boundary.
Boundary effects due to the homogeneous Neumann boundary conditions at the top and bottom of the domain can also be observed.

In \Cref{fig:energy} we plot the energy \eqref{eq:Energy} for different spatial mesh sizes $h=2^{-5},2^{-6},2^{-7}$.
The energy is computed at constant times $t$ multiple of 0.0025 as 
$\calE(t;w,\btau):=\frac12\int_\Omega (c^{-2} w^2+|\btau|^2) \di S$,
by forcing the tent pitched mesh into slabs.
As observed in \S\ref{s:WellP}, the method is dissipative.
For $h=\delta=2^{-5}$ there are not enough elements to resolve the wave front with sufficient accuracy, and the energy dissipates very quickly. 
For the two finer meshes the energy loss behaves much better; in particular for $h=2^{-7}$ only 0.076\% of the initial energy is lost at the final time $T=1$.

\begin{figure}[ht]
\centering
\includegraphics[width=0.24\textwidth]{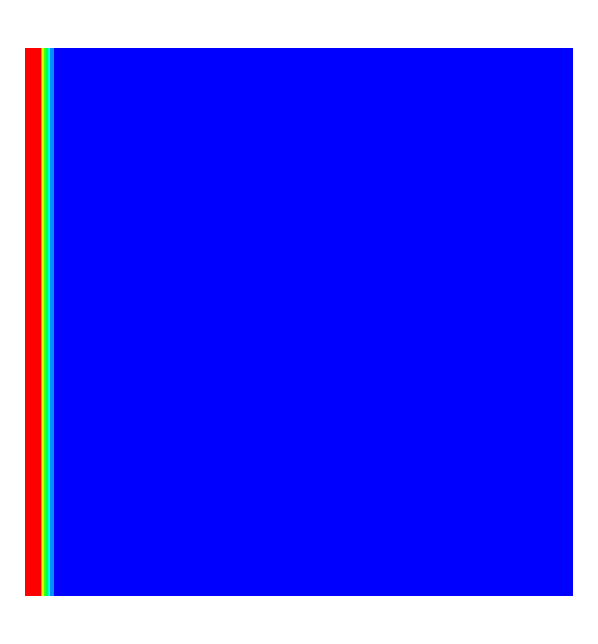}
\includegraphics[width=0.24\textwidth]{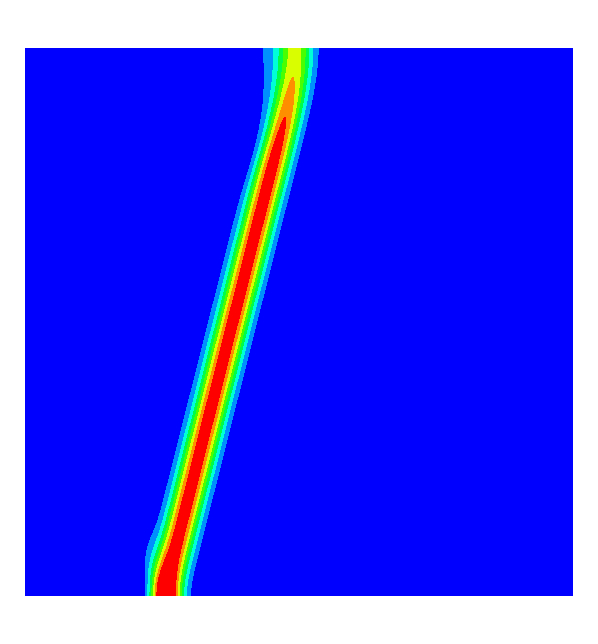}
\includegraphics[width=0.24\textwidth]{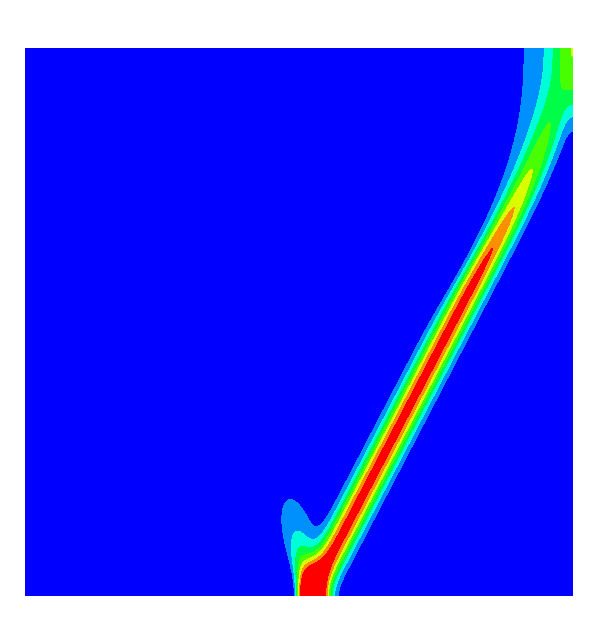}
\includegraphics[width=0.24\textwidth]{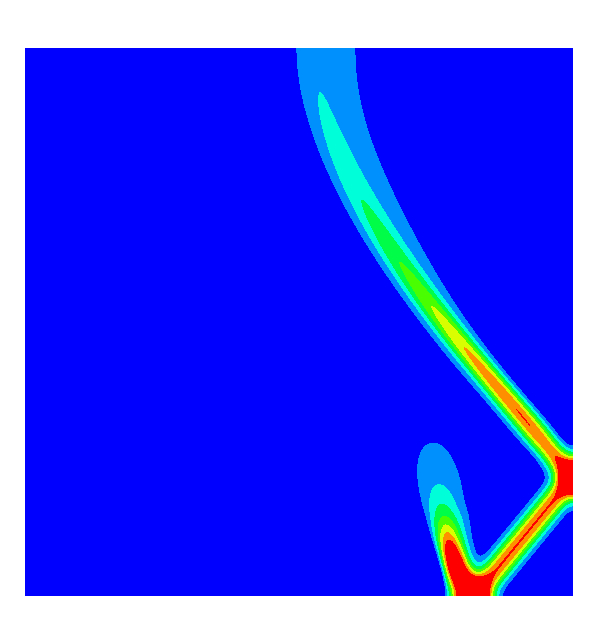}
\caption{Snapshots of the solution of the problem described in Section~\ref{sec:numgauss} at times $0, 0.25, 0.5, 0.75$.}
\label{fig:pulse}
\end{figure}

\pgfplotscreateplotcyclelist{paulcolors3}{%
orange,densely dashed,every mark/.append style={fill=orange}\\%
violet,densely dotted,every mark/.append style={fill=violet}\\%
teal,every mark/.append style={fill=orange}\\%
}
\begin{figure}[ht]
\resizebox{.45\linewidth}{!}{
    \begin{tikzpicture}
        \begin{axis}
            [ xlabel={time(s)},
            ylabel={energy},
            ymajorgrids=true,
            grid style=dashed,
            legend pos=south west,
            cycle list name=paulcolors3,
            ]
            \addplot+[very thick,discard if not ={h}{5.0}, mark=none] table [x=t, y=energy, col sep=comma] {./energy.csv};
            \addlegendentryexpanded{$h=2^{-5}$}
            \addplot+[very thick,discard if not ={h}{6.0}, mark=none] table [x=t, y=energy, col sep=comma] {./energy.csv};
            \addlegendentryexpanded{$h=2^{-6}$}
            \addplot+[very thick,discard if not ={h}{7.0}, mark=none] table [x=t, y=energy, col sep=comma] {./energy.csv};
            \addlegendentryexpanded{$h=2^{-7}$}
        \end{axis}
    \end{tikzpicture}
}
\caption{The energy \eqref{eq:Energy} of the numerical solution shown in \Cref{fig:pulse}.} 
\label{fig:energy}
\end{figure}
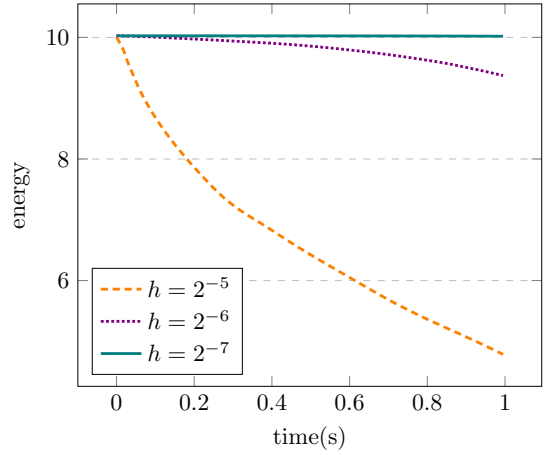

\subsection{Variable mass density}\label{sec:numrho}

We consider the setting given in \Cref{eq:exsolBessel}.
The exact solution is a first-order spherical Bessel function in space, corresponding to smoothly varying $G$ and $\rho$.
We solve in 1+1 dimensions over the space domain $\Omega=(2,3)$ using uniform Cartesian-product meshes.
The wavespeed $c(x)=(\rho G)^{-1/2}$ has range $[\sqrt{8/7},\sqrt{2}]$.

\begin{figure}[ht]
    \resizebox{.49\linewidth}{!}{
        \begin{tikzpicture}
            \begin{loglogaxis}
                [ xlabel={$h$},
                ylabel={$\Tnorm{\vs-\vsh}\DG$},
                ymajorgrids=true,
                grid style=dashed,
                legend pos=south west,
                cycle list name=paulcolors2,
                x dir=reverse
                ]
            \foreach \p in {2,3,4}{
                    \addplot+[discard if not={p}{\p}] table [x=h, y=dgerrornnn, col sep=comma] {./cartbesselgppw.csv};
                    \addlegendentryexpanded{$\QW^\p$};
                    \addplot+[discard if not={p}{\p}] table [x=h, y=dgerrornnn, col sep=comma] {./cartbesselmon.csv};
                    \addlegendentryexpanded{$\PP^\p$};
                }
                \logLogSlopeTriangle{0.5}{0.1}{0.85}{2.5}{teal};
                \logLogSlopeTriangle{0.5}{0.1}{0.65}{3.5}{orange};
                \logLogSlopeTriangle{0.5}{0.1}{0.45}{4.5}{violet};
        \end{loglogaxis}
    \end{tikzpicture}
}
\resizebox{.49\linewidth}{!}{
    \begin{tikzpicture}
        \begin{loglogaxis}
            [ xlabel={$h$},
            ylabel={condition number},
            ymajorgrids=true,
            grid style=dashed,
            legend pos=north west,
            cycle list name=paulcolors2,
            x dir=reverse
            ]
            \foreach \p in {2,3,4}{
                \addplot+[discard if not={p}{\p}] table [x=h, y=condnnn, col sep=comma] {./cartbesselgppw.csv};
                \addlegendentryexpanded{$\QW^\p$};
                \addplot+[discard if not={p}{\p}] table [x=h, y=condnnn, col sep=comma] {./cartbesselmon.csv};
                \addlegendentryexpanded{$\PP^\p$};
            }
        \end{loglogaxis}
    \end{tikzpicture}
}
\caption{
Numerical results for the approximation of the exact solution given in \Cref{eq:exsolBessel}.
Left panel: $h$-convergence rates. Right panel: condition number.
}
\label{fig:bessel}
\end{figure}
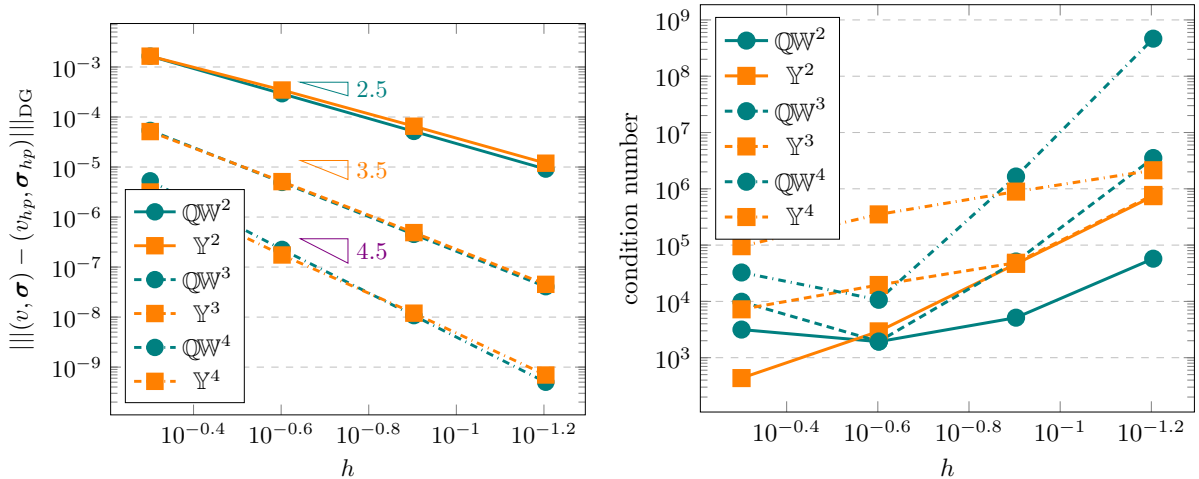

\begin{table}[!ht]\centering
\begin{tabular}{lllllll}
\toprule & 
\multicolumn{3}{c}{$\mu_1|_K=0$,} &
\multicolumn{3}{c}{$\mu_1|_K= r_{K,c} \N{c}_{L^\infty(K)}^{-1}$,} \\
&
\multicolumn{3}{c}{$\alpha=0,\ \beta=0$} &
\multicolumn{3}{c}{$\alpha=(\rho c)^{-1},\ \beta=\rho c$} \\
\midrule $p$ & DG-error & $L^2$-error & condition & DG-error & $L^2$-error & condition \\
\midrule
\begin{filecontents*}{./cartbesselgppw.csv}
condnnn,dgerrornnn,dgratennn,errornnn,h,hnr,p,ratennn,errorccrc,dgerrorccrc,rateccrc,dgrateccrc,condccrc
1095.4573100101843,0.012330099379521424,0.0,0.001315810118462156,0.5,1,1,0.0,0.002009083663400466,0.024222996585709634,0.0,0.0,2072.2717132267167
1430.745179270139,0.005628771601594669,1.131292413447347,0.00027244407308117945,0.25,2,1,2.271919301005903,0.000551158077006758,0.009956589467475262,1.8659995819026292,1.2826537984996496,3246.7983683004322
5015.086491654165,0.0021738983007408906,1.3725356594048879,6.114097399035866e-05,0.125,3,1,2.155748665967179,0.00012118354560442177,0.003793739093327991,2.185272333971434,1.3920311820592557,12420.821891443022
19288.672942476016,0.0007952984529923667,1.4507161800946038,1.4558820379616987e-05,0.0625,4,1,2.0702460681826365,2.6790250538997927e-05,0.0013829758695336986,2.1774138436043113,1.4558444795022176,50683.19706682943
3158.529936103092,0.0016485308353780858,0.0,0.00013518000468884604,0.5,1,2,0.0,0.00011034950685845645,0.0028635554709085234,0.0,0.0,7008.57564985583
1927.382605789696,0.00029194612838852183,2.497406788300824,1.1314408842769935e-05,0.25,2,2,3.5786486559469273,1.3374055050074096e-05,0.0005156685829442045,3.044571316408945,2.473291493391761,4482.4206570985025
5138.22009817591,5.151079004898106e-05,2.50275560637419,1.0489442416005035e-06,0.125,3,2,3.431151312209827,1.7506181075758066e-06,9.176954076657988e-05,2.9335006575073557,2.4903568589574197,12727.083474540925
57370.48914311111,9.086244282228378e-06,2.503118669974015,1.0184052130127121e-07,0.0625,4,2,3.36455437671537,2.2476143332891702e-07,1.627978416782285e-05,2.9613979874371035,2.4949338151553113,113233.34353401027
9860.118104236819,5.395140954209147e-05,0.0,4.2366821490963085e-06,0.5,1,3,0.0,1.5711413800219043e-06,0.00015928032072830324,0.0,0.0,24355.8176821793
1942.9654613455605,4.924561400110445e-06,3.453593508681139,1.5599018476870417e-07,0.25,2,3,4.763407737026446,1.0345388132680983e-07,1.4336737402714107e-05,3.924753330602083,3.4737793773912857,4643.062578676116
51163.16445253854,4.5035134122991666e-07,3.45087237529146,5.970281026981495e-09,0.125,3,3,4.7075126021460445,7.872395875535465e-09,1.274225935834421e-06,3.7160411916003655,3.4920237358393953,94441.37963263354
3523384.9402588387,4.048718893489529e-08,3.475513574183483,2.512453484356459e-10,0.0625,4,3,4.570630050608781,5.45705671543621e-10,1.1285284303714611e-07,3.8506078277448683,3.49710643865044,8217415.218489129
32621.251605902078,5.312017207296665e-06,0.0,7.348685601544366e-07,0.5,1,4,0.0,4.317034941080978e-07,1.3166480851092788e-05,0.0,0.0,84943.07988593556
10668.963331150857,2.2560509566860762e-07,4.557388261136978,1.5204480464735627e-08,0.25,2,4,5.594917804516674,1.5662896410582578e-08,6.165985343426613e-07,4.784617841548677,4.4163945219004255,9351.605788022036
1653722.993978011,1.069267123380681e-08,4.399105437771318,3.4390801256204474e-10,0.125,3,4,5.4663299812943595,5.389374012218556e-10,2.7767120558533256e-08,4.861089502329204,4.4728819771424515,1288832.9927211131
464890583.512827,4.980230229144694e-10,4.424266063048901,8.05713408398807e-12,0.0625,4,4,5.41561215420006,1.75724819291465e-11,1.2349275011311677e-09,4.938727834045938,4.49087932476763,436678551.91238946
22946223.430312607,8.616088512832804e-11,11.144725324345842,1.8153573035996977e-12,0.125,3,5,13.000961202429139,1.8765157117792003e-12,3.990160788923575e-10,12.985026922969285,10.4076113553025,42392364.781869575
197398303109105.1,1.5113466182596864e-13,14.196410220550687,2.5904500099131777e-15,0.125,3,10,16.151906226783403,5.4196670347431255e-16,1.3074068809013154e-13,16.904215099332944,14.266119012767755,122509235273935.1
1.527819192496108e+20,3.401516358785461e-12,12.698986370407498,4.3648880865968255e-14,0.125,3,15,14.79368288800859,5.8559975478180456e-15,4.967896641009374e-12,15.759668824604237,12.516833993533368,1.145888964866626e+20
1.2145706719605645e+27,2.144961704323753e-10,10.70610968616,8.777838148011868e-13,0.125,3,20,13.350399854841486,2.2086257374947452e-13,2.508449168718152e-10,14.013972088636365,10.630828415095523,1.4679647603317744e+27
\end{filecontents*}
\csvreader[head to column names,filter=\p>4]{./cartbesselgppw.csv}{}
{\p
& \num[round-precision=2,round-mode=figures, scientific-notation=true]{\dgerrornnn} 
& \num[round-precision=2,round-mode=figures, scientific-notation=true]{\errornnn} 
& \num[round-precision=3,round-mode=figures, scientific-notation=true]{\condnnn} 
& \num[round-precision=2,round-mode=figures, scientific-notation=true]{\dgerrorccrc} 
& \num[round-precision=2,round-mode=figures, scientific-notation=true]{\errorccrc} 
& \num[round-precision=3,round-mode=figures, scientific-notation=true]{\condccrc} 
\tabularnewline
}
\\\addlinespace[-\normalbaselineskip]\bottomrule
\end{tabular}
\caption{Errors committed by the quasi-Trefftz DG method for different combinations of the numerical flux parameters and volume penalisation coefficient.}
\label{tab:bessel}
\end{table}

In \Cref{fig:bessel} we present results for $p=2,3,4$, and compare the quasi-Trefftz space and the full polynomial space $\PP^p$.
Penalisation and flux parameters are set to zero.
We observe optimal convergence rates in terms of $h$.
Despite the fact that, comparing the condition numbers, the quasi-Trefftz space for larger $p$ performs worse than the full polynomial space, the performance of the two methods are very similar in terms of both convergence rate and error level, with a lower number of degrees of freedom required for the quasi-Trefftz method.

We consider very large polynomial degrees $p=5,10,15,20$ on a uniform mesh with $h=2^{-3}$ and $h_t\approx h$, for two choices of the penalisation and flux parameters: (i) all set to zero, and (ii) the choices in \eqref{eq:FluxChoice} and \eqref{eq:muChoice}.
The results are given in \Cref{tab:bessel}, where we show the error in the $\Tnorm{\cdot}\DG$ norm and in the $L^2(Q)$ norm, as well as the condition number of the system. 
Both cases show a similar behaviour: for large $p$ the condition number exceeds the inverse of machine precision and the error slightly deteriorates.
The choice $\alpha=\beta=\mu_1=0$ shows better $\Tnorm{\cdot}\DG$-error as several terms in the norm vanish, while in the $L^2$-norm choosing non-zero parameters gives a smaller error.

{
\subsection{Non-smooth solution}\label{sec:numl}
We consider the IBVP \eqref{eq:IBVP} with a hat function as initial condition:
\begin{align}\label{eq:numl}
    \bsigma_0(x)=v_0(x) = \max(0.25-|x|,0), \qquad
    G(x)=(1+x)^{-2},\qquad \rho=1,    \qquad\oon \Omega=(-0.5,0.5),
\end{align}
such that $u_0(x)=\int^x \bsigma_0(y)\di y\in H^2(\Omega)\setminus C^2(\Omega)$.
We choose a Cartesian-product mesh, such as the one on the left in Figure~\ref{fig:mesh}, with $h_t\approx h$.
We use homogeneous Neumann boundary conditions and stop the computations at $T=0.1$, before the wave interacts with the boundary.
We measure the error, as defined in \eqref{eq:fterror}, at $T=0.1$, against a reference solution computed on a very fine mesh with $h=2^{-12}$.
Due to the poor regularity of the initial condition, the order of convergence we can expect is limited to $\mathcal O(h)$ regardless of the polynomial degree used.
Therefore, we choose polynomial order $p=0$, so that the quasi-Trefftz space contains only piecewise-constant functions.
Results are shown in \Cref{tab:numl}, where we observe that the rate of convergence tends to 1.
This suggests that the convergence rates of Theorem~\ref{thm:Convergence} hold also for solutions $u\in H^{p+2}\Th$ (as opposed to $C^{p+2}\Th$, recall Remark~\ref{rem:CvsH}).

\begin{table}[!ht]\centering
\begin{tabular}{lll}
\toprule 
$h$ & error & rate \\
\midrule 
\begin{filecontents*}{./carth2.csv}
h,hnr,ndof,p,terror,trate,ttime
0.015625,6,896.0,0,0.019824699579321068,0,0.21463823318481445
0.0078125,7,3328.0,0,0.011991162259883278,0.7253275029701787,0.18764114379882812
0.00390625,8,13312.0,0,0.006802008989296486,0.8179386829543698,0.4282410144805908
0.001953125,9,53248.0,0,0.0036864992378977885,0.8837094539377222,1.3767409324645996
0.0009765625,10,210944.0,0,0.0018168080777810109,1.0208454331302583,8.111159324645996
\end{filecontents*}
\csvreader[head to column names]{./carth2.csv}{}
{
$2^{-\hnr}$
& \num[round-precision=2,round-mode=figures, scientific-notation=true]{\terror} 
& \num[round-precision=2,round-mode=figures, scientific-notation=true]{\trate} 
\tabularnewline
}
\\\addlinespace[-\normalbaselineskip]\bottomrule
\end{tabular}
\caption{{Errors committed by the quasi-Trefftz DG method for a non-regular solution with parameters as in 
\eqref{eq:numl}.}}
\label{tab:numl}
\end{table}

\begin{figure}[ht]
\centering
\includegraphics[width=\textwidth,clip, trim = 0 70 0 70]{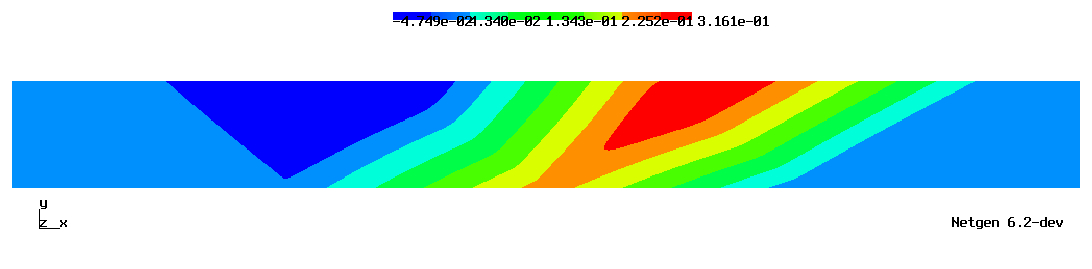}\\ \ \\
\includegraphics[width=\textwidth,clip, trim = 0 70 0 70]{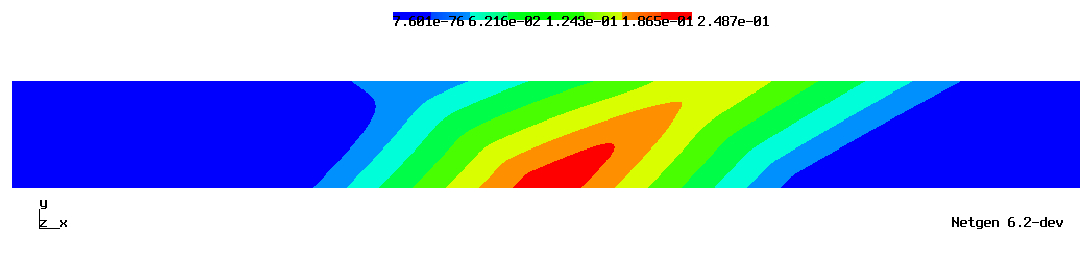}
\caption{{Solution of the first-order wave equation with initial conditions as in \eqref{eq:numl} on a mesh with $h=2^{-10}$, with $v$ shown on top and $\bsigma$ on the bottom.}}
\label{fig:numlsol}
\end{figure}
}

\phantomsection
\section*{Acknowledgments}
\addcontentsline{toc}{section}{Acknowledgments}
The authors would like to acknowledge the kind hospitality of the Erwin Schr\"odinger International Institute for Mathematics and Physics (ESI), where part of this research was developed under the frame of the Thematic Programme {\it Numerical Analysis of Complex PDE Models in the Sciences}.

L.-M. Imbert-G\'erard acknowledges support from the US National Science Foundation: this material is based upon work supported by the NSF under Grants No. DMS-1818747 and DMS-2105487.

A.\ Moiola acknowledges support from GNCS--INDAM, from PRIN project ``NA\_FROM-PDEs'' and from MIUR through the ``Dipartimenti di Eccellenza'' Programme (2018--2022)---Dept.\ of Mathematics, University of Pavia.

P. Stocker has been supported by the Austrian Science Fund (FWF) through the projects F~65 and W~1245.

\phantomsection
\addcontentsline{toc}{section}{References}
\bibliographystyle{siam}
\bibliography{references_andrea2020}

\end{document}